%% file: main.tex
\documentclass{article}

\usepackage{Settings}
\title{Product sets in sets of returns and positivity of symmetric ergodic averages}
\author{Vitaly Bergelson and Saúl Rodríguez-Martín}
\date{}

\begin{document}
\maketitle 

\begin{abstract}

Let $G$ be a countable group and let $(T_g)_{g\in G}$ be a measure-preserving action of $G$ on a probability space $(X,\mathcal{B},\mu)$.
We investigate when sets of returns, namely sets of the form $\{g\in G:\mu(Y\cap T_gY)>0\}$ where $Y\in\mathcal{B},\mu(Y)>0$, contain product sets $BB$ where $B\subseteq G$ is large in an appropriate sense.
For example, if $G$ is amenable, then a natural notion of largeness is that of having positive upper density along a prescribed F{\o}lner sequence, while for the free group $\mathbb{F}_2$, one may consider positive upper density along the sequence of word-metric balls. 
We relate the problem of finding such large sets $B$ to asymptotic positivity of averages of symmetric 
correlation functions
$\phi(g)=\mu(T_g^{-1}Y\cap T_gY)$. 
Using this connection, we show that, for broad classes of amenable 
groups, including finitely generated nilpotent groups and certain 
solvable non-nilpotent groups, for every set $A\subseteq G$ of positive 
left upper Banach density and any
F{\o}lner sequence $(F_N)$ in $G$, there is a set $B\subseteq G$ with positive upper density along $(F_N)$ such that $BB\subseteq AA^{-1}$. We obtain some further results for non-amenable groups. We also obtain conditions under which a set of measurable returns for a \(G\times G\) action contains \(B\times B\) for some large \(B\subseteq G\).
Here is a sample of some results obtained in this paper.

\begin{itemize}
\item Sets of returns for $\mathbb{F}_2\times\mathbb{F}_2$-actions contain Cartesian squares $B\times B$, with $B$ of positive upper density along the sequence of word-metric balls in $\mathbb{F}_2$.
    \item For amenable groups $G$, sets of measurable returns for $G\times G$-actions
    contain Cartesian products $B\times B$ with $B$ large. We obtain quantitative bounds on the upper density of the set $B$.
    \item For finitely generated nilpotent $G$ and polynomial maps $p,q:G\to G$ such that $p(1_G)=q(1_G)=1_G$, sets of returns for $G$- and $G\times G$-actions contain
$p(B_1)q(B_1)$ and $p(B_2)\times q(B_2)$, respectively, for some sets $B_i$ of positive upper density.

\item Sets of returns for $\mathbb{Z}^2$-actions contain Cartesian squares $B\times B$, where $B$ has positive asymptotic density. 

\item For every $A\subseteq\mathbb{N}$ of positive upper density, every set $E$ of measurable recurrence (such as the squares or the shifted primes $\{p-1:p\text{ is prime}\}$), and every infinite $C\subseteq\mathbb{N}$, there exist $x\in E$ and $B\subseteq C$ such that $B$ has positive upper density relative to $C$ and $x(B+B)\subseteq A-A$.

\item If $C=\{c_1<c_2<\dots\}\subseteq\mathbb{N}$ is a set of averaging recurrence, meaning that for any measure-preserving system $(X,\mathcal{B},\mu,T)$ and $A\in\mathcal{B}$ with $\mu(A)>0$, $\overlim_{N}\frac{1}{N}\sum_{n=1}^N\mu(A\cap T^{-c_n}A)>0$, then every set of returns for a $\mathbb{Z}^2$-action contains a Cartesian square $B\times B$ with $B\subseteq C$ of positive upper density relative to $C$.
\end{itemize}

\end{abstract}

\tableofcontents

\input{S1-Introduction2}

\input{S2-Preliminaries}

\input{S3-ErgThToCombi}
\input{S4-NilpotentGroups2}
\input{S5-Solvable}

\input{S6-Counterexamples}

\input{S7-Questions}

\input{S8-Appendix}

\subsubsection*{AI statement}

The mathematical ideas outside the appendix are the authors' own. As for the appendix, AI involvement was twofold. 
First, Sohail Farhangi provided us with a GPT-generated proof
of \Cref{ScalarConvergenceSquareAverages}, thereby supplying an answer to Question 8.2 in the first version of this paper, \cite{BR26}. We thank Sohail Farhangi for his involvement.
We also used AI to obtain the main ideas for the proofs of \Cref{NormConvergenceSquareAverages,HigherPowerAveragesCounterexample}. The authors scrutinized and simplified the proofs and take responsibility for their correctness.

\subsubsection*{Acknowledgement}
Shortly before uploading this version of the article to arXiv, we learned that Xiaofang Luo had independently established \Cref{NormConvergenceSquareAverages} and obtained a partial answer to \Cref{BigQuest} in the recent preprint \cite{Luo26}.

\bibliographystyle{myalpha}
\bibliography{bibliography2}
\end{document}

%% file: S1-Introduction2.tex
\section{Introduction}
\label{SecIntro}
The goal of this paper is to amplify and generalize, in more than one direction, the results about properties of sets of differences of ``large'' sets in $\mathbb{Z}$ and $\mathbb{Z}^2$ which were obtained in \cite{Be85,BeRu09}. For a set $A\subseteq\mathbb{Z}$, the \textit{upper density} of $A$ and the \textit{upper Banach density} of $A$ are defined respectively by
\begin{align}
\label{DefUDinZ}
\overline{d}(A)&=\overlim\limits_{N\to\infty}\frac{|A\cap\{-N,\dots,N\}|}{2N+1}\\
\label{DefUBDinZ}d^*(A)&=\sup_{(I_N)_{N\in\mathbb{N}}}\overlim\limits_{N\to\infty}\frac{|A\cap I_N|}{|I_N|},
\end{align}
where $(I_N)$ ranges over all sequences of intervals $I_N=[a_N,b_N]\cap\mathbb{Z}$ such that $\lim_{N\to\infty}(b_N-a_N)=\infty$. If $A\subseteq\mathbb{Z}^2$, then the upper Banach density of $A$ is given by
\begin{equation}
\label{DefUBDInZn}
d^*(A)=\sup_{(R_N)}\overlim\limits_{N\to\infty}\frac{|A\cap R_N|}{|R_N|},
\end{equation}
where $(R_N)$ ranges over all sequences of rectangles $R_N=([a_N,b_N]\times[c_N,d_N])\cap\mathbb{Z}^2$ such that $\lim_{N\to\infty}|b_N-a_N|=\lim_{N\to\infty}|d_N-c_N|=\infty$. 

Throughout the paper, when $(G,+)$ is an abelian group and $A,B\subseteq G$, we use the notation $
A+B=\{a+b;a\in A,b\in B\},
A-B=\{a-b;a\in A,b\in B\}$ (when using multiplicative notation $(G,\cdot)$, we write $AB=\{ab;a\in A,b\in B\}$ and $AB^{-1}=\{ab^{-1};a\in A,b\in B\}$). Here are the formulations of the results from \cite{Be85} which served as an impetus for this paper:

\begin{theorem}[{\cite[Corollary 3.1.1.]{Be85}}]
\label{BxBinA-AinZ}
Let  $A\subseteq\mathbb{Z}^2$ be a set with $d^*(A)>0$. Then there exists $B\subseteq\mathbb{Z}$ such that $\overline{d}(B)>0$ and $B\times B\subseteq A-A$.
\end{theorem}

\begin{theorem}[{\cite[Corollary 3.1.2.]{Be85}}]
\label{B+BinA-AinZ}
Let  $A\subseteq\mathbb{Z}$ be a set with $d^*(A)>0$. Then there exists $B\subseteq\mathbb{Z}$ such that $\overline{d}(B)>0$ and $B+B\subseteq A-A$.
\end{theorem}

\Cref{BxBinA-AinZ,B+BinA-AinZ}, which establish non trivial (and indeed, somewhat unexpected) properties of sets of differences of sets of positive upper density in $\mathbb{Z}$ and $\mathbb{Z}\times\mathbb{Z}$, bring to light some natural questions. 
For example, can the sets $B$ which appear in \Cref{BxBinA-AinZ,B+BinA-AinZ}, be guaranteed to have positive upper density with respect to any given in advance F{\o}lner sequence in $\mathbb{Z}$ (see \Cref{DefAmenableGroup})? What is a lower bound for $\overline{d}(B)$ in terms of $d^*(A)$? Can $B$ be chosen to have positive relative density in classical arithmetical sets, such as, for example, shifted primes or $\{\left[p_n^c\right];n\in\mathbb{N}\}$, where $p_n$ is the $n$-th prime and $c>1$ is not an integer? We have some interesting results pertaning to the above questions. This being said, a bigger goal of this paper is to extend \Cref{BxBinA-AinZ,B+BinA-AinZ} to more general groups. 
While these theorems can be rather routinely extended to the setup of general abelian groups (see \Cref{B+BinA-AAbelian,BxBinA-AAbelian}), the noncommutative case brings to life delicate issues.
In particular, whereas in the abelian setting, \Cref{B+BinA-AinZ} can be easily derived from \Cref{BxBinA-AinZ}, this is no longer true in general (see the discussion after \Cref{B+BinA-AAbelian}).

\msubsection{Cartesian products in sets of returns for $G\times G$-actions}

\Cref{BxBinA-AinZ}, which is of \textit{combinatorial} nature, is a consequence of a general, \textit{ergodic-theoretic} result about the structure of sets of returns for measure preserving actions of $\mathbb{Z}\times\mathbb{Z}$. This ergodic result is actually a special case of a general fact, \Cref{GeneralBxBinReturns} (proved in \Cref{SecErgToComb}), pertaining to measure preserving actions of cartesian squares of arbitrary groups. Before formulating it, we need to introduce a definition.

\begin{definition}
\label{Defupperdensity}
Let $G$ be a nonempty set. Let $E\subseteq G$ and $F=(F_N)$ a sequence of nonempty, finite subsets of $G$. We define the \textit{upper and lower densities} of $E$ along $F$ as 
\begin{align}
\label{DefUpperDensity}
\overline{d}_F(E)=\overline{\lim}_{N\to\infty}\frac{|E\cap F_N|}{|F_N|}.\\\notag
\underline{d}_F(E)=\underline{\lim}_{N\to\infty}\frac{|E\cap F_N|}{|F_N|}.
\end{align}
If $\overline{d}_F(E)=\underline{d}_F(E)$, we denote their common value by $d_F(E)$ and call it the \textit{density} of $E$ along $F$. 
\end{definition}

\begin{remark}
\label{ExamplesFiniteSequences}
While the more familiar context for \Cref{Defupperdensity} is that of a countable amenable group $G$ with $(F_N)$ being a Følner sequence (see \Cref{DefAmenableGroup}), the generality of \Cref{Defupperdensity} allows for novel interesting applications. Some of the results in our paper, such as \Cref{CartProdInRetsFree} and \Cref{PrimesMinus1ArePIR}, will utilize the following special cases of \Cref{Defupperdensity}.
\begin{enumerate}[label=\alph*)]
    \item $G=\mathbb{Z}^3$, $F_N=\{p(1),\dots,p(N)\}$, where $p:\mathbb{Z}\to\mathbb{Z}^3$ is a non constant polynomial with $p(0)=0$.
    \item $G=\mathbb{N}$, $F_N=\{[p_1^c],\dots,[p_N^c]\}$ where $p_N$ is the $N^{\textup{th}}$ prime, $c>1$ and $[x]=\lfloor x+1/2\rfloor$ is the integer closest to $x$. \item\label{ExamplesFiniteSequences3} $\mathbf{F}_2$ is the group freely generated by $a,b$, and $B_N\subseteq\mathbf{F}_2$, $N\in\mathbb{N}$, are the closed balls of radius $N$, i.e. the set of reduced words of length $\leq N$ in the alphabet $\{a,a^{-1},b,b^{-1}\}$.
\end{enumerate}
\end{remark}

\begin{theorem}
\label{GeneralBxBinReturns}
Let $G$ be a countably infinite group and $\varepsilon,\delta>0$. Suppose a sequence $F=(F_N)$ of finite subsets of $G$ has the property that, for all m.p.s.\footnote{A measure preserving system (m.p.s.) is a quadruple $(X,\mathcal{B},\mu,(T_g)_{g\in G})$, where $(X,\mathcal{B},\mu)$ is a probability space and $(T_g)_{g\in G}$ is a left action of a group $(G,\cdot)$ on $X$ by measure preserving maps $T_g:X\to X$, so that $T_{gh}=T_g\circ T_h$.} $(X,\mathcal{B},\mu,(T_g)_{g\in G})$ and all $Y\in\mathcal{B}$ satisfying $\mu(Y)\geq\varepsilon$, we have
\begin{equation*}
\overlim_{N\to\infty}\frac{1}{|F_N|}\sum_{g\in F_N}\mu(Y\cap T_{g}Y)
\geq\delta.
\end{equation*}
Then for all m.p.s. $(X,\mathcal{B},\mu,(T_{(g,h)})_{(g,h)\in G\times G})$, and all $Y\in\mathcal{B}$ with $\mu(Y)\geq\varepsilon$, there exists $B\subseteq G$ such that $\overline{d}_F(B)\geq\delta$ and 
\begin{equation}
\label{Form6hehehe}
B\times B\subseteq\left\{(g,h)\in G\times G;\mu\left(Y\cap T_{(g,h)}Y\right)>0\right\}. 
\end{equation}
\end{theorem}

\Cref{BxBinA-AinZ} follows from \Cref{GeneralBxBinReturns} and the following result, which we later generalize in \Cref{Returns=DifferencesIntro}:
for any $A\subseteq \mathbb{Z}^2$, there exists a m.p.s. $(X,\mathcal{B},\mu,(T_{(x,y)})_{(x,y)\in\mathbb{Z}^2})$ and $Y\in\mathcal{B}$ such that $\mu(Y)\geq d^*(A)$ and 
\begin{equation*}
\left\{(x,y)\in\mathbf{Z}^2;\mu\left(Y\cap T_{(x,y)}Y\right)>0\right\}\subseteq A-A.
\end{equation*}
Indeed, applying \Cref{GeneralBxBinReturns} with $G=\mathbb{Z}$ and $F_N=\{-N,\dots,N\}$, and using the classical corollary of von Neumann's ergodic theorem
, which states that for every m.p.s. $(X,\mathcal{B},\mu,(T_n)_{n\in\mathbb{Z}})$ and $Y\in\mathcal{B}$,
\begin{equation}
\label{BigCesaroAvgsInZFolner}
\lim_{N\to\infty}\frac{1}{2N+1}\sum_{n={-N}}^N\mu\left(Y\cap T^nY\right)\geq\mu(Y)^2,
\end{equation}
yields \Cref{BxBinA-AinZ}.
Below we use \Cref{GeneralBxBinReturns} to extend \cite[Corollary 3.1.1]{Be85} to all F{\o}lner sequences in countable\footnote{Throughout the paper, by `countable' we mean `in bijective correspondence with $\mathbb{N}$'. However, almost all results in the article, except some implications in \Cref{Returns=Differences}, also hold for finite groups $G$, where F{\o}lner sequences are sequences of sets $(F_N)$ such that $F_N=G$ for big enough $N$.} amenable groups (see \Cref{BxBinAA^-1}).

\begin{remark}
\label{BeRuCounterBxBxB}
In view of \Cref{Form6hehehe} from \Cref{GeneralBxBinReturns} it is natural to inquire under which conditions on $G$ and $(F_N)$ we have that, given a m.p.s. $(X,\mathcal{B},\mu,(T_{(g_1,g_2,g_3)})_{(g_1,g_2,g_3)\in G\times G\times G})$, and $Y\in\mathcal{B}$ with $\mu(Y)>0$, there exists $B\subseteq G$ such that $\overline{d}_F(B)>0$ and 
\begin{equation}
\label{BxBxBEq}
B\times B\times B\subseteq
\left\{(g_1,g_2,g_3)\in G\times G\times G;\mu\left(Y\cap T_{(g_1,g_2,g_3)}Y\right)>0\right\}.
\end{equation}
Somewhat surprisingly, even for $G=\mathbb{Z}$ and $F_N=\{1,\dots,N\}$, one can find a m.p.s. $(X,\mathcal{B},\mu,(T_g)_{g\in\mathbb{Z}^3})$ and $A\in\mathcal{B}$ such that $\mu(A)>0$ but no set $B\subseteq\mathbb{Z}$ with $\overline{d}_F(B)>0$ satisfies \Cref{BxBxBEq}; this follows from \cite[Theorem 1.5]{BeRu09} and from the correspondence between sets of differences and sets of returns (precisely stated in \Cref{Returns=Differences}).
\end{remark}

We have, for any $c>1$ non integer, an ergodic theorem along the sequence $([p_n^c])_{n\in\mathbb{N}}$, where $p_n$ is the $n$-th prime; a more general version is given in \cite{BKMST}. 

\begin{theorem}[Cf. {\cite[Theorem 3.1]{BKMST}}]
For all m.p.s. $(X,\mathcal{B},\mu,(T_n)_{n\in\mathbb{Z}})$, $c>1$ and $Y\in\mathcal{B}$, we have
\begin{equation*}
\lim_{N\to\infty}\frac{1}{N}\sum_{n=1}^N\mu\left(Y\cap T^{\left[p_n^c\right]}Y\right)\geq\mu(Y)^2.
\end{equation*}
\end{theorem}

Using \Cref{GeneralBxBinReturns}, we will deduce in \Cref{SecErgToComb} the following result:
\begin{cor} 
\label{PrimesMinus1ArePIR}
Let $c>1,c\not\in\mathbb{N}$ and $F_N=\{\left[p_1^c\right],\left[p_2^c\right],\cdots,\left[p_N^c\right]\}$. For any $A\subseteq\mathbb{Z}$ with $d^*(A)>0$, there exists $B\subseteq\mathbb{Z}$ such that $\overline{d}_F(B)\geq d^*(A)^2$ and $B\times B\subseteq A-A$.
\end{cor}

We now present an example in which \Cref{BxBinAA^-1} applies to a non-amenable group. Let $\mathbf{F}_2$ be the free group on two generators and let $B_N\subseteq\mathbf{F}_2$, $N\in\mathbb{N}$, be the closed balls of radius $N$, as in \Cref{ExamplesFiniteSequences}\ref{ExamplesFiniteSequences3}. We deduce in \Cref{SecErgToComb} the following result from an ergodic theorem of Guivarc'h:

\begin{restatable}{theorem}{RestatableC}
\label{BigAvgsBallFreeGroup}
For all m.p.s. $(X,\mathcal{B},\mu,(T_g)_{g\in\mathbf{F}_2})$ and all $Y\in\mathcal{B}$, we have
\begin{equation*}
\lim_{N\to\infty}\frac{1}{|B_{2N}|}\sum_{g\in B_{2N}}\mu(Y\cap T_gY)\geq\mu(Y)^2.
\end{equation*}
\end{restatable}
Using \Cref{GeneralBxBinReturns} and \Cref{BigAvgsBallFreeGroup}, we obtain:

\begin{cor}
\label{CartProdInRetsFree}
 For any m.p.s. $(X,\mathcal{B},\mu,(T_{(g,h)})_{(g,h)\in \mathbf{F}_2\times\mathbf{F}_2})$ and $Y\in\mathcal{B}$, there is $C\subseteq\mathbf{F}_2$ such that $\overline{d}_{(B_N)_{N\in\mathbb{N}}}(C)\geq\mu(Y)^2$ and 
 \begin{equation*}
 C\times C\subseteq\{(g,h)\in\mathbf{F}_2\times\mathbf{F}_2;\mu(Y\cap T_{(g,h)}Y)>0\}.
 \end{equation*}
\end{cor}

In \Cref{SecGeneralDensities} we prove a generalization of \Cref{CartProdInRetsFree} which holds for any countable group, see \Cref{ThmConvsAreNiceAndHaveBxBinAA-1} and \Cref{RemarkConvsAndBallsFreeProds}.

In order to state other applications of \Cref{GeneralBxBinReturns}, we need to introduce some additional definitions.
Since the groups to which the following definitions apply are not necessarily abelian, we use multiplicative notation.

\begin{definition}
\label{DefAmenableGroup}
Let $(G,\cdot)$ be a countable group. A \textit{left F{\o}lner sequence} $F=(F_N)_{N\in\mathbb{N}}$ in $G$ is a sequence of finite sets $F_N\subseteq G$ such that, for all $g\in G$, $\lim_{N\to\infty}\frac{|gF_N\Delta F_N|}{|F_N|}=0$. We say $G$ is \textit{left amenable} if it has a left F{\o}lner sequence. 
Similarly, $(F_N)$ is \textit{right F{\o}lner} if $\lim_{N\to\infty}\frac{|(F_Ng)\Delta F_N|}{|F_N|}=0$; a group $G$ is left amenable iff it is right amenable, as for any left F{\o}lner sequence $(F_N)$, $\left(F_N^{-1}\right)$ is right F{\o}lner. 
\end{definition}

The notion of amenability admits several equivalent formulations. In particular, a countable group is amenable if and only if there are finitely additive probability measures on $(G,\mathcal{P}(G))$ which are invariant under the left translation maps $L_g:G\to G;x\mapsto gx$ for all $g\in G$.

\begin{definition}
\label{DefupperBanachdensity}
Let $G$ be a countable amenable group. The \textit{left upper Banach density} of a subset $E$ of $G$, $d^*_l(E)$, is defined by
\begin{equation*}
d^*_l(E)=\sup\left\{\overline{d}_F(E);F\text{ left F{\o}lner sequence in }G\right\}.
\end{equation*}
One can similarly define the right upper Banach density $d^*_r$ by using right F{\o}lner sequences. When $G$ is an abelian (or finite) group, then clearly $d^*_l=d^*_r$ and we just write $d^*$, or $d^{*}_{G}$ if the group $G$ is not clear from context.
\end{definition}

For $\mathbb{Z}^n$, the definitions of $d^*$ from \Cref{DefUBDInZn} and \Cref{DefupperBanachdensity} coincide; a general version of this fact for countable amenable groups is proved in \cite[Corollary 3.6]{BG}.

\begin{remark}
\label{Rmk113}
The densities $d^*_l$ and $d^*_r$ need not always coincide. They do coincide for groups with the property that all left invariant means are right invariant and viceversa. 
One can show that this class of groups coincides with the so called FC-groups, that is, groups in which every conjugacy class is finite (see \cite[(4.23iii)]{Pat}).
If a countable amenable group $G$ is not an FC-group, then one can construct a set $A\subseteq G$ satisfying $d^*_l(A)=1$ but $d^*_r(A)\leq\frac{1}{2}$ (see \Cref{UpperAndLowerUBDDiffer} for an elementary proof). The constant $\frac{1}{2}$ is optimal without further restrictions on $G$, as if we let $G$ be the dihedral group $D_\infty=\langle r,s|s^2=1,sr=r^{-1}s\rangle$, then one can check that all $A\subseteq G$ such that $d^*_l(A)=1$ must satisfy $d^*_r(A)\geq\frac{1}{2}$. A way to construct sets $A$ such that $d^*_l(A)=1$ but $d^*_r(A)<1$ is by letting $A=\bigcup_NF_N$, where $(F_N)$ is an adequate left F{\o}lner sequence that is not right F{\o}lner. We use this technique in the proof of \Cref{CounterBBinA-1A} to construct a subset $A$ of the Heisenberg group $G=\mathrm{UT}_3(\mathbb{Z})$, i.e. the group of unipotent upper triangular integer matrices, such that $d^*_l(A)=1$, $d_r^*(A)=0$. Another example of such a subset of $\mathrm{UT}_3(\mathbb{Z})$ appears in \cite[Theorem 3.4]{BHM}.
\end{remark}

We are going now to state an ergodic-theoretic version of \Cref{BxBinA-AinZ} in the setting of countable amenable groups. 
But first, we need to cite the following variant of von Neumann's mean ergodic theorem in the amenable setting (for a proof of this result see the discussion at the beginning of \Cref{SecPrelims}). 
\begin{prop}
\label{VonNeumannmuASquaredIntro}
Let $(F_N)$ be a left or a right F{\o}lner sequence in a countable amenable group $G$. Then, for any m.p.s. 
$(X,\mathcal{B},\mu,(T_g)_{g\in G})$ and $Y\in\mathcal{B}$, we have
\begin{equation*}
\lim_{N\to\infty}\frac{1}{|F_N|}\sum_{g\in F_N}\mu(Y\cap T_gY)\geq\mu(Y)^2.
\end{equation*}
\end{prop}

The following corollary is a consequence of \Cref{VonNeumannmuASquaredIntro} and \Cref{GeneralBxBinReturns} with $\delta=\varepsilon^2$:

\begin{restatable}[Cartesian products in sets of returns]{cor}{CartProdInRets} 
\label{CartProdInRets}
Let $G$ be a countable amenable group with a left or right F{\o}lner sequence $F=(F_N)_{N\in\mathbb{N}}$, a m.p.s. $(X,\mathcal{B},\mu,(T_{(g,h)})_{(g,h)\in G\times G})$ and $Y\in\mathcal{B}$. Then there is $B\subseteq G$ such that $\overline{d}_F(B)\geq\mu(Y)^2$ and $B\times B\subseteq\{(g,h)\in G\times G;\mu(Y\cap T_{(g,h)}Y)>0\}$.
\end{restatable}

The purely combinatorial version of \Cref{CartProdInRets}, which generalizes \Cref{BxBinA-AinZ}, reads as follows.

\begin{theorem}
\label{BxBinAA^-1}
Let $G$ be a countable amenable group, let $A\subseteq G\times G$ satisfy $d^*_l(A)>0$. Then for any left or right F{\o}lner sequence $F$ in $G$ there exists $B$ such that $\overline{d}_F(B)\geq d^*_l(A)^2$ and $B\times B\subseteq AA^{-1}$.
\end{theorem}

The equivalence between \Cref{BxBinAA^-1} and \Cref{CartProdInRets} is immediate from the following result:

\begin{prop}[Cf. \Cref{Returns=Differences}]
\label{Returns=DifferencesIntro}
Let $G$ be a countable amenable group. 
\begin{enumerate}
    \item For all $A\subseteq G$, there is a m.p.s. $(X,\mathcal{B},\mu,(T_g)_{g\in G})$ and $Y\in\mathcal{B}$ with $\mu(Y)=d^*_l(A)$ and 
\begin{equation*}
\{g\in G;\mu(Y\cap T_gY)>0\}\subseteq AA^{-1}
\end{equation*}
    \item For all m.p.s. $(X,\mathcal{B},\mu,(T_g)_{g\in G})$ and $Y\in\mathcal{B}$, there is a set $A\subseteq G$ with $d^*_l(A)\geq\mu(Y)$ and 
\begin{equation*}
AA^{-1}\subseteq\{g\in G;\mu(Y\cap T_gY)>0\}
\end{equation*}
\end{enumerate}
\end{prop}

In \Cref{SecErgToComb} we will show, using a result from the recent preprint \cite{GXZ26}, that in the case $G=\mathbb{Z}$ and $F_N=\{1,\dots,N\}$, \Cref{BxBinAA^-1} has a significantly sharper version. Namely, the set $B$ can be chosen to have asymptotic density:

\begin{restatable}{theorem}{CartesianSquareAsymptoticDensityInN} 
\label{CartesianSquareAsymptoticDensityInN}
Let $A\subseteq\mathbb{Z}^2$ satisfy $d^*(A)>0$. Then there exists $B\subseteq\mathbb{N}$ such that $B$ has asymptotic density $d(B)\geq d^*(A)^2$ and $B\times B\subseteq A-A$.
\end{restatable}

In particular, \Cref{CartesianSquareAsymptoticDensityInN} implies the result of Bergelson and Ruzsa \cite[Theorem 1.3]{BeRu09} which states that, if $A\subseteq\mathbb{N}$ and $\overline{d}(A)>0$, then $A-A$ contains $B+B$ for some set $B\subseteq\mathbb{N}$ with positive asymptotic density.

It is worth noting that there is a natural analogue of \Cref{BxBinAA^-1} for right upper Banach density, which follows from the fact that $d^*_l(A)=d^*_r
\left(A^{-1}\right)$ for all 
$g\in G$ (see \Cref{FolnerProps}):
\begin{theorem}
\label{BxBinAA^-1right}
Let $G$ be a countable amenable group, let $A\subseteq G\times G$ satisfy $d^*_r(A)>0$. Then for any left or right F{\o}lner sequence $F$ in $G$ there exists $B$ such that $\overline{d}_F(B)\geq d^*_r(A)^2$ and $B\times B\subseteq A^{-1}A$.
\end{theorem}

In particular, since abelian groups are amenable, we have the following enhanced form of \Cref{BxBinA-AinZ}:

\begin{theorem}
\label{BxBinA-AAbelian}
Let $(G,+)$ be a countable abelian group and let $A\subseteq G\times G$. Then for any F{\o}lner sequence $F$ in $G$ there is some $B\subseteq G$ such that $\overline{d}_F(B)\geq d^*(A)^2$ and $B\times B\subseteq A-A$. 
\end{theorem}

\begin{remark}
\label{Rmk120}
One may ask whether \Cref{BxBinAA^-1} remains valid with the inclusion
$B\times B\subseteq A^{-1}A$
in place of \(B\times B\subseteq AA^{-1}\), while retaining the left
upper Banach density \(d_l^*\). In general, it does not: a counterexample
in the Heisenberg group \(\mathrm{UT}_3(\mathbb{Z})\) is given in
\Cref{CounterBxBinA-1A}.

\Cref{CounterBBinA-1A} provides an analogous counterexample for product sets, with \(B_1B_2\) in place of the Cartesian square \(B\times B\): there is
a set \(A\subseteq\mathrm{UT}_3(\mathbb{Z})\) satisfying \(d_l^*(A)=1\)
for which there are no sets \(B_1,B_2\subseteq\mathrm{UT}_3(\mathbb{Z})\)
such that \(d_l^*(B_i)>0\) for \(i=1,2\) and
\[
B_1B_2\subseteq A^{-1}A.
\]
A detailed analysis of the distinct behaviour of the sets $A^{-1}A$ and $AA^{-1}$, for sets $A$ with $d^*_l(A)>0$, can be found in \cite{BHM}.
\end{remark}

Here is a two-sided version of \Cref{BxBinAA^-1}, which is proved in \Cref{SecErgToComb}.

\begin{restatable}{theorem}{BxBinAAminusonetwosided}
\label{BxBinAAminus1twosided}
Let $G$ be a countable amenable group, let $A_l,A_r\subseteq G\times G$. Then for any left or right F{\o}lner sequence $F$ in $G$, there exists $B\subseteq G$ such that $\overline{d}_F(B)\geq d^*_l(A_l)^2d^*_r(A_r)^2$ and $B\times B\subseteq A_lA_l^{-1}\cap A_r^{-1}A_r$.
\end{restatable}

Until now, we have restricted the discussion to densities $\overline{d}_F$ along sequences $F=(F_N)$ of finite sets in a countable group $G$. Introducing more general notions of density allows us to obtain amplifications of some of our results. We provide now an example of this phenomenon, which involves the notion of upper logarithmic density.

\begin{definition}
The upper logarithmic density of a set $A$ of natural numbers is defined as 
\begin{equation*}
\overline{d}_{\log}(A)
=
\overlim_{N\to\infty}\frac{1}{\log(N)}\sum_{n=1}^N\frac{1_A(n)}{n}
=
\overlim_{N\to\infty}\frac{1}{s_N}\sum_{n=1}^N\frac{1_A(n)}{n},
\end{equation*}
where $s_N=\sum_{n=1}^N\frac{1}{n}$, so that $\lim_{N\to\infty}\frac{\log(N)}{s_N}=1$, and $1_A:\mathbb{Z}\to\{0,1\}$ is the characteristic function of $A$.
\end{definition}

\begin{restatable}{prop}{RestatableD}
\label{B+BinA-AinZLogDensity}
Let $A\subseteq\mathbb{Z}^2$ be a set with $d^*(A)>0$. Then there exists $B\subseteq\mathbb{Z}$ such that $\overline{d}_{\log}(B)\geq d^*(A)^2$ and $B\times B\subseteq A-A$.
\end{restatable}
\Cref{B+BinA-AinZLogDensity} is stronger than \Cref{BxBinA-AinZ} because every set with positive upper logarithmic density has positive upper density, but not the other way around. For example, if $A=\mathbb{N}\cap\bigcup_{N\in\mathbb{N}}\left[2^{2^N},N\cdot 2^{2^N}\right]$, then $\overline{d}\left(A\right)=1$ but $\overline{d}_{\log}(A)=0$.

The proof of \Cref{B+BinA-AinZLogDensity} is given in \Cref{SecGeneralDensities}, where we actually show that 
many of our results can be amplified by utilizing generalized densities (see \Cref{5terdf9ordgfl}) along sequences of probability measures $(\mu_N)_{N\in\mathbb{N}}$ on a given countable group $G$.

\msubsection{Product sets in sets of returns for $G$-actions}

We now move our discussion to variants of \Cref{B+BinA-AinZ} for more general groups. The following result is proved in \Cref{SecErgToComb}:

\begin{theorem}[Products in sets of returns in $G$]
\label{ProdsInRetsMeasurableIntro}
Let $F=(F_N)$ be a sequence of finite subsets of a group $G$, and let $(X,\mathcal{B},\mu,(T_g)_{g\in G})$ be a m.p.s. and $Y\in\mathcal{B}$. Suppose that
\begin{equation*}
0<\delta:=\overlim\limits_{N\to\infty}\frac{1}{|F_N|}\sum_{g\in F_N}\mu\left(T_{g}^{-1}Y\cap T_{g}Y\right)=\overlim\limits_{N\to\infty}\frac{1}{|F_N|}\sum_{g\in F_N}\mu\left(Y\cap T_{g}^2Y\right).
\end{equation*}
Then there exists $B\subseteq G$ such that $\overline{d}_F(B)\geq\delta$ and 
\begin{equation*}
BB\subseteq\left\{g\in G;\mu\left(Y\cap T_{g}Y\right)>0\right\}. 
\end{equation*}
\end{theorem}

We now list some applications of this theorem.

Firstly, suppose $G$ is an abelian group. In this case, if $(T_g)_{g\in G}$ is a measure preserving action of $G$, then so is $\left(T_g^2\right)_{g\in G}$. So by
\Cref{VonNeumannmuASquaredIntro}, for any m.p.s. $(X,\mathcal{B},\mu,(T_g)_{g\in G})$ and any $Y\in\mathcal{B}$ we have
\begin{equation}
\label{BigIntsAbelianSquares}
\lim_{N\in\mathbb{N}}\frac{1}{|F_N|}\sum_{g\in F_N}\mu\left(Y\cap T_g^2Y\right)\geq\mu(Y)^2.
\end{equation}

Therefore, \Cref{ProdsInRetsMeasurableIntro} and \Cref{Returns=DifferencesIntro} imply the following result.
\begin{restatable}{theorem}{RestatableA}
\label{B+BinA-AAbelian}
Let $(G,+)$ be a countable abelian group and let $A\subseteq G$. Then for any F{\o}lner sequence $F=(F_N)_N$ in $G$ there is some $B\subseteq G$ such that $\overline{d}_F(B)\geq d^*(A)^2$ and $B+B\subseteq A-A$. 
\end{restatable}

\begin{remark} 
We do not know whether in \Cref{BxBinAA^-1} the value $d^*(A)^2$ is optimal. For some comments and a question about this topic see \Cref{OptimalBForB+BinA-A?}.
\end{remark}

In the abelian setting, one may also derive \Cref{B+BinA-AAbelian} directly from its cartesian product counterpart, \Cref{BxBinA-AAbelian}.
Indeed, given a set $A\subseteq G$, one can check that the set $A_2:=\{(x,y)\in G\times G;x+y\in A\}$ satisfies $d^{*}_{G\times G}(A_2)=d^{*}_G(A)$. Therefore, by \Cref{BxBinA-AAbelian}, there exists $B\subseteq G$ such that $d^*(B)\geq d^{*}_G(A)^2$ and 
\begin{align*}
B\times B\subseteq A_2-A_2&=\{(x_1,y_1)-(x_2,y_2);x_i,y_i\in G,x_i+y_i\in A\}\\
&\subseteq\{(x,y)\in G\times G;x+y\in A-A\},
\end{align*}
so $B+B\subseteq A-A$.  
However, for general (noncommutative) amenable groups, an analogous reasoning does not work. Indeed, in \Cref{BxBinA-AAbelian} we obtain a set $B\subseteq G$ such that $B\times B\subseteq AA^{-1}$ and $\overline{d}_F(B)\geq d^*_l(A)^2$. However, as the following (admittedly degenerated) example shows, there is an amenable group $G$ and a set $A\subseteq G$ such that there exists no $B\subseteq G$ satisfying $BB\subseteq AA^{-1}$ and with upper density at least $d^*(A)^2$.
Let $Q_8=\{\pm1,\pm i,\pm j,\pm k\}$ be the quaternion group, and let $
A=\{1,i,j,k\}\subseteq Q_8$.
Then $d^{*}_{Q_8}(A)=\frac{|A|}{|Q_8|}=\frac{1}{2}$, but there is no $B\subseteq Q_8$ such that $d^{*}_{Q_8}(B)>\frac{1}{8}$ and $BB\subseteq AA^{-1}=\{1,\pm i,\pm j,\pm k\}$. To obtain an infinite example, consider the subset $A\times\mathbb{Z}$ of $Q_8\times\mathbb{Z}$.

The set $A=\{1,i,j,k\}\subseteq Q_8$ also provides a non abelian example where \Cref{BigIntsAbelianSquares} is not satisfied: if we equip $Q_8$ with the normalized counting measure $\mu$ and consider the action $(T_g)_{g\in Q_8}$ of $Q_8$ on $(Q_8,\mathcal{P}(Q_8),\mu)$ by left translations, then $\mu(A)=\frac{1}{2}$ but 
\begin{equation}
\label{54terfd90reodfpslk}
\frac{1}{8}\sum_{g\in Q_8}\mu(A\cap g^2A)=\frac{1}{8}<\frac{1}{4}=\mu(A)^2.
\end{equation}
A similar example in an infinite group is given by $A\times\mathbb{Z}\subseteq Q_8\times\mathbb{Z}$.

We now turn our attention to finitely generated nilpotent groups. In \Cref{HereIsrweoidsioklfsd} we prove the following:
\begin{restatable}{theorem}{BigAvgsAlongSquaresFGNilp}
\label{BigAvgsAlongSquaresFGNilp}
For every finitely generated nilpotent group $G$ there exists $\lambda_G>0$ such that, for all m.p.s. $(X,\mathcal{B},\mu,(T_g)_{g\in G})$, $Y\in\mathcal{B}$ and all left or right F{\o}lner sequences $(F_N)$ in $G$,
\begin{equation*}
\lim_{N\to\infty}\frac{1}{|F_N|}\sum_{g\in F_N}\mu(Y\cap T_g^2Y)\geq \lambda_G\cdot\mu(Y)^2.
\end{equation*}
One can take $\lambda_G=1/[G:H]$, where $H$ is a torsion-free finite-index subgroup of $G$ such that $\{h^2;h\in H\}$ is a subgroup. Such an $H$ exists by \Cref{NilpotentSquareSubgroup}. 
\end{restatable}

\begin{cor}
\label{BBinAA-1FGNilpotent}
Let $(G,\cdot)$ be a finitely generated nilpotent group and let $A\subseteq G$ satisfy $d^*_l(A)>0$. Then for any left or right F{\o}lner sequence $F=(F_N)_N$ in $G$ there is some $B\subseteq G$ such that $\overline{d}_F(B)\geq\lambda_G d_l^*(A)^2$ and $BB\subseteq AA^{-1}$, where $\lambda_G$ is as in \Cref{BigAvgsAlongSquaresFGNilp}.
\end{cor} 

Our proof of \Cref{BigAvgsAlongSquaresFGNilp} combines elementary arguments with \Cref{ConvSquareAvgsMeasure} below. This convergence result holds for all countable amenable groups and follows from \Cref{ScalarConvergenceSquareAverages}, proved in \Cref{SecSquareAveragesAppendix}, applied to the Koopman representation and the vector $1_Y$:
\begin{theorem}
\label{ConvSquareAvgsMeasure}
If $G$ is a countable amenable group, $(X,\mathcal{B},\mu,(T_g)_{g\in G})$ is a m.p.s. and $Y\in\mathcal{B}$, the following limit exists and is the same for all left and right F{\o}lner sequences $(F_N)$:
\begin{equation}
\label{r9efdioslreiodflsk}
\lim_{N}\frac{1}{|F_N|}\sum_{g\in F_N}\mu(Y\cap T_g^2Y). 
\end{equation}
\end{theorem}
Curiously enough, for every integer $n\geq3$, the limit in \Cref{r9efdioslreiodflsk} need not exist, even along two-sided F{\o}lner sequences; see \Cref{HigherPowerMeasureCorrelationsCounterexample}.

In the setting of \Cref{ConvSquareAvgsMeasure}, we will denote the limit in \Cref{r9efdioslreiodflsk} by
\begin{equation*}
    \UClim_{g\in G}\mu(Y\cap T_g^2Y) \quad\textup{ or }\quad  \UClim_{g\in G}\mu(T_g^{-1}Y\cap T_gY)
\end{equation*}

We move now to discussing some variants of \Cref{B+BinA-AAbelian} for solvable groups of exponential growth. First, we introduce the following definition.
\begin{definition}
\label{DefWeakEpsDeltaAvgs}
Let $G$ be a countable amenable group and $\varepsilon,\delta>0$. We say $G$ has the \textit{$(\varepsilon,\delta)$-SAR} (Symmetric Averaging Recurrence) property if for all m.p.s. $(X,\mathcal{B},\mu,T)$ and for all $Y\in\mathcal{B}$ such that $\mu(Y)\geq\varepsilon$, we have
\begin{equation}
\label{EqDefWeakEpsDeltaAvgs}
\UClim_{g\in G}\mu(T_g^{-1}Y\cap T_gY)\geq\delta.
\end{equation}
\end{definition}
We say $G$ has the SAR property if for all $\varepsilon>0$ there is $\delta>0$ such that $G$ has the $(\varepsilon,\delta)$-SAR property.

For example, it follows from \Cref{BigIntsAbelianSquares} that abelian groups have the $(\varepsilon,\varepsilon^2)$-SAR property for all $\varepsilon>0$, and so they have the SAR property. \Cref{BigAvgsAlongSquaresFGNilp} implies that finitely generated nilpotent groups also have the SAR property.

Finally, recall that for any given groups $H,K$ and any homomorphism $\phi:K\to\textup{Aut}(H)$, sending each $k\in K$ to an automorphism $\phi_k:H\to H$, we denote by $G=H\rtimes_\phi K$ the semidirect product group with underlying set $H\times K$ and with operation 
\begin{equation*}
(h,k)\cdot(h',k')=(h\phi_k(h'),kk').
\end{equation*}

In \Cref{SecSemiProds} we prove that the SAR property is preserved under certain semidirect products:
\begin{restatable}{theorem}{BigSqAvgsInSemiprods}
\label{BigSqAvgsInSemiprods}
Suppose $\varepsilon,\delta>0$ and $G=H\rtimes_\phi K$ is a semidirect product of two countable groups $H,K$. If $H$ is abelian and $K$ has the $(\varepsilon,\delta)$-SAR property, then $G$ has the $(\varepsilon,\delta^2)$-SAR property.
\end{restatable}
As abelian groups have the $(\varepsilon,\varepsilon^2)$-SAR property for all $\varepsilon>0$, \Cref{BigSqAvgsInSemiprods}, together with \Cref{ReturnsInDifferences,ProdsInRetsMeasurableIntro}, implies the following:
\begin{cor}
Let $G$ be a semidirect product of two abelian groups and let $A\subseteq G$ satisfy $d^*_l(A)>0$. Then for any left (or right) F{\o}lner sequence $F=(F_N)$ in $G$ there is some $B\subseteq G$ such that $\overline{d}_F(B)\geq d_l^*(A)^4$ and $BB\subseteq AA^{-1}$.
\end{cor}

Examples of semidirect products of two abelian groups which have exponential growth are the lamplighter group\footnote{The lamplighter group is the semidirect product $\left(\bigoplus_{n\in\mathbb{Z}}\mathbb{Z}_2\right)\rtimes_\phi\mathbb{Z}$, where for each $k\in\mathbb{Z}$, $\phi_k:\bigoplus_{n\in\mathbb{Z}}\mathbb{Z}_2\to\bigoplus_{n\in\mathbb{Z}}\mathbb{Z}_2$ is given by $\phi_k((a_n)_{n\in\mathbb{Z}})=(a_{n-k})_{n\in\mathbb{Z}}$.
} $\mathbb{Z}_2\wr\mathbb{Z}$, and the affine group over $\mathbb{Q}$, $\mathrm{Aff}_2(\mathbb{Q})$.

\Cref{BigSqAvgsInSemiprods} also allows us to prove a variant of \Cref{B+BinA-AAbelian} for some nilpotent matrix groups which are not finitely generated.
Let \(n\in\mathbb{N}\) and let \(R\) be a countable commutative ring, not necessarily unital; if \(R\) has no multiplicative identity, we view it as a subring of a unital ring with identity \(1\). Let \(\mathrm{UT}_n(R)\) denote the group of upper triangular matrices
\[
\begin{pmatrix}
1 & a_{1,2} & a_{1,3} & \cdots & a_{1,n-1} & a_{1,n} \\
0 & 1 & a_{23} & \cdots & a_{2,n-1} & a_{2,n} \\
0 & 0 & 1 & \cdots & a_{3,n-1} & a_{3,n} \\
\vdots & \vdots & \vdots & \ddots & \vdots & \vdots \\
0 & 0 & 0 & \cdots & 1 & a_{n-1,n} \\
0 & 0 & 0 & \cdots & 0 & 1
\end{pmatrix}, \ a_{ij} \in R.
\]
 
Groups of the form $\textup{UT}_n(R)$ also have the SAR property (see \Cref{SecSemiProds}):
\begin{restatable}{theorem}{HnRHasBiSquareAverages}
\label{HnRHasBiSquareAverages}
Let $R$ be any countable commutative ring, $G=\mathrm{UT}_n(R)$ for some $n\in\mathbb{N}$, and $A\subseteq G$. Then for any left or right F{\o}lner sequence $F=(F_N)$ in $G$ there exists $B\subseteq G$ such that $\overline{d}_F(B)\geq d^*_l(A)^{2^{n-1}}$ and $BB\subseteq AA^{-1}$.
\end{restatable}

Here is a similar result involving some non-solvable groups, which will be proved in \Cref{SecSemiProds}.
\begin{theorem}
\label{BBinAA-1SomeNonSolvableMatrixGroups}
Let $Q$ be a countable field of finite characteristic, which is an algebraic extension over its prime field, and let $n\in\mathbb{N}$. Then $G:=\textup{GL}_n(Q)$ is amenable and, for all m.p.s. $(X,\mathcal{B},\mu,(T_g)_{g\in G})$ and $Y\in\mathcal{B}$, we have
\begin{equation*}
\UClim_{g\in G}\mu(Y\cap T_g^2Y)\geq\frac{\mu(Y)^2}{n!}.
\end{equation*}
Thus, for every $A\subseteq G$ and every left or right F{\o}lner sequence $F=(F_N)$ in $G$, there is $B\subseteq G$ such that $\overline{d}_F(B)\geq \frac{d^*_l(A)^2}{n!}$ and $BB\subseteq AA^{-1}$.
\end{theorem}

\Cref{BBinAA-1FGNilpotent} and \Cref{HnRHasBiSquareAverages,BBinAA-1SomeNonSolvableMatrixGroups} can be seen as natural analogues of \Cref{B+BinA-AAbelian} for some classes of non abelian groups. It is unclear whether such an analog holds for all countable amenable groups, which leads to the following question:

\begin{restatable}{question}{BigQuest}
\label{BigQuest}
Let $G$ be a countable amenable group, $(X,\mathcal{B},\mu,(T_g)_{g\in G})$ a m.p.s. and $Y\in\mathcal{B}$ with $\mu(Y)>0$. Is it true that
\begin{equation*}
\UClim_{g\in G}\mu\left(Y\cap T_g^2Y\right)>0?
\end{equation*}
\end{restatable}

A positive answer to \Cref{BigQuest} would imply that for all countable amenable groups $G$, all $A\subseteq G$ with $d^*(A)>0$ and all F{\o}lner sequences $F=(F_N)$ in $G$, there is $B\subseteq G$ such that $\overline{d}_F(B)>0$ and $BB\subseteq AA^{-1}$. We do not have this result, but we have a positive answer to the following weaker version of \Cref{BigQuest}:

\begin{question}
\label{3rewds09poefd9soc}
Given a countable amenable group $G$, $A\subseteq G$ with $d^*_l(A)>0$ and a left F{\o}lner sequence $(F_N)$ in $G$, are there sets $B_1,B_2\subseteq G$ such that $\overline{d}_F(B_1),\overline{d}_F(B_2)>0$ and $B_1B_2\subseteq AA^{-1}$?
\end{question}
The following theorem, proved in \Cref{SecErgToComb}, shows that not only \Cref{3rewds09poefd9soc} has a positive answer, but one can also guarantee that the set $B_2$ from \Cref{3rewds09poefd9soc} is a left translate of $B_1$. 
Moreover, the result holds for arbitrary sequences $(F_N)$ of finite subsets of $G$, not just left F{\o}lner sequences. 

\begin{theorem}
\label{rewds9edwpspo}
Let $G$ be a group, $F=(F_N)$ a sequence of finite, nonempty subsets of $G$, $(X,\mathcal{B},\mu,(T_g)_{g\in G})$ a m.p.s., $Y\in\mathcal{B}$ and $\varepsilon>0$. Then there exists $h\in G$ and $B\subseteq G$ such that $\overline{d}_F(B)\geq\mu(Y)^2-\varepsilon$ and 
\begin{equation*}
h^2BB\subseteq\{g\in G;\mu(Y\cap T_gY)>0\}.
\end{equation*}
\end{theorem}

\begin{cor}
\label{BBhinAA-1}
Let  $F=(F_N)$ be a sequence of finite nonempty subsets of a countable amenable group $G$, let $A\subseteq G$ have $d^*_l(A)>0$ and let $\varepsilon>0$. Then there exist $B\subseteq G$ with $\overline{d}_F(B)>d^*(A)^2-\varepsilon$ and $h\in G$ such that $h^2BB\subseteq AA^{-1}$. 
\end{cor}
The conclusion of \Cref{BBhinAA-1} remains true with $BBh^2$ in place of $h^2BB$, see \Cref{BBhinMeasurableReturns}. However, the corresponding conclusion with $BB$ in place of $h^2BB$ is false, even if $(F_N)$ is a left F{\o}lner sequence. For example, letting $A\subseteq Q_8$ be as in \Cref{54terfd90reodfpslk} and $A_1=A\times\mathbb{Z}\subseteq Q_8\times\mathbb{Z}$, we proved in the discussion after \Cref{B+BinA-AAbelian} that $d^*(A_1)=\frac{1}{2}$, but there is no $B\subseteq Q_8\times\mathbb{Z}$ such that $d^*(B)>\frac{1}{8}$ and $BB\subseteq A_1A_1^{-1}$.

The following result (proved in \Cref{SecErgToComb}) holds for any sequence $(F_N)$ of finite nonempty sets.

\begin{theorem}
\label{BkBkinAA-1}
Let $G$ be a countable amenable group, let $A\subseteq G$ have $d^*(A)>0$ and let $F=(F_N)$ be a sequence of finite, nonempty subsets of $G$. Then there exist $k\in\mathbb{N}$ and $B\subseteq G$ such that $\overline{d}_F(B)>0$ and $B^kB^k:=\{x^ky^k;x,y\in B\}\subseteq AA^{-1}$. We can take $k\leq\frac{1}{d^*(A)^2}+1$.
\end{theorem}

Recall that a set $R\subseteq\mathbb{N}$ is a \textit{set of measurable recurrence} if, for all m.p.s. $(X,\mathcal{B},\mu,(T^n)_{n\in\mathbb{Z}})$ and all $Y\in\mathcal{B}$ with $\mu(Y)>0$, there exists $r\in R$ such that $\mu(Y\cap T^{-r}Y)>0$. There are many nontrivial examples of sets of recurrence, including $\{n^2;n\in\mathbb{N}\}$ and $\{p-1;p\textup{ prime}\}$.

The technique which is used in the proof of \Cref{BkBkinAA-1} allows us to prove (in \Cref{SecErgToComb}) the following result. 
\begin{restatable}{prop}{RestatableB}
\label{BkBkinAA-1ConcreteCase}
Let $C$ be an infinite set of natural numbers, $R\subseteq\mathbb{N}$ a set of measurable recurrence and let $A\subseteq\mathbb{N}$ satisfy $d^*(A)>0$. Then there exists $r\in R$ and $B\subseteq C$ such that
\begin{equation*}
\overline{d}_C(B):=\overlim_{N\to\infty}\frac{|B\cap\{1,\dots,N\}|}{|C\cap\{1,\dots,N\}|}>0
\end{equation*} 
and $r(B+B)\subseteq A-A$.
\end{restatable}

\msubsection{Polynomial results}
In this subsection we discuss results of polynomial nature. Most of the results in this section are amplifications of the following polynomial variant of \Cref{BxBinA-AinZ} obtained in \cite{BeRu09}:

\begin{theorem}[{\cite[Theorem 3.1]{BeRu09}}]
\label{AbelianCartesianProdPolys}
Let $A\subseteq\mathbb{Z}^2$ and suppose $\overline{d}_F(A)>0$, where $F=(\{-N,\dots,N\})_{N\in\mathbb{N}}$. Let $p_i,q_i\in\mathbb{Z}[x]$, $i=1,\dots,m$, satisfy $p_i(0)=q_i(0)=0$ for all $i$. Then there exists $B\subseteq\mathbb{Z}$ such that $\overline{d}_F(B)>0$ and
\begin{equation*}
\bigcup_{i=1}^m(p_i(B)\times q_i(B))\subseteq A-A,
\end{equation*}
where $p_i(B):=\{p_i(b);b\in B\}$.
\end{theorem}

Recent work of Kosz, Mirek, Peluse, Wan and Wright \cite{KMPWW24} yields the following polynomial analogue of \Cref{CartesianSquareAsymptoticDensityInN} in which $B$ has positive asymptotic density. It should be noted though, that 
\Cref{CartesianSquareAsymptoticDensityInN} is not a special case of \Cref{PolynomialCartesianProductAsymptoticDensityInN},
since the result which we use from \cite{KMPWW24} requires the polynomials $p_1,p_2$ which appear in the following formulation to have distinct degrees.

\begin{restatable}{theorem}{PolynomialCartesianProductAsymptoticDensityInN}
\label{PolynomialCartesianProductAsymptoticDensityInN}
Let $A\subseteq\mathbb{Z}^2$ satisfy $d^*(A)>0$, and let $p_1,p_2\in\mathbb{Z}[x]$ be nonconstant polynomials of distinct degrees with $p_1(0)=p_2(0)=0$. Then there exists $B\subseteq\mathbb{N}$ with positive asymptotic density such that
\[
p_1(B)\times p_2(B)\subseteq A-A.
\]
\end{restatable}

\Cref{PolynomialCartesianProductAsymptoticDensityInN} is proved in \Cref{SecErgToComb}.

For polynomial maps over countably infinite fields, the polynomial mean ergodic theorem of Larick \cite{Lar98}, in the form given in \cite[Theorem 3.10]{BLM05}, implies the following quantitative result.

\begin{restatable}{theorem}{PolynomialCartesianProductInFields}
\label{PolynomialCartesianProductInFields}
Let $K$ be a countably infinite field, let $A\subseteq K^2$ satisfy $d^*(A)>0$, and let $p_1,p_2\in K[x]$ satisfy $p_1(0)=p_2(0)=0$. Then, for every F{\o}lner sequence $F=(F_N)$ in $(K,+)$, there exists $B\subseteq K$ such that
\[
\overline{d}_F(B)\geq d^*(A)^2
\qquad\text{and}\qquad
p_1(B)\times p_2(B)\subseteq A-A.
\]
\end{restatable}

\begin{remark}
\label{FiniteUnionsPolynomialCartesianProducts}
\Cref{PolynomialCartesianProductAsymptoticDensityInN,PolynomialCartesianProductInFields} also have versions for finite unions of Cartesian products, similarly to \Cref{AbelianCartesianProdPolys}.
\begin{enumerate}
    \item Let $A\subseteq\mathbb{Z}^2$ satisfy $d^*(A)>0$, and let $p_1,\dots,p_r,q_1,\dots,q_r\in\mathbb{Z}[x]$ be nonconstant polynomials with zero constant term and pairwise distinct degrees. Then there exists $B\subseteq\mathbb{N}$ with positive asymptotic density such that
    \[
        \bigcup_{i=1}^r\bigl(p_i(B)\times q_i(B)\bigr)\subseteq A-A.
    \]
    \item Let $K$, $A$ and $F$ be as in \Cref{PolynomialCartesianProductInFields}, and let $p_1,\dots,p_r,q_1,\dots,q_r\in K[x]$ have zero constant term. Then there exists $B\subseteq K$ such that
    \[
        \overline{d}_F(B)\geq d^*(A)^{2r}
        \qquad\text{and}\qquad
        \bigcup_{i=1}^r\bigl(p_i(B)\times q_i(B)\bigr)\subseteq A-A.
    \]
\end{enumerate}
These assertions are proved in \Cref{SecErgToComb}.
\end{remark}

The following general version of \Cref{AbelianCartesianProdPolys}, which is a special case of \Cref{NilpPolyBB} below, allows for a richer variety of combinatorial corollaries.

\begin{theorem}
\label{AbPolyBB}
Let $d,e,s\in\mathbb{N}$, for $i=1,\dots,s$ let $p_i,q_i:\mathbb{Z}^d\to\mathbb{Z}^e$ be polynomials\footnote{That is, $p(x_1,\dots,x_d)=\left(p_1(x_1,\dots,x_d),\dots,p_e(x_1,\dots,x_d)\right)$, where $p_i\in\mathbb{Q}[x_1,\dots,x_d]$.} satisfying $p(0)=q(0)=0$, and let $A\subseteq\mathbb{Z}^e$ satisfy $d^*(A)>0$. Then, for every F{\o}lner sequence $(F_N)_{N\in\mathbb{N}}$ in $\mathbb{Z}^d$, there is a set $B\subseteq\mathbb{Z}^d$ such that $\overline{d}_F(B)>0$ and 
\begin{equation*}
\bigcup_{i=1}^s(p_i(B)+q_i(B))\subseteq A-A.
\end{equation*}
\end{theorem}

The following examples demonstrate the versatility of \Cref{AbPolyBB}.

\begin{example}
For all $A\subseteq\mathbb{N}$ with $d^*(A)>0$, there exists $B\subseteq\mathbb{N}$ such that $\overline{d}(B)>0$ and $\{b_1^2+b_2^2;b_1,b_2\in B\}\subseteq A-A$.
\end{example}

\begin{example}
For all $A\subseteq\mathbb{Z}^4$ such that $d^*(A)>0$, there exists $B\subseteq\mathbb{Z}^2$ such that $\overline{d}_{\left(\{1,\dots,N\}\times\{1,\dots,N\}\right)_{N\in\mathbb{N}}}(B)>0$ and for all $(x,y),(z,w)\in B$ we have
\begin{equation*}
\left(x^{10}+z^{10},x^{10}+w^{10},y^{10}+z^{10},y^{10}+w^{10}\right)\in A-A.
\end{equation*}
\end{example}

Note that \Cref{B+BinA-AinZ} is a very special case of \Cref{AbPolyBB}, which can be obtained by letting $d=1,e=2$ and considering the polynomials $n\mapsto(p_i(n),0)$ and $n\mapsto(0,q_i(n))$. 

\begin{remark}
The example explained in \Cref{BeRuCounterBxBxB} implies that \Cref{AbPolyBB} cannot be strengthened to a statement of the form 
\begin{equation*}
p(B)+q(B)+r(B)\subseteq A-A.
\end{equation*}
Indeed, let $p,q,r:\mathbb{Z}\to\mathbb{Z}^3$ be given by $p(x)=(x,0,0)$, $q(y)=(0,y,0)$, $r(z)=(0,0,z)$. Then by \Cref{BeRuCounterBxBxB} and \Cref{Returns=DifferencesIntro} there exists $A\subseteq\mathbb{Z}^3$ such that $d^*(A)>0$, but any $B\subseteq\mathbb{Z}$ such that $\overline{d}(B)>0$ does not satisfy 
\begin{equation*}
p(B)+q(B)+r(B)=B\times B\times B\subseteq A-A.
\end{equation*}
\end{remark}

There is a natural notion of a polynomial map from an arbitrary group $G$ to any nilpotent group $H$, studied by Leibman in \cite{LePoly}. We define polynomial maps in \Cref{SecNilpotent}. For now, we list some of their basic properties which help motivate the role of polynomial maps in the sequel.
\begin{itemize}
    \item If $G=\mathbb{Z}^m$ and $H=\mathbb{Z}^n$ for some $n,m\in\mathbb{N}$, then a map $p:G\to H$ is polynomial iff it is a polynomial in the usual sense (details are provided in \Cref{LeibPolysZn}).
    \item If $p,q:G\to H$ are polynomial, then $g\mapsto p(g)q(g)$ and $g\mapsto p(g)^{-1}$ are polynomial maps from $G$ to $H$.
    \item The composition of two polynomial maps between nilpotent groups is polynomial.
    \item Homomorphisms are polynomial maps.
\end{itemize}
For example, if $G$ is a nilpotent group and $g_1,g_2\in G$, then $f_1:G\to G;g\mapsto g^2$ and $f_2:\mathbb{Z}\to G;f(n)=g_1^{n^2+n}g_2^{-3n^7}$ are polynomial maps. 

The following result is a  generalization of \Cref{AbPolyBB} to finitely generated nilpotent groups, to be proved in \Cref{SecNilpotent}.

\begin{restatable}{theorem}{NilpPolyBB}
\label{NilpPolyBB}
Let $G$ be a finitely generated group and $H$ a nilpotent group, with polynomial maps $p_j,q_j:G\to H$, $j=1,\dots,s$, satisfying $p_i(1_G)=q_i(1_G)=1_H$, and let $A\subseteq H$ satisfy $d_l^*(A)>0$. Then for every F{\o}lner sequence $F=(F_N)_{N\in\mathbb{N}}$ in $G$ there is a set $B\subseteq G$ such that $\overline{d}_F(B)>0$ and 
\begin{equation*}
\bigcup_{j=1}^sp_j(B) q_j(B)\subseteq AA^{-1}.
\end{equation*}
\end{restatable}

Substituting $j=1$, $H=G$ and $p_1(g)=q_1(g)=g$ for all $g\in G$ in \Cref{NilpPolyBB}, we obtain a general version of \Cref{BBinAA-1FGNilpotent}, but without the explicit bounds on $\overline{d}_F(B)$.

The cartesian product version of \Cref{NilpPolyBB} is a natural corollary of \Cref{NilpPolyBB} itself, because for any nilpotent group $G$ and any polynomial map $p:G\to G$, the maps $G\to G\times G$ given by $g\mapsto(1_G,p(g))$ and $g\mapsto(p(g),1_G)$ are polynomial.
\begin{restatable}{cor}{NilpPolyBxB}
\label{NilpPolyBxB}
Let $G$ be a finitely generated group and $H$ a nilpotent group, with polynomial maps $p_i,q_i:G\to H$, $j=1,\dots,s$, satisfying $p_i(1_G)=q_i(1_G)=1_H$, and let $A\subseteq H\times H$ satisfy $d^*_l(A)>0$. Then for every F{\o}lner sequence $F=(F_N)_{N\in\mathbb{N}}$ in $G$ there is a set $B\subseteq G$ such that $\overline{d}_F(B)>0$ and 
\begin{equation*}
\bigcup_{j=1}^sp_j(B)\times q_j(B)\subseteq AA^{-1}.
\end{equation*}
\end{restatable}

\Cref{NilpPolyBB} is deduced in \Cref{SecNilpotent} from \Cref{PolyCesaroNilpotent}, a variant of a polynomial multiple recurrence theorem for finitely generated nilpotent groups, which is a non trivial consequence of the work of Zorin-Kranich (\cite{ZK14,Zo16}). Before formulating \Cref{PolyCesaroNilpotent}, we have to introduce some definitions.

Let $\mathcal{P}_f(\mathbb{N})$ be the family of finite subsets of $\mathbb{N}$, and $\mathcal{F}=\mathcal{P}_f(\mathbb{N})\setminus\{\varnothing\}$. The \textit{(left) IP-system} generated by a sequence $(g_n)$ of elements of a group $G$ is the map $g:\mathcal{P}_f(\mathbb{N})\to G$ given by 
\begin{equation*}
g(\alpha)=\prod_{i\in\alpha}^<g_i:=g_{i_1}\cdots g_{i_k}\textup{ (so $g(\varnothing)=1_G$)},
\end{equation*}
where $\alpha=\{i_1,\dots,i_k\}$ and $i_1<i_2<\dots<i_k$. A (left) IP set in $G$ is the set $g(\mathcal{F})$ associated to some IP-system $g$, and we say a set $A\subseteq G$ is IP$^*$ in $G$ if it has nonempty intersection with every IP set in $G$.
\begin{restatable}{theorem}{PolyCesaroNilpotent}
\label{PolyCesaroNilpotent}
Let $G$ be a finitely generated group and let $H$ be a nilpotent group. For some $s\in\mathbb{N}$ let $p_1,\dots,p_s:G\to H$ be polynomial maps such that $p_j(1_G)=1_H$ for all $j=1,\dots,s$. Then for any m.p.s. $(X,\mathcal{B},\mu,(T_a)_{a\in H})$ and any $B\in\mathcal{B}$ with $\mu(B)>0$, 
the set of returns
\begin{equation}
\label{EqReturnsNilpPoly}
R=\left\{g\in G;\mu\left(\bigcap_{j=1}^s T_{p_j(g)}B\right)>0\right\}
\end{equation}
is IP$^*$ in $G$.
Moreover, if $G$ is amenable and $(F_N)$ is a left or right F{\o}lner sequence in $G$, then
\begin{equation}
\label{EqBigAvgReturnsPolyNil}
\lim_{N\to\infty}\frac{1}{|F_N|}\sum_{g\in F_N}\mu\left(\bigcap_{j=1}^s T_{p_j(g)}B\right)=:\lambda(B)>0,
\end{equation}
and the constant $\lambda(B)$ does not depend on the F{\o}lner sequence. 
\end{restatable}

\paragraph{Structure of the paper.} 
\Cref{SecPrelims}, provides background material required for the rest of the paper. 
In \Cref{SecErgToComb}, we prove our main technical result, \Cref{TheBigTechnicalTheorem}, and deduce from it \Cref{GeneralBxBinReturns,BigAvgsBallFreeGroup,B+BinA-AAbelian,BxBinAAminus1twosided,CartesianSquareAsymptoticDensityInN,PolynomialCartesianProductAsymptoticDensityInN,PolynomialCartesianProductInFields}, \Cref{CartProdInRets}, and \Cref{BkBkinAA-1ConcreteCase}. 
In \Cref{SecGeneralDensities}, we prove \Cref{TheBigTechnicalTheoremWithProbMeasures}, a generalization of \Cref{TheBigTechnicalTheorem} in which upper density along a sequence \((F_N)\) of finite subsets is replaced by upper density with respect to a sequence \((\mu_N)\) of probability measures on \(G\). We then apply \Cref{TheBigTechnicalTheoremWithProbMeasures} to derive \Cref{B+BinA-AinZLogDensity} and establish \Cref{ThmConvsAreNiceAndHaveBxBinAA-1}, a version of \Cref{CartProdInRetsFree} for arbitrary countable groups.
In \Cref{SecNilpotent}, we establish our main results for finitely generated nilpotent groups, namely \Cref{NilpPolyBB,PolyCesaroNilpotent,BigAvgsAlongSquaresFGNilp}.
In \Cref{SecSemiProds}, we show that some classes of amenable groups, other than abelian and finitely generated nilpotent groups, have the SAR property; in particular, we prove \Cref{BigSqAvgsInSemiprods,HnRHasBiSquareAverages,BBinAA-1SomeNonSolvableMatrixGroups}. 
\Cref{SecCounterexamples} contains the counterexamples mentioned in 
\Cref{Rmk113,Rmk120}. In \Cref{SecQuests} we formulate some questions.
Finally, in Appendix \ref{SecSquareAveragesAppendix}, which is a result of collaboration with AI, we give a proof of \Cref{ConvSquareAvgsMeasure}.
Appendix \ref{SecSquareAveragesAppendix} also contains an example of non-convergence of $\frac{1}{|F_N|}\sum_{g\in F_N}\mu(A\cap T_g^nA)$, where $n\geq3$ and $(F_N)$ is a double-sided F{\o}lner sequence in a countable amenable group $G$.

%% file: S2-Preliminaries.tex
\section{Preliminaries}
\label{SecPrelims}
This section contains the background material needed for the remainder of the paper. 

We start with a few basic facts about F{\o}lner sequences.

\begin{lemma}
\label{FolnerProps}
Let $G$ be a countable amenable group and let $A\subseteq G$.
\begin{enumerate}[label=(\alph*)]
    \item
\label{UBDisamax} There exists a left F{\o}lner sequence $F$ in $G$ such that $d_F(A)=d_l^*(A)$.
    \item \label{FleftFolneriffF-1rightFolner}$(F_N)_{N\in\mathbb{N}}$ is a left F{\o}lner sequence in $G$ iff $\left(F_N^{-1}\right)_{N\in\mathbb{N}}$ is a right F{\o}lner sequence in $G$.
    \item\label{LeftAndRightUBD} For all $g\in G$ we have $d^*_l(gA)=d^*_l(Ag)=d^*_l(A)$, and $d_l^*(A)=d_r^*(A^{-1})$. 
    \item \label{AutosPreserveFolner} For any left F{\o}lner sequence $F=(F_N)$ in $G$ and any automorphism $\phi:G\to G$, $(\phi(F_N))_{N\in\mathbb{N}}$ is a left F{\o}lner sequence in $G$.
\end{enumerate}

\end{lemma}
\begin{proof}

We start with proving \ref{UBDisamax}.
For each $k\in\mathbb{N}$ let $F_k=(F_{k,N})_{N\in\mathbb{N}}$ be a left F{\o}lner sequence satisfying $\overline{d}_{F_k}(A)> d^*_l(A)-\frac{1}{k}$. Let $(g_k)_{k\in\mathbb{N}}$ be a list of all elements of $G$. Then for each $k\in\mathbb{N}$ we can choose some big $N_k$ such that 
\begin{align*}
\frac{|A\cap F_{k,N_k}|}{|F_{k,N_k}|}&>d^*_l(A)-\frac{1}{k}\\
\frac{|g_jF_{k,N_k}\Delta F_{k,N_k}|}{|F_{k,N_k}|}&<\frac{1}{k}\textup{ for all }j=1,\dots,k.
\end{align*}
It follows that $(F_{k,N_k})_{k\in\mathbb{N}}$ is a left F{\o}lner sequence satisfying $\underline{d}_F(A)\geq d^*_l(A)$. By definition of $d^*_l(A)$, we also have $\overline{d}_F(A)\leq d^*_l(A)$, so we are done.

\Cref{FleftFolneriffF-1rightFolner} is immediate from the definitions.

The second part of \ref{LeftAndRightUBD} follows from \ref{FleftFolneriffF-1rightFolner}. We now prove the first part of \ref{LeftAndRightUBD}. If $g\in G$, then $d_{F}(A)=d_F(gA)$ for all left F{\o}lner sequences $F$, so $d^*_l(A)=d^*_l(gA)$. Moreover, for any left F{\o}lner sequence $(F_N)$, the sequence $(F_Ng)$ is also left F{\o}lner, and $d_{(F_N)}(A)=d_{(F_Ng)}(Ag)$, so taking the supremum for all F{\o}lner sequences we obtain $d^*_l(A)=d^*_l(Ag)$.

Finally, \ref{AutosPreserveFolner} follows from the fact that, for all $a\in G$ and $N\in\mathbb{N}$, 
\begin{equation*}
|\phi(F_N)\Delta a\phi(F_N)|=|F_N\Delta\phi^{-1}(a)F_N|.\qedhere
\end{equation*}
\end{proof}

We will recall now some basic facts from ergodic theory. Let $G$ be a countable amenable group, $(X,\mathcal{B},\mu,(T_g)_{g\in G})$ a m.p.s., and $\langle\cdot,\cdot\rangle$ the inner product in $L^2(X,\mu)$. The left action $(T_g)_{g\in G}$ on $X$ induces a right unitary antiaction on $L^2(X)=L^2(X,\mu)$, given by $T_gf=f\circ T_g$ (so $T_gT_hf=T_{hg}f$). 

Let $\pi:L^2(X)\to L^2(X)$ be orthogonal projection to the subspace of $T$-invariant functions (note that $1_X$ is in the image of $\pi$). Then, for all $Y\in\mathcal{B}$, we have
\begin{equation}
\label{InnerProdWithProjection}
\langle 1_Y,\pi1_Y\rangle=\langle\pi1_Y,\pi1_Y\rangle=\langle\pi1_Y,\pi1_Y\rangle\cdot\langle1_X,1_X\rangle\geq\langle\pi 1_Y,1_X\rangle^2=\langle1_Y,1_X\rangle^2=\mu(Y)^2.
\end{equation}
Von Neumann's ergodic theorem for amenable groups (see, for example, \cite[Theorem II.1.3]{OW}) implies that, for all right F{\o}lner sequences $(F_N)$ in $G$ and $f\in L^2(X)$, 
\begin{equation}
\label{VonNeumannAmenableL2}
\lim_{N\to\infty}\frac{1}{|F_N|}\sum_{g\in F_N}T_gf=\pi f.
\end{equation}
It follows from \Cref{VonNeumannAmenableL2} and \Cref{InnerProdWithProjection} that, if $(F_N)$ is a left F{\o}lner sequence,
\begin{equation}
\label{LinearAvgsAreBig}
\lim_{N\to\infty}\frac{1}{|F_N|}\sum_{g\in F_N}\mu(Y\cap T_gY)
=
\lim_{N\to\infty}\frac{1}{|F_N|}\sum_{g\in F_N}
\left\langle1_Y,T_g^{-1}1_Y\right\rangle
=
\left\langle1_Y,\pi1_Y\right\rangle
\geq\mu(Y)^2.
\end{equation}
\Cref{LinearAvgsAreBig} also holds for right F{\o}lner sequences, as $\mu(Y\cap T_gY)=\mu(Y\cap T_{g^{-1}}Y)$.

The following result is a slight generalization of \cite[Lemma 5.10]{Be3}; for the original version in $\mathbb{N}$ see \cite[Theorem 1.1]{Be85}. 
The general form of \Cref{IntLemma} ($G$ is a set and $(F_N)$ is a sequence of finite nonempty subsets of $G$) will allow us to get new results, such as Corollaries \ref{CartProdInRetsFree} and \ref{PrimesMinus1ArePIR}.

\begin{lemma}[Intersectivity lemma]
\label{IntLemma}
Let $G$ be a set, $F=(F_N)_{N\in\mathbb{N}}$ a sequence of finite, nonempty subsets of $G$, and $(Y_g)_{g\in G}$ measurable subsets of a probability space $(X,\mathcal{B},\mu)$. Then there exists a set $B\subseteq G$ such that the following holds:
\begin{enumerate}[label=(\arabic*)]
\item\label{IntLemmaItem1} For each finite set $B_0\subseteq B$, we have $\mu\left(\bigcap_{b\in B_0}Y_b\right)>0$.
    \item \label{IntLemmaItem2}$\displaystyle\overline{d}_F(B)\geq \lambda=\overlim\limits_{N\to\infty}\frac{1}{|F_N|}\sum_{g\in F_N}\mu(Y_g)$.
\end{enumerate}
\end{lemma}

\begin{remark}
\label{FinAddRemark}
As explained in the discussion after \cite[Lemma 4.4]{BBF}, Lemma \ref{IntLemma} also holds for finitely additive probability spaces $(X,\mathcal{B},\mu)$, where $\mathcal{B}\subseteq\mathcal{P}(X)$ is an algebra and $\mu:\mathcal{B}\to[0,1]$ is a finitely additive probability measure.
\end{remark}

\begin{remark}
\label{IntLemmaDensityExists}
The proof of \Cref{IntLemma} also shows that, if the averages
\[
    \frac{1}{|F_N|}\sum_{g\in F_N}1_{Y_g}(x)
\]
converge for $\mu$-almost every $x\in X$, then $B$ can be chosen so that its density $d_F(B)$ exists. 
\end{remark}

\begin{proof}[Proof of \Cref{IntLemma}]
We may assume $G=\bigcup_NF_N$, so that $G$ is countable. First let $\mathcal{F}$ be the family of finite subsets $G_0\subseteq G$ such that $\mu\left(\cap_{g\in G_0}Y_g\right)=0$, and let
\begin{equation*}
X'=X\setminus\bigcup_{G_0\in\mathcal{F}}\cap_{g\in G_0}Y_g,
\end{equation*}
so that $\mu(X')=1$. Let $f:X\to[0,1]$ be given by 
\begin{equation*}
f(x)=\overlim_{N\to\infty}\frac{1}{|F_N|}\sum_{g\in F_N}1_{Y_g}(x).
\end{equation*}
 By Fatou's lemma, we have
\begin{equation*}
\int_{X'}fd\mu=\int_Xfd\mu\geq\overlim_{N\to\infty}\int_X\frac{1}{|F_N|}\sum_{g\in F_N}1_{Y_g}d\mu=\lambda.
\end{equation*}
Thus, there is a point $x_0\in X'$ such that $f(x_0)\geq \lambda$. Let $B=\{g\in G;x_0\in Y_g\}$. For any finite $B_0\subseteq B$, we have $x\in X'\cap_{g\in B_0}Y_g$, so that, by definition of $X'$, \ref{IntLemmaItem1} must hold. The fact that $f(x_0)\geq\lambda$ means that \ref{IntLemmaItem2} holds too.
\end{proof}

We now recall the close relationship between sets of differences and sets of returns. If $G$ is a group, $Y\subseteq G$ and $F=(F_N)$ is a sequence  of finite, nonempty subsets of $G$, the set of \textit{left returns} of $A$ is given by
\begin{equation*}
R^l_F(A)=\{g\in G;\overline{d}_F(A\cap gA)>0\}\subseteq AA^{-1}.
\end{equation*}
Similarly, the \textit{right returns} of $A$ are given by
\begin{equation*}
R^r_F(A)=\{g\in G;\overline{d}_F(A\cap Ag)>0\}\subseteq A^{-1}A.
\end{equation*}

For a m.p.s. $(X,\mathcal{B},\mu,(T_g)_{g\in G})$, we denote the \textit{set of measurable returns} of $Y\in\mathcal{B}$ is given by 
\begin{equation*}
R_\mu^T(Y)=\{g\in G;\mu(Y\cap T_gY)>0\}.
\end{equation*}

\begin{lemma}[Sets of returns vs sets of differences]
\label{Returns=Differences}
Let $G$ be an infinite countable amenable group with a left F{\o}lner sequence $F=(F_N)_{N\in\mathbb{N}}$.
\begin{enumerate}[label=(\arabic*)]
    \item\label{Returns=Differences2} For all $A\subseteq G$ there is a m.p.s. $(X,\mathcal{B},\mu,(T_g)_{g\in G})$ and $Y\in\mathcal{B}$ such that $\overline{d}_F(A)=\mu(Y)$ and $R^T_\mu(Y)\subseteq R^l_F(A)\subseteq AA^{-1}$. 
    
    In fact, for some subsequence $F'$ of $F$, $\overline{d}_F(A)=d_{F'}(A)$ and $d_{F'}(A\cap gA)=\mu(Y\cap T_gY)$ for all $g\in G$.
    
    \item\label{Returns=Differences1} For all m.p.s. $(X,\mathcal{B},\mu,(T_g)_{g\in G})$ and $Y\in\mathcal{B}$ there is $A\subseteq G$ such that $d_F(A)=\mu(Y)$ and $R^T_\mu(Y)=R^l_F(A)\subseteq AA^{-1}$.
    
    In fact, $d_F(A\cap gA)=\mu(Y\cap T_gY)$ for all $g\in G$.

    \item\label{Returns=Differences3}
    For all m.p.s. $(X,\mathcal{B},\mu,(T_g)_{g\in G})$ and $Y\in\mathcal{B}$ there is $A\subseteq G$ such that $\overline{d}_F(A)=\mu(Y)$ and $AA^{-1}\subseteq R^T_\mu(Y)$.
\end{enumerate}
\end{lemma}

\begin{remark}
We emphasize in \Cref{Returns=Differences} that $G$ is infinite because, if $G$ is finite, then in \Cref{Returns=Differences}.\ref{Returns=Differences2} we can only achieve the inequality $d_F(A)\geq\mu(Y)$ instead of $d_F(A)=\mu(Y)$, and similarly in \Cref{Returns=Differences}.\ref{Returns=Differences3}.
\end{remark}

\begin{proof}
\Cref{Returns=Differences2,Returns=Differences1} follow immediately from the version of Furstenberg's correspondence theorem stated in \cite[Theorem 3.14]{Ro25}. We now prove \ref{Returns=Differences3}. Let $Y_g=T_g(Y)$, so that $\mu(Y_g)=\mu(Y)$ for all $g$. As in the proof of \Cref{IntLemma}, there must be some $x_0\in X$ such that the set $A=\{g\in G;x_0\in Y_g\}$ satisfies $\overline{d}_F(A)\geq\mu(Y)$. Moreover, for all $g,g'\in A$ we have that $x_0\in T_gY\cap T_{g'}Y$, so $Y\cap T_{g^{-1}g'}Y\neq\varnothing$, which we may assume implies $\mu(Y\cap T_{g^{-1}g'}Y)>0$ (after choosing $x_0$ inside some appropriate set $X'\subseteq X$ of measure $0$, as in the proof of \Cref{IntLemma}), so $g^{-1}g'\in R^T_\mu(Y)$. By taking an appropriate subset $A'$ of $A$ we ensure $\overline{d}_F(A)=\mu(Y)$ instead of just $\overline{d}_F(A)\geq\mu(Y)$
\end{proof}

We can also apply \Cref{Returns=Differences} to the `opposite' group $G^{\textup{op}}$, which has the same underlying set as $G$ and binary operation $a\cdot_{\textup{op}}b=b\cdot a$. This yields a `right analogue' of \Cref{Returns=Differences}, obtained by changing the expressions \textit{left F{\o}lner}, $gA$, $AA^{-1}$ and $R_F^l$ to \textit{right F{\o}lner}, $Ag$, $A^{-1}A$ and $R_F^r$ respectively.

\begin{remark}
The following idea goes back to the proof of \cite[Theorem 1.2]{Be85}. 
Let $G$ be a countable amenable group and let $(F_N)$ be a left F{\o}lner sequence in $G$ such that for any m.p.s. $(X,\mathcal{B},\mu,(T_g)_{g\in G})$ and any $f\in L^\infty(\mu)$ the sequence $f_N=\frac{1}{|F_N|}\sum_{g\in F_N}f\circ T_g,N\in\mathbb{N}$, converges pointwise when $N\to\infty$ (it is well known that any countable amenable group admits such a F{\o}lner sequence).
Then the proof of \Cref{Returns=Differences}.\ref{Returns=Differences3} implies that $d_F(A)\geq\mu(Y)$, so that $A$ has a defined density. One can in fact obtain $d_F(A)=\mu(Y)$, using the (non trivial) 
 fact that for any F{\o}lner sequence $F=(F_N)$ in a countable amenable group $G$, any $A\subseteq G$ with density $d_F(A)=a$ and any $\lambda\in[0,a]$, there is some $E\subseteq A$ such that $d_F(E)=\lambda$.
\end{remark}

In the sequel we will need a version of \Cref{Returns=Differences}.\ref{Returns=Differences1} for upper Banach density:

\begin{lemma}
\label{ReturnsInDifferences}
Let $G$ be a countable amenable group and let $A\subseteq G$. There is a m.p.s. $(X,\mathcal{B},\mu,(T_g)_{g\in G})$ and $Y\in\mathcal{B}$ such that $\mu(Y)=d^*_l(A)$ and $R^T_\mu(Y)\subseteq AA^{-1}$.
\end{lemma}

\begin{proof}
It follows from \Cref{FolnerProps}.\ref{UBDisamax} and \Cref{Returns=Differences}.
\end{proof}

%% file: S3-ErgThToCombi.tex
\section{Main result and some quick applications}
\label{SecErgToComb}
In this section we state and prove \Cref{TheBigTechnicalTheorem}, a result about product sets in sets of measurable returns which, together with Furstenberg's correspondence principle, will allow us to obtain a variety of combinatorial corollaries.
In particular, we will use \Cref{TheBigTechnicalTheorem} to prove \Cref{GeneralBxBinReturns,BigAvgsBallFreeGroup,B+BinA-AAbelian,BxBinAAminus1twosided}, \Cref{CartProdInRets}, and \Cref{BkBkinAA-1ConcreteCase}.

\begin{theorem}[Product sets in sets of measurable returns]
\label{TheBigTechnicalTheorem}
Let $G$ be a set, $H$ a group, $\phi_1,\dots,\phi_k,\psi_1,\dots,\psi_k:G\to H$ functions, $(X,\mathcal{B},\mu,(T_h)_{h\in H})$ a m.p.s., $Y\in\mathcal{B}$ and $F=(F_N)_{N\in\mathbb{N}}$ a sequence of nonempty subsets of $G$. Let
\begin{equation*}
\lambda=\overlim\limits_{N\to\infty}\frac{1}{|F_N|}\sum_{g\in F_N}\mu\left(\bigcap_{i=1}^kT_{\phi_i(g)}Y\cap T_{\psi_i(g)}^{-1}Y\right).
\end{equation*}
Then there exists $B\subseteq G$ such that $\overline{d}_F(B)\geq \lambda$ and
\begin{equation*}
\bigcup_{i,j=1}^k\psi_i(B)\cdot \phi_j(B)\subseteq R_\mu^T(Y).
\end{equation*}
\end{theorem}

For a 
more general version of \Cref{TheBigTechnicalTheorem} where, instead of a sequence $(F_N)$ of finite subsets in $G$, we consider a finite sequence $(\nu_n)$ of probability measures on $G$, see \Cref{TheBigTechnicalTheoremWithProbMeasures}

\begin{proof}
For each $g\in G$ let $Y_g=\bigcap_{i=1}^kT_{\phi_i(g)}Y\cap T_{\psi_i(g)}^{-1}Y$. By \Cref{IntLemma} there is a set $B\subseteq G$ such that $\mu(Y_b\cap Y_b')>0$ for all $b,b'\in B$, and such that 
\begin{equation*}
\overline{d}_F(B)\geq\overlim\limits_{N\to\infty}\frac{1}{|F_N|}\sum_{g\in F_N}\mu(Y_g)=\lambda.
\end{equation*}
Moreover, for each $i,j\in\{1,\dots,k\}$ and $b,b'\in B$ we have
\begin{equation*}
\mu\left(T_{\psi_i(b)\phi_j(b')}Y\cap Y\right)=\mu\left(T_{\phi_j(b')}Y\cap T_{\psi_i(b)}^{-1}Y\right)\geq \mu(Y_{b'}\cap Y_{b})>0.
\end{equation*}
Thus, $\psi_i(b)\phi_j(b')\in R_\mu^T(Y)$.
\end{proof}

\begin{remark}
\label{TheBigTechnicalTheoremDensityExists}
If the averages
\[
    \frac{1}{|F_N|}\sum_{g\in F_N}
    1_{\bigcap_{i=1}^k\left(T_{\phi_i(g)}Y\cap T_{\psi_i(g)}^{-1}Y\right)}(x)
\]
converge for $\mu$-almost every $x\in X$, then the set $B$ in \Cref{TheBigTechnicalTheorem} can be chosen so that its density $d_F(B)$ exists and satisfies $d_F(B)\geq\lambda$ (see \Cref{IntLemmaDensityExists}).
\end{remark}

We state separately the case $k=1$ of \Cref{TheBigTechnicalTheorem}, as it will be repeatedly utilized in the sequel.

\begin{cor}
\label{ProdsInReturnsk=1}
Let $G$ be a set, $H$ a group, $\phi,\psi:G\to H$ functions, $(X,\mathcal{B},\mu,(T_h)_{h\in H})$ a m.p.s., $Y\in\mathcal{B}$ and $F=(F_N)_{N\in\mathbb{N}}$ a sequence of nonempty subsets of $G$. Let
\begin{equation*}
\lambda=\overlim\limits_{N\to\infty}\frac{1}{|F_N|}\sum_{g\in F_N}\mu\left(T_{\phi(g)}Y\cap T_{\psi(g)}^{-1}Y\right).
\end{equation*}
Then there exists $B\subseteq G$ such that $\overline{d}_F(B)\geq \lambda$ and
\begin{equation*}
\psi(B)\cdot \phi(B)\subseteq R_\mu^T(Y).
\end{equation*}
\end{cor}

Letting $H=G\times G$, $\phi(g)=(g,1_G)$ and $\psi(g)=(1_G,g)$ for $g\in G$, we obtain the following stronger version of \Cref{GeneralBxBinReturns}.

\begin{cor}
\label{CartProdsInReturnsk=1}
Let $G$ be a group,
$(X,\mathcal{B},\mu,(T_{g,h})_{(g,h)\in G\times G})$ a m.p.s. and $Y\in\mathcal{B}$. Then, for any sequence $F=(F_N)_{N\in\mathbb{N}}$ of nonempty subsets of $G$, letting
\begin{equation*}
\lambda=\overlim\limits_{N\to\infty}\frac{1}{|F_N|}\sum_{g\in F_N}\mu\left(T_{(g,1_G)}Y\cap S_{(1_G,g)}^{-1}Y\right)
=
\overlim\limits_{N\to\infty}\frac{1}{|F_N|}\sum_{g\in F_N}\mu\left(T_{(g,g)}Y\cap Y\right),
\end{equation*}
there exists $B\subseteq G$ such that $\overline{d}_F(B)\geq \lambda$ and
\begin{equation*}
B\times B\subseteq R_\mu^{T}(Y)=\left\{(g,h)\in G\times G;\mu(T_{(g,h)}Y\cap Y)>0\right\}.
\end{equation*}
\end{cor}

Taking $H=G$ and $\psi=\phi=\textup{Id}_G$ in \Cref{ProdsInReturnsk=1} yields the following result.

\begin{cor}[Products in sets of returns]
\label{ProdsInRetsMeasurable}
Let $G$ be a group,  $(X,\mathcal{B},\mu,(T_g)_{g\in G})$ a m.p.s., $Y\in\mathcal{B}$ and $F=(F_N)_{N\in\mathbb{N}}$ a sequence of nonempty subsets of $G$. Let
\begin{equation*}
\lambda=\overlim\limits_{N\to\infty}\frac{1}{|F_N|}\sum_{g\in F_N}\mu\left(T_{g}Y\cap T_{g}^{-1}Y\right)=\overlim\limits_{N\to\infty}\frac{1}{|F_N|}\sum_{g\in F_N}\mu\left(T_{g}^2Y\cap Y\right).
\end{equation*}
Then there exists $B\subseteq G$ such that $\overline{d}_F(B)\geq \lambda$ and
\begin{equation*}
BB\subseteq R_\mu^T(Y).
\end{equation*}
\end{cor}

\CartProdInRets*
\begin{proof}
Let $H=G\times G$, let $\phi:G\to H;g\mapsto(g,1_G)$ and $\psi:G\to H;g\mapsto(1_G,g)$. Note that, letting $S_g=T_{(g,g)}$, we have $S_g\circ S_h=S_{gh}$ for all $g,h\in G$, so $(X,\mathcal{B},\mu,(S_g)_{g\in G})$ is a m.p.s. We have, using \Cref{LinearAvgsAreBig},
\begin{align*}
\overlim\limits_{N\to\infty}\frac{1}{|F_N|}\sum_{g\in F_N}\mu\left(T_{\phi(g)}Y\cap T_{\psi(g)}^{-1}Y\right)
&=\overlim\limits_{N\to\infty}\frac{1}{|F_N|}\sum_{g\in F_N}\mu\left(T_{(g,g)}Y\cap Y\right)\\
&=\overlim\limits_{N\to\infty}\frac{1}{|F_N|}\sum_{g\in F_N}\mu\left(S_{g}Y\cap Y\right)\geq\mu(Y)^2.
\end{align*}
So \Cref{ProdsInReturnsk=1} implies that there is $B\subseteq G$ such that $\overline{d}_F(B)\geq\mu(Y)^2$ and 
\begin{equation*}
B\times B=\psi(B)\cdot\phi(B)\subseteq R_\mu^T(Y).\qedhere
\end{equation*}
\end{proof}
\Cref{BxBinAA^-1} follows from \Cref{CartProdInRets} and \Cref{ReturnsInDifferences}.

\CartesianSquareAsymptoticDensityInN*
\begin{proof}
We adapt the proof of \Cref{CartProdInRets} with $G=\mathbb{Z}$ and $F_N=\{1,\dots,N\}$. The averages
\[
    \frac{1}{N}\sum_{n=1}^N
    1_{T_{(n,0)}Y\cap T_{(0,-n)}Y}(x)
    =\frac{1}{N}\sum_{n=1}^N
    1_Y\bigl(T_{(-n,0)}x\bigr)\,1_Y\bigl(T_{(0,n)}x\bigr)
\]
converge for $\mu$-almost every $x$ by \cite[Theorem 1.1]{GXZ26}, applied to the commuting transformations $T_{(-1,0)}$ and $T_{(0,1)}$. Thus, by \Cref{TheBigTechnicalTheoremDensityExists}, the same proof gives a set $B\subseteq\mathbb{N}$ whose asymptotic density exists and satisfies $d(B)\geq\mu(Y)^2$, with $B\times B\subseteq R_\mu^T(Y)$. Applying \Cref{ReturnsInDifferences} to $A\subseteq\mathbb{Z}^2$ gives such a system with $\mu(Y)=d^*(A)$ and $R_\mu^T(Y)\subseteq A-A$, proving the result.
\end{proof}

\PolynomialCartesianProductAsymptoticDensityInN*
\begin{proof}
As in the proof of \Cref{CartesianSquareAsymptoticDensityInN}, use \Cref{ReturnsInDifferences} to choose a $\mathbb{Z}^2$-system and $Y$ with $\mu(Y)=d^*(A)>0$ and $R_\mu^T(Y)\subseteq A-A$. Set $\phi(n)=(p_1(n),0)$ and $\psi(n)=(0,p_2(n))$. The averages
\[
    \frac{1}{N}\sum_{n=1}^N
    1_{T_{(p_1(n),0)}Y\cap T_{(0,-p_2(n))}Y}(x)
    =\frac{1}{N}\sum_{n=1}^N
    1_Y\bigl(T_{(-p_1(n),0)}x\bigr)\,1_Y\bigl(T_{(0,p_2(n))}x\bigr)
\]
converge almost everywhere to a function with positive integral by \cite[Theorems 1.3 and 1.4]{KMPWW24}.
Applying \Cref{TheBigTechnicalTheoremDensityExists} with $F_N=\{1,\dots,N\}$ therefore gives $B\subseteq\mathbb{N}$ with positive asymptotic density and
\[
    p_1(B)\times p_2(B)=\psi(B)+\phi(B)\subseteq R_\mu^T(Y)\subseteq A-A.\qedhere
\]
\end{proof}

\PolynomialCartesianProductInFields*
\begin{proof}
Apply \Cref{ReturnsInDifferences} and then \Cref{ProdsInReturnsk=1} with $\phi(t)=(p_1(t),0)$ and $\psi(t)=(0,p_2(t))$. The required bound $\lambda\geq\mu(Y)^2$ follows from \cite[Theorem 3.10]{BLM05}, applied to $t\mapsto(p_1(t),p_2(t))$, and the projection estimate \eqref{InnerProdWithProjection}.
\end{proof}

\begin{proof}[Proof of \Cref{FiniteUnionsPolynomialCartesianProducts}]
Choose a system and $Y$ as above, with acting group $\mathbb{Z}^2$ or $K^2$, respectively. Apply \Cref{TheBigTechnicalTheorem} with $k=r$, $\phi_i(t)=(p_i(t),0)$ and $\psi_i(t)=(0,q_i(t))$. In the integer case, \cite[Theorems 1.3 and 1.4]{KMPWW24} give almost-everywhere convergence of the relevant averages to a function with positive integral. Thus \Cref{TheBigTechnicalTheoremDensityExists} gives $B$ with positive asymptotic density and the required inclusions.

For the quantitative field assertion, instead use the $r$-fold product system on $X^r$ and the set $Y^r$, with
\[
    \phi(t)=\bigl((p_i(t),0)\bigr)_{i=1}^r,
    \qquad
    \psi(t)=\bigl((0,q_i(t))\bigr)_{i=1}^r.
\]
As above, \cite[Theorem 3.10]{BLM05} and \eqref{InnerProdWithProjection}, applied to the polynomial map $\phi+\psi:K\to(K^2)^r$, bound the corresponding correlation average below by $\mu^{\otimes r}(Y^r)^2=\mu(Y)^{2r}$. By \Cref{ProdsInReturnsk=1}, there is $B\subseteq K$ of upper density at least $\mu(Y)^{2r}$ such that $\phi(B)+\psi(B)$ lies in the set of returns of $Y^r$. Projecting onto each coordinate gives $p_i(B)\times q_i(B)\subseteq R_\mu^T(Y)\subseteq A-A$ for all $i$.
\end{proof}

\BxBinAAminusonetwosided*

\begin{proof}
Let $A_1=A_l$ and $A_2=A_r^{-1}$, so that $d_l^*(A_2)=d_r^*(A_r)$ and $A_r^{-1}A_r=A_2A_2^{-1}$. By \Cref{ReturnsInDifferences}, there are m.p.s. $(X_i,\mathcal{B}_i,\mu_i,(T_{i,g})_{g\in G})$ and $Y_i\in\mathcal{B}_i$, for $i=1,2$, such that $\mu_i(Y_i)=d_l^*(A_i)$ and $R_{\mu_i}^{T_i}(Y_i)\subseteq A_iA_i^{-1}$. We now consider the product m.p.s.
\begin{equation*}
(X,\mathcal{B},\mu,(T_g)_{g\in G})=(X_1\times X_2,\mathcal{B}_1\otimes\mathcal{B}_2,\mu_1\otimes\mu_2,(T_{1,g}\otimes T_{2,g})_{g\in G}),
\end{equation*}
and $Y=Y_1\times Y_2$. By \Cref{CartProdInRets} there is $B\subseteq G$ such that
\begin{align*}
B\times B&\subseteq R_\mu^T(Y)=R_{\mu_1}^{T_1}(Y_1)\times R_{\mu_2}^{T_2}(Y_2)\subseteq A_1A_1^{-1}\cap A_2A_2^{-1}
=A_lA_l^{-1}\cap A_r^{-1}A_r\\
\overline{d}_F(B)&\geq\mu(Y)^2=d_l^*(A_l)^2d_r^*(A_r)^2.\qedhere
\end{align*}
\end{proof}

\RestatableA*
\begin{proof}
By \Cref{ReturnsInDifferences} there is a m.p.s. $(X,\mathcal{B},\mu,(T_g)_{g\in G})$ and $Y\in\mathcal{B}$ such that $\mu(Y)=d_l^*(A)$ and $R_\mu^T(Y)\subseteq AA^{-1}$. As $G$ is abelian, $(T_{2g})_{g\in G}$ is an action of $G$, so by \Cref{LinearAvgsAreBig},
\begin{equation*}
\overlim\limits_{N\to\infty}\frac{1}{|F_N|}\sum_{g\in F_N}\mu\left(T_{g}^2Y\cap Y\right)\geq\mu(Y)^2.
\end{equation*}
The result then follows from \Cref{ProdsInRetsMeasurable}.
\end{proof}
\Cref{ProdsInRetsMeasurable} yields the inclusion \(BB\subseteq R_\mu^T(Y)\), for some set $B\subseteq G$ with $d_{(F_N)}(B)>0$. In \Cref{hBBinMeasurableReturns} below, we prove a similar result which has a somewhat weaker conclusion (we obtain $h^2BB\subseteq R_\mu^T(Y)$ for some $h\in G$) but much broader generality.

The proof relies on the following lemma.
\begin{lemma}
\label{SomeBigInnerProds}
Let $a\in(0,1)$, $\varepsilon>0$, $(X,\mathcal{B},\mu)$ a probability space and $f_N:X\to[0,1]$ measurable functions for all $N\in\mathbb{N}$, such that $\int f_Nd\mu\geq a$. Then there exists $M\in\mathbb{N}$ such that $\overlim_{N\to\infty}\langle f_M,f_N\rangle\geq a^2-\varepsilon$.
\end{lemma}

\begin{proof}
Proof by contradiction. Suppose that for all fixed $k\in\mathbb{N}$ there is $N_k$ such that, for all $N>N_k$, $\langle f_k,f_N\rangle<a^2-\varepsilon$. Then we can construct by recursion a subsequence $(g_n)_{n\in\mathbb{N}}$ of $(f_N)$ such that $\langle g_n,g_m\rangle< a^2-\varepsilon$ if $n\neq m$. This leads to contradiction, as it would imply, for arbitrarily big $K\in\mathbb{N}$, that
\begin{equation*}
\left(Ka\right)^2=\left\langle1,\sum_{k=1}^Kg_k\right\rangle^2\leq
\left\|\sum_{k=1}^Kg_k\right\|^2
=\sum_{k=1}^K\|g_k\|^2+2\sum_{1\leq k<l\leq K}\langle g_k,g_l\rangle\leq K+K^2(a^2-\varepsilon),
\end{equation*}
so $K^2\varepsilon\leq K$, which must be false for big enough $K$.
\end{proof}

\begin{lemma}
Let $G$ be a group, $F=(F_N)$ a sequence of finite, nonempty subsets of $G$, $(X,\mathcal{B},\mu,(T_g)_{g\in G})$ a m.p.s., $Y\in\mathcal{B}$ and $\varepsilon>0$. Then there exists $h\in G$ such that
\begin{equation}
\label{5terdfoireodf9}
\overlim_{N\to\infty}\frac{1}{|F_N|}\sum_{g\in F_N}\mu\left(T_{h}^2Y\cap T_g^2Y\right)\geq\mu(Y)^2-\varepsilon.
\end{equation}
\end{lemma}

\begin{proof}
We let $Y_g=T_{g}^2Y$, and for each $N\in\mathbb{N}$ we define 
\begin{equation*}
f_N=
\frac{1}{|F_N|}
\sum_{g\in F_N}1_{T_{g}^2Y}\in L^\infty(X).
\end{equation*}
By \Cref{SomeBigInnerProds} there is $M\in\mathbb{N}$ such that $\overlim_{N\to\infty}\langle f_{M},f_N\rangle\geq\mu(Y)^2-\varepsilon$, which implies that $\overlim_{N\to\infty}\langle 1_{T_{h}^2Y},f_N\rangle\geq\mu(Y)^2-\varepsilon$ for some $h\in F_M$, as we wanted.
\end{proof}

\begin{prop}
\label{hBBinMeasurableReturns}
Let $G$ be a group, $F=(F_N)$ a sequence of finite, nonempty subsets of $G$, $(X,\mathcal{B},\mu,(T_g)_{g\in G})$ a m.p.s., $Y\in\mathcal{B}$ and $\varepsilon>0$. Then there exists $h\in G$ and $B\subseteq G$ such that $\overline{d}_F(B)\geq\mu(Y)^2-\varepsilon$ and $BB\subseteq h^2R_\mu^T(Y)$.
\end{prop}

\begin{proof}
Let $h\in G$ satisfy \Cref{5terdfoireodf9}, so that 
\begin{align*}
\mu(Y)^2-\varepsilon
&\leq\overlim_{N\to\infty}\frac{1}{|F_N|}\sum_{g\in F_N}\mu\left(T_g^{-1}T_{h}^2Y\cap T_gY\right).
\end{align*}
So taking $G=H$, $\phi(g)=g$ and $\psi(g)=h^{-2}g$ in \Cref{ProdsInReturnsk=1}, we obtain $B\subseteq G$ such that $\overline{d}_F(B)\geq\mu(Y)^2-\varepsilon$ and $h^{-2}BB=\psi(B)\cdot\phi(B)\subseteq R_\mu^T(Y)$.
\end{proof}

\begin{cor}
\label{BBhinMeasurableReturns}
Let $G$ be a group, $F=(F_N)$ a sequence of finite, nonempty subsets of $G$, $(X,\mathcal{B},\mu,(T_g)_{g\in G})$ a m.p.s., $Y\in\mathcal{B}$ and $\varepsilon>0$. Then there exists $h\in G$ and $B\subseteq G$ such that $\overline{d}_F(B)\geq\mu(Y)^2-\varepsilon$ and $BBh^2\subseteq R_\mu^T(Y)$.
\end{cor}
\begin{proof}
By \Cref{hBBinMeasurableReturns} applied to the sequence $F^{-1}:=(F_N^{-1})_{N\in\mathbb{N}}$, we obtain $a\in G$ and a set $C\subseteq G$ such that $\overline{d}_{F^{-1}}(C)\geq\mu(Y)^2-\varepsilon$ and $a^2CC\subseteq R_\mu^T(Y)$. Thus, letting $h=a^{-1}$, $B=C^{-1}$, 
\begin{equation*}
BBh^2=(a^2CC)^{-1}\subseteq R_\mu^T(Y)^{-1}=R_\mu^T(Y).\qedhere
\end{equation*}
\end{proof}

Of course, we have a more general version of \Cref{hBBinMeasurableReturns} in the style of \Cref{ProdsInReturnsk=1},
which follows from the same proof:

\begin{prop}
\label{hBBinMeasurableReturnsgeneral}
Let $G$ be a set, $H$ a group, $F=(F_N)$ a sequence of finite, nonempty subsets of $G$, $(X,\mathcal{B},\mu,(T_h)_{h\in H})$ a m.p.s., $Y\in\mathcal{B}$ and $\varepsilon>0$. Then there exists $g_0\in G$ and $B\subseteq G$ such that $\overline{d}_F(B)\geq\mu(Y)^2-\varepsilon$ and $\psi(B)\phi(B)\subseteq \psi(g_0)\phi(g_0)R_\mu^T(Y)$.\qed
\end{prop}

\Cref{BBhinAA-1} follows from \Cref{hBBinMeasurableReturns}, \Cref{BBhinMeasurableReturns} and \Cref{ReturnsInDifferences}.
Similarly, \Cref{BkBkinAA-1} follows from the following ergodic version of itself:
\begin{prop}
\label{BkBkinAA-1ergodic}
Let $G$ be a countable amenable group, $(X,\mathcal{B},\mu,(T_g)_{g\in G})$ a m.p.s., $Y\in\mathcal{B}$, $a=\mu(Y)$ and $(F_N)_{N\in\mathbb{N}}$ a sequence of finite, nonempty subsets of $G$. Then there exist $B\subseteq G$ and $k\in\mathbb{N}$, $k\leq\frac{5}{a}$, such that $\overline{d}_F(B)\geq\frac{a^2}{4}$ and $B^kB^k:=\{b^k(b')^k;b,b'\in B\}\subseteq R_\mu^T(Y)$. 
\end{prop}

To prove \Cref{BkBkinAA-1ergodic} we use an optimal lower bound for ergodic averages obtained by Leibman:
\begin{lemma}[{\cite[Theorem 1.2]{Le02}}]
\label{LeibmannErgodicAvgsInfBound}
For any m.p.s. $(X,\mathcal{B},\mu,T)$, $Y\in\mathcal{B}$ such that $\mu(Y)=:a\in[0,1]$ and $N\in\mathbb{N}$, we have
\begin{equation*}
\frac{1}{N}\sum_{n=0}^{N-1}\mu(Y\cap T^{-n}Y)\geq\sqrt{a^2+(1-a)^2}+a-1\geq\frac{a^2}{2}.
\end{equation*}
\end{lemma}

\begin{proof}[Proof of \Cref{BkBkinAA-1ergodic}]
We may assume $a>0$. It follows from \Cref{LeibmannErgodicAvgsInfBound} that for all $g\in G$ we have, letting $K=\left\lfloor\frac{5}{a}\right\rfloor\geq\frac{4}{a}$,
\begin{equation*}
\frac{1}{K}\sum_{k=1}^K\mu(Y\cap T_g^{-2k}Y)\geq\frac{1}{K}\sum_{k=0}^{K-1}\mu(Y\cap T_g^{-2k}Y)-\frac{a}{K}\geq\frac{a^2}{2}-\frac{a}{K}\geq\frac{a^2}{4}.
\end{equation*}
So
\begin{align*}
\frac{1}{K}\sum_{k=1}^K\overlim_{N}\frac{1}{|F_N|}\sum_{g\in F_N}\mu(Y\cap T_g^{-2k}Y)
&\geq
\overlim_{N}\frac{1}{|F_N|}\sum_{g\in F_N}\frac{1}{K}\sum_{k=1}^K\mu(Y\cap T_g^{-2k}Y)\\
&\geq\overlim_{N}\frac{1}{|F_N|}\sum_{g\in F_N}\frac{a^2}{4}=\frac{a^2}{4}.
\end{align*}
Thus, there is $k\leq K$, which we fix from now, such that 
\begin{equation*}
\frac{a^2}{4}\leq\overlim_{N}\frac{1}{|F_N|}\sum_{g\in F_N}\mu(Y\cap T_g^{-2k}Y)=\overlim_{N}\frac{1}{|F_N|}\sum_{g\in F_N}\mu(T_{g^k}Y\cap T_{g^{k}}^{-1}Y).
\end{equation*}
Applying \Cref{ProdsInReturnsk=1} with $H=G$, $\phi(g)=\psi(g)=g^k$ finishes the proof.
\end{proof}

\RestatableB*
\begin{proof}
As $R$ is a set of measurable recurrence, by uniformity of recurrence (see \cite[Theorem 2.1]{Fo} or \cite[Proposition 1.3]{BH}) for all $\varepsilon>0$ there exist $\delta>0$ and finite $R_0\subseteq R$ such that, for all m.p.s. $(X,\mathcal{B},\mu,(T^n)_{n\in\mathbb{Z}})$ and all $Y\in\mathcal{B}$ with $\mu(Y)\geq\varepsilon$, there exists $r\in R_0$ with $\mu(Y\cap T^rY)\geq\delta$.

By \Cref{Returns=Differences} there is a m.p.s. $(X,\mathcal{B},\mu,(T^n)_{n\in\mathbb{Z}})$ and $Y\in\mathcal{B}$ such that $\lambda:=\mu(Y)=d^*(A)>0$ and $\{n\in\mathbb{Z};\mu(Y\cap T^nY)>0\}\subseteq A-A$. Let $\delta,R_0$ be as above for $\varepsilon=\lambda$. Then, letting $C_N=C\cap\{1,\dots,N\}$ we have 
\begin{align*}
\sum_{r\in R_0}\overlim_{N}\frac{1}{|C_N|}\sum_{c\in C_N}\mu(Y\cap T^{-2rc}Y)
&\geq
\overlim_{N}\frac{1}{|C_N|}\sum_{c\in C_N}\sum_{r\in R_0}\mu(Y\cap T^{-2rc}Y)\\
&\geq\lim_{N}\frac{1}{|C_N|}\sum_{c\in C_N}\delta=\delta.
\end{align*}
Thus, there is $r\in R_0$, fixed from now, such that 
\begin{equation*}
\frac{\delta}{|R_0|}
\leq
\overlim_{N\to\infty}\sum_{c\in C_N}\mu(Y\cap T^{-2rc}Y)
=
\overlim_{N\to\infty}\sum_{c\in C_N}\mu(T^{rc}Y\cap T^{-rc}Y).
\end{equation*}
Applying \Cref{ProdsInReturnsk=1} with $H=G=\mathbb{Z}$, $\phi(n)=\psi(n)=rn$ completes the proof.
\end{proof}

There are other ergodic-theoretic results which produce corollaries of \Cref{ProdsInReturnsk=1}.
For example, using arguments of Arnold and Krylov, Guivarc'h \cite{Gui} proved the following. Let $\mathbf{F}_2$ be the group freely generated by $S=\{a,b\}$ and let $R_{n}\subseteq \mathbf{F}_2$ be the set of reduced words of length $n$ in the alphabet $\{a,b,a^{-1},b^{-1}\}$.\\[-5pt]
\begin{theorem}[Cf. {\cite[Lemma 10.2]{Tem}}]
\label{GuivarchThmFreeGroups}
For any unitary action $(U_g)_{g\in G}$ of $\mathbf{F}_2$ on a Hilbert space $\mathcal{H}$, let $\mathcal{H}_1=\{h\in \mathcal{H};U_ah=U_bh=-h\}$, $\mathcal{H}_0=\mathcal{H}_1^\perp\subseteq \mathcal{H}$ and $I=\{h\in \mathcal{H};U_ah=U_bh=h\}\subseteq \mathcal{H}_0$. For $N\in\mathbb{N}$ we denote $\widehat{U}_N=\frac{1}{|R_{N}|}\sum_{g\in R_{N}}U_g$. Then for all $h\in \mathcal{H}_0$,
\begin{equation*}
\lim_{N\to\infty}\widehat{U}_N(h)=\pi^{\mathcal{H}_0}_I(h),
\end{equation*}
where $\pi^{\mathcal{H}_0}_I:\mathcal{H}_0\to I$ is orthogonal projection.
\end{theorem}

\begin{cor}
\label{BigAvgsSphereFreeGroup}
For all m.p.s. $(X,\mathcal{B},\mu,(T_g)_{g\in \mathbf{F}_2})$ and all $Y\in\mathcal{B}$, we have
\begin{equation}
\label{EqBigAvgsSphereFreeGroup}
\lim_{N\to\infty}\frac{1}{|R_{2N}|}\sum_{g\in R_{2N}}\mu(Y\cap T_gY)\geq\mu(Y)^2.
\end{equation}
\end{cor}

\begin{proof}
Let $(U_g)_{g\in G}$ be the unitary action on $\mathcal{H}:=L^2(X)$ given by $U_gf(x)=f(T_{g^{-1}}x)$. Using the notation of \Cref{GuivarchThmFreeGroups}, for all $h\in\mathcal{H}$ we have $\lim_{N\to\infty}\widehat{U}_{2N}h=\pi^{\mathcal{H}}_I(h)$. The projection $\pi^{\mathcal{H}}_I$ satisfies $\pi^{\mathcal{H}}_I(1_X)=1_X$, so we can conclude the result via the same reasoning in \Cref{InnerProdWithProjection,VonNeumannAmenableL2,LinearAvgsAreBig}.
\end{proof}

Note that in \Cref{EqBigAvgsSphereFreeGroup} does not hold if $R_{2N}$ replaced by $R_{2N+1}$. For example, consider the m.p.s. given by $\mathbf{F}_2$ acting on $X=\mathbb{Z}_2$ and $T_ax=T_bx=x+1$. For $Y=\{0\}$, we have $\mu(Y)=\frac{1}{2}$ but $\mu(Y\cap T_gY)=0$ for all $g\in R_{2N+1}, N\in\mathbb{N}$. 

We now give an analog of \Cref{BigAvgsSphereFreeGroup} for the balls $B_N:=\bigcup_{n\leq N}R_n$ instead of the spheres $R_N$ in $\mathbf{F}_2$; balls are a closer analog to the F{\o}lner sequences we have been using in many of our results in this paper, as balls in finitely generated abelian groups form a F{\o}lner sequence.

\RestatableC*
\begin{proof}
Let $(U_g)_{g\in G}$ be the unitary action of $\mathbf{F}_2$ on $L^2(X)$ given by $U_gf(x)=f(T_g^{-1}x)$, and let $\widehat{U}_N,\mathcal{H}_0,\mathcal{H}_1,I,\pi^{\mathcal{H}_0}_I$ be as in \Cref{GuivarchThmFreeGroups}. It follows from \Cref{GuivarchThmFreeGroups} that, for all $h\in \mathcal{H}$, if we let $h=h_0+h_1$ with $h_i\in \mathcal{H}_i$, then $\lim_{N\to\infty}\widehat{U}_{2N}h=\pi^{\mathcal{H}_0}_I(h_0)+h_1$ and $\lim_{N\to\infty}\widehat{U}_{2N+1}h=\pi^{\mathcal{H}_0}_I(h_0)-h_1$. Note that, as $|R_{n+1}|=3|R_n|$ for $n\geq1$, we have $\lim_{N\to\infty}\sum_{n=1}^N
\frac{|R_{2n-1}|}{|B_{2N}|}=\frac{1}{4}$. So
\begin{align*}
\lim_{N\to\infty}\frac{1}{|B_{2N}|}\sum_{g\in B_{2N}}U_g(h)
&=
\lim_{N\to\infty}
\sum_{n=1}^N
\left(\frac{|R_{2n}|}{|B_{2N}|}\widehat{U}_{2n}+\frac{|R_{2n-1}|}{|B_{2N}|}\widehat{U}_{2n-1}\right)(h)\\
&=
\lim_{N\to\infty}
\sum_{n=1}^N
\frac{|R_{2n-1}|}{|B_{2N}|}\left(3\widehat{U}_{2n}(h)+\widehat{U}_{2n-1}(h)\right)\\
&=
\frac{1}{4}\left(3\left(\pi^{\mathcal{H}_0}_I(h_0)+h_1\right)+\left(\pi^{\mathcal{H}_0}_I(h_0)-h_1\right)\right)\\
&=\pi^{\mathcal{H}_0}_I(h_0)+\frac{h_1}{2}.
\end{align*}
Thus, letting $h=1_Y$, and using the same argument as in the proof of \Cref{InnerProdWithProjection}, we obtain
\begin{align*}
&\lim_{N\to\infty}\frac{1}{|B_{2N}|}\sum_{g\in B_{2N}}\mu(Y\cap T_gY)=\left\langle h,\pi^{\mathcal{H}_0}_I(h_0)+\frac{h_1}{2}\right\rangle
=\left\langle h_0,\pi^{\mathcal{H}_0}_I(h_0)\right\rangle+\left\langle h_1,\frac{h_1}{2}\right\rangle\\
&\geq\left\langle h_0,\pi^{\mathcal{H}_0}_I(h_0)\right\rangle
=\left\|\pi^{\mathcal{H}_0}_I(h_0)\right\|^2=\left\|\pi^{\mathcal{H}}_I(1_Y)\right\|^2\geq\mu(Y)^2.\qedhere
\end{align*}
\end{proof}

\begin{remark}
The analog of \Cref{GuivarchThmFreeGroups} for the free group $\mathbf{F}_r$ of rank $r>2$ holds with the same proof, using the expression $\widehat{U}_{n+1}=\frac{2r}{2r-1}\widehat{U}_1\widehat{U}_n-\frac{1}{2r-1}\widehat{U}_{n-1}$ instead of $\widehat{U}_{n+1}=\frac{4}{3}\widehat{U}_1\widehat{U}_n-\frac{1}{3}\widehat{U}_{n-1}$. So, \Cref{BigAvgsBallFreeGroup} and \Cref{CartProdInRetsFree} extend to $\mathbf{F}_r$ for $r>2$.
\end{remark}

\section{Density along a sequence of probability measures}
\label{SecGeneralDensities}
In this section, we prove \Cref{TheBigTechnicalTheoremWithProbMeasures}, a general version of \Cref{TheBigTechnicalTheorem} in which upper density along a sequence $(F_N)$ of finite subsets is replaced by upper density along a sequence $(\mu_N)$ of probability measures on $G$. We then use \Cref{TheBigTechnicalTheoremWithProbMeasures} to deduce \Cref{B+BinA-AinZLogDensity} and prove \Cref{ThmConvsAreNiceAndHaveBxBinAA-1}, which is a version of \Cref{CartProdInRetsFree} for general countable groups.

\begin{definition}
\label{5terdf9ordgfl}
Let $(\nu_N)_{N\in\mathbb{N}}$ be a sequence of probability measures on a measurable space $(X,\mathcal{B})$. The \textit{upper density} of a set $A\in\mathcal{B}$ along $(\nu_N)$ is defined as 
\begin{equation*}
\overline{d}_{(\nu_N)}(A)=\overlim_{N\to\infty}\nu_N(A).
\end{equation*}
\end{definition}

Note that the previously defined upper density along a sequence of finite nonempty sets $(F_N)$, $\overline{d}_{(F_N)}$, is a special case of \Cref{5terdf9ordgfl}. Indeed, letting $\nu_N(A)=\frac{|F_N\cap A|}{|F_N|}$ for $A\subseteq G$, we have $\overline{d}_{(F_N)}=\overline{d}_{(\nu_N)}$.

\begin{theorem}
\label{TheBigTechnicalTheoremWithProbMeasures}
Let $G$ be a countable set, $H$ a group, $\phi_1,\dots,\phi_k,\psi_1,\dots,\psi_k:G\to H$ functions, $(X,\mathcal{B},\mu,(T_h)_{h\in H})$ a m.p.s., $Y\in\mathcal{B}$ and $(\nu_N)_{N\in\mathbb{N}}$ a sequence of probability measures on $(G,\mathcal{P}(G))$. Let
\begin{equation*}
\lambda=\overlim\limits_{N\to\infty}\sum_{g\in G}\nu_N(g)\mu\left(\bigcap_{i=1}^kT_{\phi_i(g)}Y\cap T_{\psi_i(g)}^{-1}Y\right).
\end{equation*}
Then there exists $B\subseteq G$ such that $\overline{d}_{(\nu_N)}(B)\geq \lambda$ and
\begin{equation*}
\bigcup_{i,j=1}^k\psi_i(B)\cdot \phi_j(B)\subseteq R_\mu^T(Y).
\end{equation*}
\end{theorem}

\begin{proof}
One way of proving \Cref{TheBigTechnicalTheoremWithProbMeasures} is to establish first (using essentially the same argument) a version of \Cref{IntLemma} for a sequence $(\nu_N)$ of probability measures, and then to proceed along the lines of the proof of \Cref{TheBigTechnicalTheorem}.

Alternatively, one can directly deduce \Cref{TheBigTechnicalTheoremWithProbMeasures} from \Cref{TheBigTechnicalTheorem} by utilizing the fact that for every countable set $G$ and every sequence of probability measures $(\nu_N)$ on $(G,\mathcal{P}(G))$, there exists a countable set $G_2$, a sequence $(F_N)$ of finite sets in $G_2$ and a map $f:G_2\to G$ such that
\begin{equation*}
\lim_{N\to\infty}\sum_{g\in G}
\left|\nu_N(g)-\nu_N'(g)
\right|=0,
\end{equation*}
where $\nu_N'$ is the probability measure on $G$ given by $\nu_N'(A)=\frac{|\{g\in F_N;f(g)\in A\}|}{|F_N|}$.
\end{proof}
We now restate, for a future reference, some of the corollaries of \Cref{TheBigTechnicalTheorem} in this more general form (their proofs are the same as in \Cref{SecErgToComb}).

\begin{cor}
Let $G$ be a countable set, $H$ a group, $\phi,\psi:G\to H$ functions, $(X,\mathcal{B},\mu,(T_h)_{h\in H})$ a m.p.s., $Y\in\mathcal{B}$ and $(\nu_N)_{N\in\mathbb{N}}$ a sequence of probability measures on $(G,\mathcal{P}(G))$. Let
\begin{equation*}
\lambda=\overlim\limits_{N\to\infty}\sum_{g\in G}\nu_N(g)\mu\left(T_{\phi(g)}Y\cap T_{\psi(g)}^{-1}Y\right).
\end{equation*}
Then there exists $B\subseteq G$ such that $\overline{d}_{(\nu_N)}(B)\geq \lambda$ and
\begin{equation*}
\psi(B)\cdot\phi(B)\subseteq R_\mu^T(Y).
\end{equation*}
\end{cor}

\begin{cor}
\label{CartProdsInReturnsk=1Reiter}
Let $G$ be a group,
$(X,\mathcal{B},\mu,(T_{g,h})_{(g,h)\in G\times G})$ a m.p.s. and $Y\in\mathcal{B}$. Then, for any sequence $(\nu_N)$ of probability measures on $(G,\mathcal{P}(G))$, letting
\begin{equation*}
\lambda=
\overlim\limits_{N\to\infty}\sum_{g\in G}\nu_N(g)\mu\left(T_{(g,g)}Y\cap Y\right),
\end{equation*}
there exists $B\subseteq G$ such that $\overline{d}_{(\nu_N)}(B)\geq \lambda$ and
\begin{equation*}
B\times B\subseteq R_\mu^{T}(Y)=\left\{(g,h)\in G\times G;\mu(T_{(g,h)}Y\cap Y)>0\right\}.
\end{equation*}
\end{cor}

The following definition may be viewed as an analogue of asymptotic invariance of F{\o}lner sets in the setting of sequences of probability measures on a countable group.
\begin{definition}
A sequence of probability measures $(\mu_N)_{N\in\mathbb{N}}$ on a group $G$ is a \textit{left Reiter sequence} if, for all $g\in G$, we have
\begin{equation}
\label{rewfdsc9zoilk}
\lim_{N\to\infty}\sum_{h\in G}|\mu_N(h)-\mu_N(gh)|=0.
\end{equation}
\end{definition}
Similarly one defines a right Reiter sequence by switching $\mu_N(gh)$ to $\mu_N(hg)$ in \Cref{rewfdsc9zoilk}.
The following results are proved similarly to \Cref{BxBinAA^-1,B+BinA-AAbelian}, using a version of von Neumann's ergodic theorem (see \Cref{VonNeumannAmenableL2}) for Reiter sequences, which has the same proof as with F{\o}lner sequences.
\begin{theorem}
\label{BxBinAA^-1Reiter}
Let $G$ be a countable amenable group, let $A\subseteq G\times G$ satisfy $d^*_l(A)>0$. Then for any left or right Reiter sequence $(\nu_N)$ on $G$ there exists $B$ such that $\overline{d}_{(\nu_N)}(B)\geq d^*_l(A)^2$ and $B\times B\subseteq AA^{-1}$.
\end{theorem}

\begin{theorem}
\label{B+BinA-AAbelianReiter}
Let $(G,+)$ be a countable abelian group and let $A\subseteq G$. Then for any Reiter sequence $(\nu_N)$ on $G$ there is some $B\subseteq G$ such that $\overline{d}_{(\nu_N)}(B)\geq d^*(A)^2$ and $B+B\subseteq A-A$. 
\end{theorem}

\RestatableD*

\begin{proof}
Note that logarithmic density is $\overline{d}_{(\nu_N)}$, where $\nu_N$ are the probability measures on $\mathbb{N}$ given by $\nu_N(n)=\frac{1}{\sum_{j=1}^N\frac{1}{j}}\cdot\frac{1}{n}$ for all $n=1,\dots,N$. Moreover, $(\nu_N)$ is a Reiter sequence, so we are done by \Cref{BxBinAA^-1Reiter}.
\end{proof}

Even if a countable group $G$ is not amenable, we now give an ample family of sequences $(\nu_N)$ of probability measures on $G$ that satisfy a version of \Cref{BxBinAA^-1}. Before stating our result, we need some definitions.

\begin{definition}
Let $G$ be a group and $\nu_1,\nu_2$ probability measures on $(G,\mathcal{P}(G))$. 

The \textit{convolution} $\nu_1*\nu_2$ is the probability measure on $(G,\mathcal{P}(G))$ given by
\begin{equation*}
\nu_1*\nu_2(A)=\nu_1\otimes\nu_2(\{(g,h)\in G\times G;gh\in A\}).
\end{equation*}
Equivalently, $\nu_1*\nu_2$ is the pushforward of the product measure $\nu_1\otimes\nu_2$ by the group operation $G\times G\to G$. We denote $\nu^{n}=\nu*\cdots*\nu$ $n$ times.
\end{definition}
For a unitary action $(U_g)_{g\in G}$ of a group $G$ on a Hilbert space and a probability measure $\nu$ on $G$, denote $U_\nu=\sum_g\nu(g)U_g$. One can readily check that, if $\nu_1,\nu_2$ are probability measures on $G$, then $U_{\nu_1*\nu_2}=U_{\nu_1}\circ U_{\nu_2}$. This fact lies at the heart of several ergodic theorems for convolutions of measures. In the sequel we will utilize the following variant of von Neumann's ergodic theorem.
\begin{theorem}[{\cite[Theorem 3.1]{JRT}}]
\label{CesarosOfConvsAreNice}
If $\nu$ is a probability measure on a countable group $G$, and $(X,\mathcal{B},\mu,(T_g)_{g\in G})$ is a m.p.s., then for all $f\in L^2(X)$ we have
\begin{equation}
\label{refds9oieriodfi}
\lim_{N\to\infty}\frac{1}{N}\sum_{n=1}^N\left(\sum_{g\in G}\nu^n(g)T_gf\right)=\pi(f)
\end{equation}
in $L^2$ norm, where $\pi$ is the orthogonal projection to the subspace
\begin{equation*}
\{h\in L^2(X);T_gh=h\textup{ for all }g\in G\textup{ such that }\nu(g)>0\}. 
\end{equation*}
\end{theorem}

\begin{remark}
In the setup of \Cref{CesarosOfConvsAreNice}, let $S=\{g\in G;\nu(g)>0\}$. In \cite[Theorem 2]{Ose} it is proved that, if the subgroups of $G$ generated by $S$ and by $SS^{-1}$ are the same, then one has a stronger form of \Cref{refds9oieriodfi}:
\begin{equation*}
\lim_{N\to\infty}\sum_{g\in G}\nu^{N}(g)T_gf=\pi(f).
\end{equation*}
\end{remark}

Using the same reasoning as in the derivation of \Cref{LinearAvgsAreBig}, we deduce from \Cref{CesarosOfConvsAreNice} the following:
\begin{cor}
\label{BigAvgsAlongConvs}
Let $G$ be a countable group, $(X,\mathcal{B},\mu,(T_g)_{g\in G})$ a m.p.s. and $\nu$ a probability measure on $(G,\mathcal{P}(G))$. Then for all $Y\in\mathcal{B}$ we have
\begin{equation*}
\lim_{N\to\infty}
\frac{1}{N}
\sum_{n=1}^N\sum_{g\in G}\nu^{n}(g)\cdot\mu(Y\cap T_gY)\geq\mu(Y)^2.
\end{equation*}
\end{cor}

\Cref{CartProdsInReturnsk=1Reiter} and \Cref{BigAvgsAlongConvs} imply the following:
And we conclude this `$B\times B$ in returns' result for non amenable groups
\begin{theorem}
\label{ThmConvsAreNiceAndHaveBxBinAA-1}
Let $G$ be a countable group, and let $\nu$ be a probability measure on $(G,\mathcal{P}(G))$. Then, for any m.p.s. $(X,\mathcal{B},\mu,(T_{(g,h)})_{(g,h)\in G\times G})$ and any 
$Y\in\mathcal{B}$, there exists $B\subseteq G$ such that $\overline{d}_{\left(\nu^{N}\right)}(B)\geq\mu(Y)^2$ and 
\begin{equation*}
B\times B\subseteq\{(g,h)\in G\times G;\mu\left(Y\cap T_{(g,h)}Y\right)>0\}.
\end{equation*}
\end{theorem}

\begin{proof}
Note that by \Cref{BigAvgsAlongConvs}, for all $Y$ of positive measure in a m.p.s. $(X,\mathcal{B},\mu,(S_g)_{g\in G})$,
\begin{equation*}
\overlim_{N\to\infty}\sum_{g\in G}\nu^{N}(g)\mu(Y\cap T_gY)
\geq\lim_{N\to\infty}
\frac{1}{N}
\sum_{n=1}^N
\sum_{g\in G}\nu^{n}(g)
\int\mu(Y\cap T_gY)\geq\mu(Y)^2.
\end{equation*}
So we are done by \Cref{CartProdsInReturnsk=1Reiter}.
\end{proof}

\begin{remark}
\label{RemarkConvsAndBallsFreeProds}
\Cref{CartProdInRetsFree} can be seen as a consequence of \Cref{ThmConvsAreNiceAndHaveBxBinAA-1}. Indeed, let $\nu$ be the probability measure on $\mathbf{F}_2=\langle a,b\rangle$ which gives mass $\frac{1}{5}$ to each of the elements $1,a,b,a^{-1},b^{-1}$. Then a careful analysis shows that, for all $A\subseteq\mathbf{F}_2$, we have
\begin{equation*}
\overline{d}_{(B_N)}(A)\geq
\overline{d}_{(\nu^N)}(A).
\end{equation*}
We have chosen to keep a separate proof of \Cref{CartProdInRetsFree} because that way we can ensure the existence of the limit in \Cref{BigAvgsBallFreeGroup} (this does not follow from the proof of \Cref{ThmConvsAreNiceAndHaveBxBinAA-1}), which is a result of independent interest.
\end{remark}

%% file: S4-NilpotentGroups2.tex
\section{Polynomials, finitely generated nilpotent groups}
\label{SecNilpotent}
In this section, we prove \Cref{NilpPolyBB,PolyCesaroNilpotent,BigAvgsAlongSquaresFGNilp}, which were formulated in the introduction. In subsection \ref{SecPrelimsNilp} we review some pertinent preliminary material. \Cref{NilpPolyBB,PolyCesaroNilpotent} are proved in subsection \ref{Sec5.2}. Finally, the proof of \Cref{BigAvgsAlongSquaresFGNilp} is given in \ref{Sec5.3}.

\msubsection{Preliminary material}
\label{SecPrelimsNilp}

At the beginning of this subsection we review some basic facts about polynomial mappings between nilpotent groups. We then introduce some technical definitions and results which will be instrumental for the proof of \Cref{PolyCesaroNilpotent} in the next subsection. In doing so, we draw extensively on \cite{BMc,ZK14}, with some modifications.

\begin{definition}[Polynomial mappings between groups, {cf. \cite[Section 1.1]{LePoly}}]
\label{DefPolyLeib}
Let $G,F$ be groups. Given $d\in\mathbb{N}$ and a function $\varphi:G\to F$ and $h\in G$, we define the derivative $D_h\phi:G\to F$ by $D_h\varphi(g)=\varphi(g)^{-1}\varphi(gh)$. 
We say $\varphi:G\to F$ is a polynomial map of degree $\leq d$ if for all $h_1,\dots,h_{d+1}\in G$, $D_{h_1}\cdots D_{h_{d+1}}\varphi=1_F$.\footnote{In \cite{LePoly}, Leibman defines polynomials with respect to an arbitrary generating set $S$ of $G$; in \Cref{DefPolyLeib} we just take $S=G$. This does not affect which functions to nilpotent groups are polynomials, see \cite[Proposition 3.5]{LePoly}}

The degree of $\varphi$ is defined to be $-\infty$ if $\varphi(g)=1_F$ for all $g\in G$, and is otherwise the smallest $d\in\mathbb{N}$ such that $\varphi$ has degree $\leq d$.
\end{definition}

Homomorphisms $\varphi:G\to F$ are polynomials of degree $\leq 1$, and if $F$ is nilpotent, then the family of all polynomials $p:G\to F$ forms a group under pointwise product (see \cite[Theorems 3.2, 3.4]{LePoly} or \cite[Theorem 2.5]{ZK14}). 

\begin{example}
If $F$ is nilpotent, then the maps $p,q:F\to F;p(g)=g^{-1}$ and $q(g)=ag^7bg^{-2}c$, where $a,b,c\in F$, are polynomial maps. Moreover, using \cite[Theorems 3.4]{LePoly} one can check that the degree of $p,q$ is at most the nilpotency class of $F$, as they are in the group generated by the identity $g\mapsto g$ and constant polynomials.
\end{example}

The following result is formulated without proof in \cite[Section 1.8]{LePoly}; we include a proof for convenience of the reader.
\begin{prop}
\label{LeibPolysZn}
A map $p:\mathbb{Z}^n\to\mathbb{Z}^m$ is polynomial in the sense of \Cref{DefPolyLeib} if and only if it is an ordinary polynomial, i.e. it is of the form 
\begin{equation*}
p(x_1,\dots,x_n)=(p_1,\dots,p_n)(x_1,\dots,x_m),\textup{ where }p_i\in\mathbb{Q}[x_1,\dots,x_m],p_i(\mathbb{Z}^m)\subseteq\mathbb{Z}.\footnote{Note that $p_i$ need not be in $\mathbb{Z}[x_1,\dots,x_m]$ (consider for example $p:\mathbb{Z}\to\mathbb{Z}$; $p(n)=\frac{n(n-1)}{2}$).}
\end{equation*}
Moreover, the degree of $p$ in the sense of \Cref{DefPolyLeib} is the ordinary degree of $p$, i.e. the maximal degree of a monomial in $p$ with nonzero coefficient.
\end{prop}

\begin{proof}
We prove by induction on $d\in\mathbb{N}$ that the family $\mathcal{P}_{\leq d}$ of polynomials $\mathbb{Z}^n\to\mathbb{Z}^m$ of degree $\leq d$ in the sense of \Cref{DefPolyLeib} coincides with the class $\mathcal{P}_{\leq d}'$ of ordinary polynomials of ordinary degree $\leq d$. The inclusion $\mathcal{P}_{\leq d}'\subseteq\mathcal{P}_{\leq d}$ is easy to check, we focus on the other inclusion. For $d=0$, both $\mathcal{P}_{\leq d}'$ and $\mathcal{P}_{\leq d}$ are just the constant maps, so assume $p\in \mathcal{P}_{\leq d}$ has degree $>0$. Letting $e_1,\dots,e_n$ form the usual basis of $\mathbb{Z}^n$, the derivatives $p_i:=D_{e_i}p(x)=p(x+e_i)-p(x)$ are in $\mathcal{P}_{\leq d-1}$, so $p_i:\mathbb{Z}^n\to\mathbb{Z}^m$ are ordinary polynomials of degree $\leq d-1$ by induction hypothesis. Now note that, for all $x_1,\dots,x_n\in\mathbb{Z}$, by definition of $p_i$ we have
\begin{equation*}
p(x_1,\dots,x_n)=p(0,\dots,0)+\sum_{i_1=0}^{x_1-1}p_1(i_1,0,\dots,0)+\cdots+\sum_{i_n=0}^{x_n-1}p_n(x_1,\dots,x_{n-1},i_n).
\end{equation*}
Here, if $k$ is negative we define $\sum_{i=0}^kx_i=-\sum_{i=-k}^{-1}x_i$.
By Faulhaber's formula, for all $k<d$, the map $f:\mathbb{Z}\to\mathbb{Z};N\mapsto\sum_{j=1}^N j^k$ is a polynomial in $\mathbb{Q}[N]$ of degree $k+1$ with coefficients in $\frac{1}{k+1}\mathbb{Z}$. So, applying Faulhaber's formula to each of the monomials of $p_1,\dots,p_n$, we conclude that $p(x_1,\dots,x_n)$ is a polynomial in $\mathbb{Q}[x_1,\dots,x_n]$ of degree $\leq d$.
\end{proof}

Recall that any torsion free, finitely generated nilpotent group $G$ has a \emph{basis}, i.e. some elements $g_1,\dots,g_t$ ($t\geq0$) of $G$ such that the map $\alpha:\mathbb{Z}^t\to G;(a_1,\dots,a_d)\mapsto g_1^{a_1}\cdots g_d^{a_d}$ is a bijection. We call $\alpha$ the \textit{parameterization} associated to the basis $(g_1,\dots,g_t)$. The proposition below shows that polynomial mappings between finitely generated torsion-free nilpotent groups are essentially ordinary polynomials.

\begin{prop}[{\cite[3.12]{LePoly}}]
\label{PolysTFFGNil}
Let $G,F$ be torsion free, finitely generated nilpotent groups with parameterizations $\beta:\mathbb{Z}^s\to G$, $\alpha:\mathbb{Z}^t\to F$. Then a mapping $\varphi:G\to F$ is polynomial if and only if the mapping $\alpha^{-1}\circ\psi\circ\beta:\mathbb{Z}^{s}\to\mathbb{Z}^t$ is polynomial.
\end{prop}

A \textit{prefiltration} $G_\bullet$ is a sequence of nested groups 
\begin{equation*}
G_0\geq G_1\geq G_2\geq\cdots\textup{ such that }[G_i,G_j]\subseteq G_{i+j}\textup{ for all }i,j\in\mathbb{N}.
\end{equation*}
A filtration on a group $G$ is a prefiltration $G_\bullet$ such that $G_0=G_1=G$. A prefiltration is said to have length $d\in\mathbb{N}$ if $G_{d+1}$ is the trivial group but $G_d$ is not trivial, or $d=-\infty$ if $G_0$ is the trivial group; a group $G$ has a filtration of finite length iff $G$ is nilpotent, in which case the lower central series $G_0=G_1=G,G_{n+1}=[G,G_n]$ forms a filtration of $G$. If $G_\bullet$ is a prefiltration and $i\in\mathbb{N}$, the sequence $G_{\bullet+i}$ given by $(G_{\bullet+i})_t=G_{i+t}$ is a prefiltration. For \(d\in\mathbb N\), we will also use the following prefiltration:
\begin{equation}
\label{BigFiltration}
G_0 \ge 
\underbrace{G_1 \ge \cdots \ge G_1}_{d\text{ times}}
\ge
\underbrace{G_2 \ge \cdots \ge G_2}_{d\text{ times}}
\ge \cdots \ge
\underbrace{G_s \ge \cdots \ge G_s}_{d\text{ times}}
\ge\cdots
\end{equation}

Let $\mathcal{P}_f(\mathbb{N})$ be the family of finite subsets of $\mathbb{N}$, and $\mathcal{F}=\mathcal{P}_f(\mathbb{N})\setminus\{\varnothing\}$. We give $\mathcal{F}$ an order $<$, defined by $\alpha<\beta$ iff $\max(\alpha)<\min(\beta)$; we also define $\varnothing<\alpha$ and $\alpha<\varnothing$ for all $\alpha\in\mathcal{P}_f(\mathbb{N})$ (so $<$ is not quite an order in $\mathcal{P}_f(\mathbb{N})$).

For a prefiltration $G_\bullet$, we define a map $g:\mathcal{P}_f(\mathbb{N})\to G_0$ to be an \emph{IP-$G_\bullet$-polynomial} as follows, by recursion on the length $d$ of $G_\bullet$. 
If $d=-\infty$, then the only map $g:\mathcal{P}_f(\mathbb{N})\to G_0$ is defined to be an IP-$G_\bullet$-polynomial. If $d\geq0$, then we say $g:\mathcal{P}_f(\mathbb{N})\to G_0$ is a IP-$G_\bullet$-polynomial if for all $\beta\in\mathcal{P}_f(\mathbb{N})$ there exists a $G_{\bullet+1}$-IP polynomial, which we denote by $D_\alpha g$, such that for all $\alpha>\beta$,
\begin{equation}
\label{DerivExpression}
D_\beta g(\alpha)=g(\alpha)^{-1}g(\beta\cup\alpha)\textup{ for all }\alpha\in\mathcal{P}_f(\mathbb{N}), \alpha>\beta.
\end{equation}
\begin{remark}\label{Beta>AlphaRemark}
The definition of IP-polynomial in \cite{ZK14} is slightly different, in that in \Cref{DerivExpression} it uses the condition $\alpha\cap\beta=\varnothing$ instead of $\alpha>\beta$ (see \cite[Equation 2.16]{ZK14}). 
We stick to the condition $\alpha>\beta$, in line with the original definition of FVIP systems (in \cite[pp. 162 -- 163]{Mc}, McCutcheon defines $D_\beta g(\alpha)$ as a function with domain $\{\alpha\in\mathcal{F};\alpha>\beta\}$).
The definition of IP-polynomial in \cite{ZK14} is a priori stronger than the one given here, and in our arguments using IP-$G_\bullet$-polynomials, we only use the equation $D_\beta g(\alpha)=g(\alpha)^{-1}g(\beta\cup\alpha)$ when $\alpha>\beta$. Some proofs in \cite{ZK14}, e.g. see Proposition 2.22 or Lemma 4.20, seem to implicitly use our definition. More details about why the arguments in \cite{ZK14} work with the condition $\alpha>\beta$ are given below. 
\end{remark}

Note that $D_\beta g$ denotes any function $f:\mathcal{P}_f(\mathbb{N})\to G_0$ satisfying \Cref{DerivExpression}. This condition does not determine $f$ uniquely, since any change of the value $f(\{1\})$ does not affect \Cref{DerivExpression}.

 By \cite[Theorem 2.5]{ZK14}, the IP-$G_\bullet$-polynomials form a group under pointwise multiplication $(g\cdot f)(\alpha)=g(\alpha)\cdot f(\alpha)$, which we denote $P(\mathcal{P}_f(\mathbb{N}),G_\bullet)$. We denote by $\textup{VIP}(G_\bullet)=\{g\in P(\mathcal{P}_f(\mathbb{N}),G_\bullet);g(\varnothing)=1_{G_0}\}$ the set of polynomials vanishing at $0$.

\begin{example}
Let $G_\bullet$ be a filtration. It follows\footnote{In \Cref{PolyTimesIPIsFVIP} let $G=F$ and $p:G\to G;g\mapsto g^{-1}$; note that the first part of the proof does not require $G$ to be finitely generated.} from \Cref{PolyTimesIPIsFVIP} that, for any sequence $(g_n)$ of elements of $G_0$, the IP-system $g:\mathcal{P}(\mathbb{N})\to G_0$ given by 
\begin{equation*}
g(\alpha)=\prod_{i\in\alpha}^<g_i:=g_{i_1}\cdots g_{i_k}\textup{ (so $g(\varnothing)=1_G$)},
\end{equation*}
is in $\textup{VIP}(G_\bullet)$.
\end{example}

Let $F$ be a subgroup of $\textup{VIP}(G_\bullet)$. We say $F$ is a \emph{VIP group} if it is closed under conjugation by constant functions $\mathcal{F}\to G_0$, and closed under taking derivatives in the sense that for all $g\in F,\beta\in\mathcal{F}$ there is some $f\in F\cap\textup{VIP}(G_{\bullet+1})$, which we denote $\tilde{D}_\beta g$, such that 
\begin{equation}
\label{SymDerivExpression}
\tilde{D}_\beta g(\alpha)=g(\alpha)^{-1}g(\beta\cup\alpha)g(\beta)^{-1}\textup{ for all }\alpha\in\mathcal{P}_f(\mathbb{N}), \alpha>\beta.
\end{equation}
An \emph{FVIP group} is a finitely generated VIP group. An \emph{FVIP system} is an element of an FVIP group. In particular, any FVIP system is an IP-$G_\bullet$-polynomial.

Finally, we recall some basics about IP rings and IP limits. Given a strictly increasing chain $\alpha_1<\alpha_2<\cdots$ of elements of $\mathcal{F}$, the \emph{IP ring} generated by $(\alpha_n)_{n\in\mathbb{N}}$ is the family $\mathcal{F}((\alpha_n)_{n\in\mathbb{N}}):=\left\{\bigcup_{n\in\beta}\alpha_n;\beta\in\mathcal{F}\right\}\subseteq\mathcal{F}$. For any IP ring $\mathcal{F}'\subseteq\mathcal{F}$, any sequence $(x_\alpha)_{\alpha\in\mathcal{F}'}$ in a Hausdorff space $X$ and $x\in X$, we say that $x$ is the IP limit of $(x_\alpha)_{\alpha\in\mathcal{F}'}$ and write
\begin{equation*}
\IPlim_{\alpha\in\mathcal{F}'}x_\alpha=x
\end{equation*}
if for any neighborhood $U$ of $x$ there exists $\beta\in\mathcal{F}$ such that $x_\alpha\in U$ for all $\alpha\in\mathcal{F'}$ with $\alpha>\beta$. 

We will need the following result: 
\begin{theorem}
\label{SimplerZKThm32}
Let $G$ be a nilpotent group with a filtration $G_\bullet$ and a left action $(T_g)_{g\in G}$ on a probability space $(X,\mathcal{A},\mu)$. Let $S_0,\dots,S_t:\mathcal{F}\to G$ be arbitrary $G_\bullet$-FVIP systems and $A\in\mathcal{A}$ with $\mu(A)>0$. Then there exists an IP ring $\mathcal{F}'\subseteq F$ such that 
\begin{equation*}
\IPlim_{\alpha\in\mathcal{F}'}\mu\left(\bigcap_{i=0}^tT_{S_i(\alpha)}A\right)>0.
\end{equation*}

\end{theorem}

We now show that \Cref{SimplerZKThm32} follows from \cite[Theorem 5.32]{ZK14}, with the definition of IP-polynomials changed as in \Cref{DerivExpression}. For the definitions of $F^{\otimes m}$, $\mathcal{F}^{m}_<$, $\lim_{\vec{\alpha}\in\mathcal{F}^m_<}$ see \cite[Definition 2.19]{ZK14} and \cite[Page 13]{BMc}.

\begin{theorem}[{\cite[Theorem 5.32]{ZK14}}]
\label{ZKThm5.32}
Let $G$ be a nilpotent group and $F\leq\textup{VIP}(G_\bullet)$ an FVIP group. Consider a right measure preserving action of $G$ on a probability space $(X,\mathcal{A},\mu)$. Let $S_0,\dots,S_t\in F^{\otimes m}$ be arbitrary polynomial expressions and $A\in\mathcal{A}$ with $\mu(A)>0$. Then there exists an IP ring $\mathcal{F}'\subseteq\mathcal{F}$ such that 
\begin{equation*}
\IPlim_{\vec{\alpha}\in(\mathcal{F}')^m_<}\mu\left(\cap_{i=0}^tAS_i(\vec{\alpha})^{-1}\right)>0.
\end{equation*}
\end{theorem}

\begin{proof}[Proof of \Cref{SimplerZKThm32} from \Cref{ZKThm5.32}]
Define a right action of $G$ on $(X,\mathcal{B},\mu)$ given by, for $x\in X$ and $g\in G$, $xg:=T_g^{-1}x$. By \cite[Lemma 4.22]{ZK14}, $S_0,\dots,S_t$ are all contained in some fixed FVIP group $F$, with their inverses $(S_i^{-1}(\alpha))_{\alpha\in\mathcal{F}}$ also being in $F$. So \Cref{ZKThm5.32} with $m=1$ (so $F^{\otimes1}=F$) says that there exists an IP ring $\mathcal{F}'\subseteq\mathcal{F}$ such that 
\begin{equation*}
0<\IPlim_{\alpha\in\mathcal{F}'}\mu\left(\cap_{i=0}^tAS_i(\alpha)^{-1}\right)
=
\IPlim_{\alpha\in\mathcal{F}'}
\mu\left(\cap_{i=0}^tT_{S_i(\alpha)}A\right).\qedhere
\end{equation*}
\end{proof}

\paragraph{Checking the results in \cite{ZK14} with the condition $\alpha>\beta$.}
For the convenience of the reader, we briefly explain why the modification elucidated in \Cref{Beta>AlphaRemark} does not affect the arguments in \cite{ZK14} leading to the proof of \cite[Theorem 5.32]{ZK14}. Rather than checking every occurrence of the derivatives $D,\tilde{D}$ individually, we focus on the places in \cite{ZK14} where \Cref{DerivExpression,SymDerivExpression} are used in a potentially problematic way, and verify that they are only invoked for $\alpha>\beta$.

\label{CheckingZKChange}
\paragraph{\cite[Section 2]{ZK14}} The proofs of all results still work with our definition of IP polynomial; in fact, our definition seems to be used in the proof of \cite[Proposition 2.22]{ZK14}, as the last equation giving the expression of $\tilde{D}_\gamma\tilde{h}(\beta)$ only works for $\beta>\gamma$.

\paragraph{\cite[Section 3]{ZK14}} The arguments in this section need no change. For example, in the proof of \cite[Theorem 3.3]{ZK14}, \Cref{SymDerivExpression} is only used in the middle of \cite[Page 96]{ZK14}, after the phrase `By choice of $n_j$ and $s_j$, we have', when substituting 
    \begin{equation*}
    \tilde{D}_{n_{i+1}\cup\cdots\cup n_{j-1}}g(n_j)=g(n_j)^{-1}g(n_{i+1}\cdots n_j)g(n_{i+1}\cdots n_{j-1})^{-1}.
    \end{equation*}
    And by the definitions of the objects in the proof, $n_j>n_{i+1}\cup\cdots\cup n_{j-1}$.

\paragraph{\cite[Section 4]{ZK14}} The only results which involve the definition of IP-polynomial are \cite[Results 4.2, 4.8, 4.9, 4.10]{ZK14} and some results in \cite[Section 4.4]{ZK14}, and their proofs still work with our definition. In fact, the proof of \cite[Lemma 4.20]{ZK14}(similarly with \cite[Lemma 4.27]{ZK14}), which gives examples of FVIP systems, seems to use our definition of polynomials, as the derivative $\tilde{D}_\beta v(\alpha)$ is only computed when $\alpha>\beta$. 
    
In the proof of \cite[Theorem 4.2]{ZK14}, whenever we compute a derivative $(\hat{D}_\alpha g)_\beta$, either $\beta>\alpha$ or we are taking an IP-limit along $\beta$, so only the values of $\beta>\alpha$ are relevant.

\paragraph{\cite[Section 5]{ZK14}}
The change in definition does not affect how most results from this section follow, as in \cite{ZK14}, from analogous results in \cite{BMc}. \cite[Lemma 5.7]{ZK14} arrives at a contradiction using the following equation:
\begin{equation}
\label{r4ewodfislk}
0\stackrel{(1)}{=}\textup{w-}\IPlim_{\alpha,\beta}\tilde{D}_\beta g(\alpha)H=
\IPlim_{\alpha,\beta}
g(\alpha)^{-1}
g(\alpha\cup\beta)
g(\beta)^{-1}H\stackrel{(2)}{=}H.
\end{equation}
The limits of the form $\IPlim_{\alpha,\beta}$ used later in \cite{ZK14} (for the definition see \cite[Page 13]{BMc}) only involve pairs $(\alpha,\beta)$ such that $\alpha<\beta$, which is a problem if we use our definition directly.
However, we need only use \Cref{SymDerivExpression} when $\alpha>\beta$ in this argument. Indeed, we still reach a contradiction if we interpret equality (2) in \Cref{r4ewodfislk} to mean that for any neighborhood $U$ of $H$ in $L^2$ there is $\gamma\in\mathcal{F}$ such that, for all $\gamma<\beta<\alpha$ in $\mathcal{F}$, $g(\alpha)^{-1}
g(\alpha\cup\beta)
g(\beta)^{-1}H\in U$. 

In the proof of \cite[Theorem 5.18]{ZK14}, we take limits of the form $\IPlim_{\beta,\alpha}$ (so $\alpha>\beta$) of expressions involving derivatives $\tilde{D}_\beta S_i(\alpha)$, and the big proof, the proof of \cite[Theorem 5.26]{ZK14}, only uses \Cref{SymDerivExpression} implicitly when applying previous results which we have already checked work with our definition.\\

Zorin-Kranich gives in \cite[Remark 2.11]{ZK14} a definition of a polynomial map $p:G\to F$ when $G,F$ are groups and $F$ is nilpotent, equivalent to that of \cite{LePoly}: let $F_\bullet$ be a prefiltration with length $d$. If $d=-\infty$, define the class of $F_\bullet$-polynomials $P(G,F_\bullet)$ to only contain the constant map $1_F$. If $d\geq0$, say a function $\varphi:G\to F_0$ is in 
$P(G,F_\bullet)$ if for all $h\in G$, the derivative $D_h\varphi$ (see \Cref{DefPolyLeib}) is in $P(G,F_{\bullet+1})$. For convenience of the reader, we show in the following proposition that the definitions of polynomial map in \cite{LePoly} and \cite{ZK14} coincide.

\begin{prop}
If $G,F$ are groups and $F$ is nilpotent, then a function $p:G\to F$ is polynomial in the sense of \cite{LePoly} if and only if there is a filtration $F_\bullet$ of $F$ such that $p\in P(G,F_\bullet)$.
\end{prop}

\begin{proof}
For any filtration $F_\bullet$ of $F$ with length $d$, if $p\in P(G,F_{\bullet})$ then for any $h_1,\dots,h_{d+1}\in G$ we must have $D_{h_1}\cdots D_{h_{d+1}}p\in P(G,F_{\bullet+d+1})$, so as $F_{d+1}$ is trivial, $D_{h_1}\cdots D_{h_{d+1}}p=1_{F}$. Reciprocally, any polynomial map $p:G\to F$ of degree $d$ is a $F_\bullet$-polynomial, where $F_\bullet$ is obtained from the lower central series of $F$ by repeating every element $d$ times, as in \Cref{BigFiltration}.
\end{proof}
The following fact will be used in the next subsection: for a fixed group $H$ and any filtration $G_\bullet$, $P(G,F_\bullet)$ forms a nilpotent group with the operation being pointwise product (see \cite[Theorem 2.5, Corollary 2.9]{ZK14}; a similar result was obtained by Leibman in \cite[Sections 3.3, 3.4]{LePoly}).

\msubsection{Proofs of \Cref{PolyCesaroNilpotent,NilpPolyBB}}
\label{Sec5.2}
\begin{lemma}
\label{PolyTimesIPIsFVIP}
Let $G,F$ be finitely generated groups, where $F$ is nilpotent and has a filtration $F_\bullet$, and consider an IP-system
\begin{equation*}
g:\mathcal{P}_f(\mathbb{N})\to G;\alpha\mapsto g_\alpha:=\prod_{i\in\alpha}^<g_i.\textup{ (So, $g(\varnothing)=1_G$).}
\end{equation*}
Let $P(g,F_\bullet)$ be the family of all maps $h:\mathcal{P}_f(\mathbb{N})\to F$ of the form $h(\alpha)=p(g_\alpha^{-1})$, where $p\in P(G,F_\bullet)$, and let $P_0(g,F_\bullet)=\{p\in P(g,F_\bullet);p(\varnothing)=1_K\}$. Then any $h\in P(g,F_\bullet)$ is an
$F_\bullet$-FVIP system. 
\end{lemma}
In the proof of \Cref{PolyTimesIPIsFVIP} we use the following notation, for a polynomial map $p$ between groups:
\begin{equation*}
\tilde{D}_hp(g)=p(g)^{-1}p(gh)p(h)^{-1}.
\end{equation*}

\begin{proof}
We first prove by induction on the length of $F_\bullet$, that the maps in $P(g,F_\bullet)$ are IP-$F_\bullet$ polynomials. The base case is obvious, so suppose $F_\bullet$ has length $\geq0$. For all $\beta\in\mathcal{F}$ and $h=p\circ\left(g^{-1}\right)$ where $p\in P(G,F_\bullet)$, we have $D_\beta h=\left(D_{g_\beta^{-1}}p\right)\circ\left(g^{-1}\right)$, as for $\alpha>\beta$
\begin{equation*}
D_\beta h(\alpha)=h(\alpha)^{-1}h(\beta\cup\alpha)=p\left(g_\alpha^{-1}\right)^{-1}p\left(g_\alpha^{-1}g_\beta^{-1}\right)=D_{g_\beta^{-1}}p\left(g_\alpha^{-1}\right).
\end{equation*}
So $D_\beta h$ is in $P(g,F_{\bullet+1})$, which by inductive hypothesis implies $D_\beta h$ is an IP-$F_{\bullet+1}$ polynomial, so $h$ is an IP-$F_{\bullet}$ polynomial.

We now show that $P_0(g,F_\bullet)$ is a VIP group; this is obvious if $F_\bullet$ has length $-\infty$. If not, then $P(g,F_\bullet)$ contains all constant maps $\mathcal{F}\to F$, so $P_0(g,F_\bullet)$ is closed under conjugation by constants. Moreover, $P_0(g,F_\bullet)$ is closed under taking symmetric derivatives, as for all $\beta\in\mathcal{F}$ and $h=p\circ g$ where $p\in P(G,F_\bullet)$, we have 
\begin{equation*}
\tilde{D}_\beta h=\left(\tilde{D}_{g_\beta^{-1}}p\right)\circ g.
\end{equation*}

It remains to verify that if $h\in P_0(g,F_\bullet)$, then $h$ is contained in some $F_\bullet$-FVIP group. Say $h(\alpha)=p(g_\alpha^{-1})$ where $p\in P(G,F_\bullet)$. Let $G_0,F_0$ be finite sets of generators for $F,G$ respectively such that $G_0=G_0^{-1}$, and consider the subgroup $\mathcal{P}_p\subseteq P(G,F_\bullet)$ generated by $p$, the constant maps $\alpha\mapsto f$ for $f\in F_0$ and the iterated derivatives $D_{g_1}D_{g_2}\cdots D_{g_k}p$, with $k\in\mathbb{N}$ and $g_1,\dots,g_k\in G_0$. So $\mathcal{P}_p$ nilpotent (as it has a filtration by \cite[Corollary 2.9]{ZK14}) and finitely generated, because $D_{g_1}D_{g_2}\cdots D_{g_k}p=1_F$ if $k$ is bigger than the length of $F$. Also let $\mathcal{P}_p^0=\{p\in \mathcal{P}_p;p(1_G)=1_K\}$. 

It will be enough to prove that $\mathcal{P}_p$ is closed under taking derivatives; if so, its subgroup $\mathcal{P}_p^0$ is nilpotent and finitely generated (as a subgroup of the nilpotent finitely generated group $\mathcal{P}_p$), and it is closed under taking symmetric derivatives and conjugation by constants, so the group $\{p\circ\left(g^{-1}\right);p\in \mathcal{P}_p^0\}$ is FVIP.

Note first that if $g_0\in G_0$, then $D_gq\in \mathcal{P}_p$ for all $q\in\mathcal{P}_p$; indeed, the set of elements $q\in\mathcal{P}_p$ such that $D_{g_0}q\in \mathcal{P}_p$ contains the generating set of $\mathcal{P}_p$ specified above, and it is closed under taking products and inverses, because for $g\in G$ and $q,r\in\mathcal{P}_p$,
\begin{equation*}
D_g(qr)=r^{-1}D_g(q)rD_g(r)\textup{, and }D_g\left(q^{-1}\right)=q(D_gq)^{-1}q^{-1}.
\end{equation*}
More generally, if $q\in\mathcal{P}_p$ and $g\in G$, then $D_gq\in \mathcal{P}_p$, as this is true if $g\in G_0=G_0^{-1}$ and, for all $g,h\in G$,
\begin{equation*}
D_{gh}q=D_gq\cdot D_hq\cdot D_gD_hq.\qedhere
\end{equation*}
\end{proof}

\begin{remark}
Although \Cref{PolyTimesIPIsFVIP} will be sufficient for our purposes, we suspect that $P(G,F_\bullet)$ is finitely generated, which would imply by the proof of \Cref{PolyTimesIPIsFVIP} that the entire $P_0(g,F_\bullet)$ is an FVIP group.
\end{remark}

\begin{cor}
\label{SubIPRingWithBigLim}
Let $G,F$ be finitely generated groups, with $F$ being nilpotent, let $p_1,\dots,p_s:G\to F$ be polynomial maps satisfying $p_i(1_G)=1_F$. Then, for any m.p.s. $(X,\mathcal{B},\mu,(T_g)_{g\in F})$, $B\in\mathcal{B}$ satisfying $\mu(B)>0$ and any IP-system $g:\mathcal{F}\to G;\alpha\mapsto g_\alpha=\prod_{i\in\alpha}^<g_i$, there is an IP ring $\mathcal{F}'\subseteq\mathcal{F}$ such that 
\begin{equation*}
\lim_{\alpha\in\mathcal{F}'}\mu\left(\bigcap_{i=1}^sT_{p_i(g_\alpha)}B\right)>0.
\end{equation*}
\end{cor}

\begin{proof}
By \Cref{InvertingInsidePolyIsInocuous}, the maps $q_i:G\to F$; $q_i(g)=p_i(g^{-1})$ are polynomial, so we may choose a filtration $F_\bullet$ (as in \Cref{BigFiltration} for big enough $d$) such that $q_1,\dots,q_s\in P(G,F_\bullet)$, and by \Cref{PolyTimesIPIsFVIP}, the maps $\alpha\mapsto q_i(g_\alpha^{-1})=p_i(g_\alpha)$ are $F_\bullet$-FVIP systems. The result then follows from \Cref{SimplerZKThm32}.
\end{proof}

\begin{lemma}
\label{InvertingInsidePolyIsInocuous}
If $p:G\to F$ is a polynomial mapping from a group $G$ to a nilpotent group $F$, then $g\mapsto p(g^{-1})$ is a polynomial too.
\end{lemma}
\begin{proof}
By \cite[Prop. 3.21]{LePoly}, 
there exist a nilpotent group $G'$, a homomorphism $\pi:G\to G'$ and a polynomial $q:G'\to F$  such that $p=q\circ\pi$. So we may see $g\mapsto p(g^{-1})=q(\pi(g^{-1}))=q(\pi(g)^{-1})$ as the composition of the three polynomial maps $\pi:G\to G'$, $\textup{inv}:G'\to G';g\mapsto g^{-1}$ and $q:G'\to F$, so by \cite[Prop. 3.22]{LePoly}, $g\mapsto p(g^{-1})$ is a polynomial.
\end{proof}

\PolyCesaroNilpotent*
\begin{proof}
We may assume $H$ is finitely generated, as \cite[Corollary 1.18]{LePoly} implies that $p_j(G)$ is contained in a finitely generated subgroup of $H$ for all $j=1,\dots,s$. The fact that the set $R$ from \Cref{EqReturnsNilpPoly} is IP$^*$ then follows from \Cref{SubIPRingWithBigLim}.

Now suppose $G$ is amenable. \cite[Theorem 1.1]{Zo16} says that the limit in \Cref{EqBigAvgReturnsPolyNil} exists and takes the same value for all left F{\o}lner sequences. The limit also has to exist and take the same value in all right F{\o}lner sequences, as the maps $g\mapsto p_i(g^{-1})$ are also polynomials by \Cref{InvertingInsidePolyIsInocuous}, and in fact the value of the limit coincides for both left and right F{\o}lner sequences, as there exist two-sided F{\o}lner sequences in $G$, i.e. sequences that are both left and right F{\o}lner (indeed, by \cite[Theorem 11.3.23]{Ba}, as $G\times G$ is amenable, there exist sequences $(F_N)$ of finite subsets of $G$ almost-invariant by the action of $G\times G$ by multiplication on both sides, i.e. $(a,b)(g)=agb^{-1}$).
Now suppose there is a left F{\o}lner sequence $(F_N)$ in $G$ such that 
\begin{equation*}
\lim_{N\to\infty}\frac{1}{|F_N|}\sum_{g\in F_N}\mu\left(\bigcap_{i=1}^sT_{p_i(g)}A\right)=0.
\end{equation*}
Then by \Cref{BadIPSystem} below there is an IP-system $g:\mathcal{F}\to G$ such that 
\begin{equation*}
\lim_{\alpha\in\mathcal{F}}\mu\left(\bigcap_{i=1}^sT_{p_i(g_\alpha)}B\right)=0,
\end{equation*}
contradicting \Cref{SubIPRingWithBigLim}.
\end{proof}

\begin{lemma}
\label{BadIPSystem}
Let $G$ be a countable amenable group with a left F{\o}lner sequence $F=(F_N)$. Let $(x(g))_{g\in G}$ be a sequence of non-negative real numbers, and suppose that 
\begin{equation}
\label{BadIPSystemEq1}
\lim_{N\to\infty}\frac{1}{|F_N|}\sum_{g\in F_N}x(g)=0.
\end{equation}
Then there exists an IP-system $g:\mathcal{F}\to G$ such that $\lim_{\alpha\in\mathcal{F}}x(g_\alpha)=0$.
\end{lemma}

\begin{proof}
We construct by recursion a sequence $(g_n)$ of elements of $G$ such that for all $n_1<\cdots<n_k$, we have $x(g_{n_1}\cdots g_{n_k})<\frac{1}{n_k}$. Clearly we can find $g_1$ such that $x(g_1)<1$, so assume we have already found $g_1,\dots,g_{n-1}$. We need to find $g_n$ such that $x(hg_n)<\frac{1}{n}$ for all $h$ in the finite set $B_n$ of products of $g_1,\dots,g_{n-1}$ with increasing index. That is, we need $g_n\in\bigcap_{h\in B_n}h^{-1}A$, where $A=\{g\in G;x(g)<\frac{1}{n}\}$. But by \Cref{BadIPSystemEq1}, we have $d_F(A)=1$, and $d_F(h^{-1}A)=d_F(A)=1$ for all $h$, because $F$ is left F{\o}lner. So $d_F\left(\bigcap_{h\in B_n}h^{-1}A\right)=1$, which allows us to pick $g_n$ and continue the recursion.
\end{proof}

\NilpPolyBB*
\begin{proof}
By \Cref{ReturnsInDifferences} there is a m.p.s. $(X,\mathcal{B},\mu,(T_a)_{a\in H})$ and $Y\in\mathcal{B}$ such that $\mu(Y)=d^*_l(A)$ and $R^T_\mu(Y)\subseteq AA^{-1}$. By \Cref{TheBigTechnicalTheorem} it is enough to prove that 
\begin{equation*}
\lim\limits_{N\to\infty}\frac{1}{|F_N|}\sum_{g\in F_N}\mu\left(\bigcap_{j=1}^sT_{q_i(g)}Y\cap T_{p_i(g)}^{-1}Y\right)>0.
\end{equation*}
And the inequality follows from \Cref{PolyCesaroNilpotent}, as for all $j$, $q_j$ and $p_j^{-1}$ are polynomials and $q_j(1_G)=p_j^{-1}(1_G)=1_H$.
\end{proof}

\NilpPolyBxB*
\begin{proof}
Note that for $j=1,\dots,s$, the maps $p_j':G\to H\times H;g\mapsto(p_j(g),1_H)$ and $q_j':G\to H\times H;g\mapsto(1_H,q_j(g))$ are both polynomial, as the composition of a polynomial and a homomorphism is a polynomial. Moreover, for any $B\subseteq G$, $p_j'(B)q_j'(B)=p_j(B)\times q_j(B)$. So \Cref{NilpPolyBxB} follows from \Cref{NilpPolyBB} applied to the group $H\times H$ and the polynomials $p_j',q_j'$
\end{proof}

\msubsection{Proof of \Cref{BigAvgsAlongSquaresFGNilp}}
\label{Sec5.3}

The goal of this subsection is to prove \Cref{BigAvgsAlongSquaresFGNilp}, which we restate here for the convenience of the reader.

\label{HereIsrweoidsioklfsd}
\BigAvgsAlongSquaresFGNilp*

We first record how F{\o}lner sequences behave under passage to a
finite-index subgroup.

\begin{lemma}
\label{SubgroupsAreNiceWrtFolners}
Let $G$ be a countable amenable group, and let $H\leq G$ have finite
index $q=[G:H]$. If $F=(F_N)$ is a left (right) F{\o}lner sequence in
$G$, then, after discarding finitely many initial terms, $(F_N\cap H)$
is a left (right) F{\o}lner sequence in $H$, and
\begin{equation}
\label{re9fsdo9redfoi}
  d_F(H)=\lim_{N\to\infty}\frac{|F_N\cap H|}{|F_N|}=\frac1q.
\end{equation}
\end{lemma}

\begin{proof}
Suppose that $(F_N)$ is left F{\o}lner. For every $t\in G$,
\[
  \bigl||F_N\cap tH|-|F_N\cap H|\bigr|
  \leq |t^{-1}F_N\mathbin{\triangle}F_N|=o(|F_N|).
\]
Summing over the $q$ left cosets of $H$ proves \eqref{re9fsdo9redfoi}.
Moreover, for every $h\in H$,
\[
  \frac{|h(F_N\cap H)\mathbin{\triangle}(F_N\cap H)|}{|F_N\cap H|}
  \leq
  \frac{|hF_N\mathbin{\triangle}F_N|}{|F_N|}
  \frac{|F_N|}{|F_N\cap H|}
  \longrightarrow0.
\]
For right F{\o}lner sequences, use right cosets and right translations.
\end{proof}

We use the Hall--Mal'cev coordinate theorem in the following form;
see \cite[Section~2 and Theorem~1]{deGraafNickel}.

\begin{lemma}[Coordinates and multiplication]\label{NilpotentCoordinates}
Let $L$ be a torsion-free finitely generated nilpotent group. There are
elements $g_1,\ldots,g_k$ such that
\[
  \Phi:\mathbb{Z}^k\longrightarrow L,\qquad
  \Phi(x)=g_1^{x_1}\cdots g_k^{x_k},
\]
is a bijection. Writing $L_i=\langle g_i,\ldots,g_k\rangle$ and
$L_{k+1}=\{1\}$, these elements can be chosen so that
\begin{equation}\label{NilpotentCoordinateSeries}
  L_i/L_{i+1}\cong\mathbb{Z},
  \qquad [L,L_i]\subseteq L_{i+1}
  \qquad (1\leq i\leq k).
\end{equation}
In particular, $[g_i,g_j]\in\langle g_{j+1},\ldots,g_k\rangle$ for $i<j$.

Identifying each element of $L$ with its coordinate tuple in $\mathbb{Z}^k$,
the product of any two elements is given by
\begin{equation}\label{NilpotentCoordinateProduct}
\begin{gathered}
  (x_1,\ldots,x_k)(y_1,\ldots,y_k)=(z_1,\ldots,z_k),
  \\
  z_i=x_i+y_i+p_i(x_1,y_1,\ldots,x_{i-1},y_{i-1}),
\end{gathered}
\end{equation}
where, for each $i$, $p_i\in\mathbb{Q}[x_1,y_1,\ldots,x_{i-1},y_{i-1}]$ is independent of $x$ and $y$. Every nonzero monomial in $p_i$ has total
degree at least two.
\end{lemma}

We write $\operatorname{Sq}:G\to G;g\mapsto g^2$ for the squaring map in a group.
\begin{lemma}\label{NilpotentEvenCoordinates}
If $L$ is a torsion-free, finitely generated nilpotent group, then the squaring map $\operatorname{Sq}:L\to L$ is injective. If $\operatorname{Sq}(L)$ is a subgroup, then
\begin{equation}\label{NilpotentEvenSquares}
  \operatorname{Sq}(L)=\{g_1^{2x_1}\cdots g_k^{2x_k}:x\in\mathbb{Z}^k\}.
\end{equation}
\end{lemma}

\begin{proof}
We use the coordinates and notation of \Cref{NilpotentCoordinates}.
By \eqref{NilpotentCoordinateProduct}, the coordinates of a square are
\begin{equation}\label{NilpotentSquareCoordinates}
  \Phi(x)^2=\Phi(s(x)),\qquad
  s_i(x)=2x_i+p_i(x_1,x_1,\ldots,x_{i-1},x_{i-1}).
\end{equation}
If $s(x)=s(y)$, its first coordinate gives $x_1=y_1$. Comparing the
remaining coordinates from $2$ up to $k$ gives $x=y$.

Now assume $S:=\operatorname{Sq}(L)$ is a subgroup. Every element with even coordinates
is a product of squares, so belongs to $S$. For the reverse inclusion,
induct on $k$. The case $k=0$ is trivial. Put
$K=L_2$. Since $L/K\cong\mathbb{Z}$, an element whose square belongs to $K$
must itself belong to $K$. Consequently,
\[
  S\cap K=\{u^2:u\in K\},
\]
which is a subgroup. If $s\in S$ is written as $s=g_1^n y$ with $y\in K$,
its image in $L/K$ shows that $n=2m$ is even. Then
$y=g_1^{-2m}s\in S\cap K$. By induction all coordinates of $y$ are
even, proving \eqref{NilpotentEvenSquares}.
\end{proof}

The following lemma was communicated to us by Yves de Cornulier.\footnote{See Yves' answer to the MathOverflow question \url{https://mathoverflow.net/q/488422}.}
His result \cite[Lemma~A.4]{Cornulier} gives a stronger statement:
the same finite-index subgroup $H$ can be chosen so that
$\{h^r:h\in H\}$ is a subgroup for every integer $r\geq1$.
We include an elementary proof of the lemma for the convenience of the reader.

\begin{lemma}
\label{NilpotentSquareSubgroup}
Every finitely generated nilpotent group $G$ contains a torsion-free
finite-index subgroup $H$ such that $\operatorname{Sq}(H)$ is a subgroup.
\end{lemma}

\begin{proof}
First take a torsion-free finite-index subgroup $L\leq G$. Its existence
is standard: finitely generated nilpotent groups are polycyclic, and
polycyclic groups are virtually torsion-free; see
\cite[Proposition 11.5 and Corollary 11.9]{DK}.
Choose coordinates $\Phi:\mathbb{Z}^k\to L$ and polynomials $p_i$ as in \Cref{NilpotentCoordinates}.
If $L=\{1\}$, take $H=L$, so assume $k\geq1$.

Let $D\in\mathbb{N}$ satisfy $Dp_i\in\mathbb{Z}[x]$ for all $i$.
For every positive integer $t$ divisible by $D$, put
\[
  L(t)=\Phi(t\mathbb{Z}^k).
\]
Because every monomial of $p_i$ has degree at least two,
\begin{equation}\label{NilpotentCoordinateDivisibility}
  p_i(tx_1,ty_1,\ldots,tx_{i-1},ty_{i-1})
  \in \frac{t^2}{D}\mathbb{Z}\subseteq t\mathbb{Z}
  \qquad (x,y\in\mathbb{Z}^k).
\end{equation}
It follows from \eqref{NilpotentCoordinateProduct} that $L(t)$ is closed under
multiplication. It is also closed under inverses: to solve
$\Phi(x)\Phi(y)=1$ with $x\in t\mathbb{Z}^k$, solve successively
\begin{align*}
  y_1&=-x_1,\\
  y_i&=-x_i-p_i(x_1,y_1,\ldots,x_{i-1},y_{i-1})
  \quad(2\leq i\leq k).
\end{align*}
Equation~\eqref{NilpotentCoordinateDivisibility} shows at each step that $y_i\in t\mathbb{Z}$.
Thus $L(t)$ is a subgroup.

It has finite index $t^k$. Indeed, $L(t)\cap L_i$ consists exactly of
the elements of $L_i$ whose coordinates are divisible by $t$.
Its image in $L_i/L_{i+1}\cong\mathbb{Z}$ is $t\mathbb{Z}$, so the subgroup index
formula gives, successively,
\[
  [L_i:L(t)\cap L_i]
  =t\,[L_{i+1}:L(t)\cap L_{i+1}].
\]

Finally, let $H=L(2D)$. Equation~\eqref{NilpotentCoordinateDivisibility} implies that for any $h\in H$, the coordinates of $h^2$ are integer multiples of $4D$, so $\operatorname{Sq}(H)\subseteq L(4D)$.
Conversely, for any $b\in\mathbb{Z}^k$, any element $\Phi(4Db)$ of $L(4D)$ is the square of some $\Phi(x)\in H$, where the coordinates of $x$ are given by
\begin{equation}\label{NilpotentSquareRoot}
\begin{aligned}
  x_1&=2Db_1,\\
  x_i&=2Db_i-\frac12p_i(x_1,x_1,\ldots,x_{i-1},x_{i-1})
  \quad(2\leq i\leq k).
\end{aligned}
\end{equation}
By recursion using \eqref{NilpotentCoordinateDivisibility}, all coordinates of $x$ are multiples of $2D$, and we have
$\Phi(x)^2=\Phi(4Db)$ by \eqref{NilpotentCoordinateProduct}.
Hence $\operatorname{Sq}(H)=L(4D)$, which is a subgroup as proved above.
\end{proof}

\begin{lemma}[F{\o}lner boxes]\label{NilpotentFolnerBoxes}
Let $H$ have coordinates $\Phi:\mathbb{Z}^k\to H$, with $k\geq1$, and
multiplication as in \Cref{NilpotentCoordinates}. Define
\[
  R_1(N)=N,\qquad R_{i+1}(N)=R_i(N)^N
  \quad(1\leq i<k).
\]
Then the coordinate boxes
\[
  B_N=\prod_{i=1}^k([-R_i(N),R_i(N)]\cap\mathbb{Z}),
  \qquad F_N=\Phi(B_N)
\]
form a two-sided F{\o}lner sequence in $H$.
\end{lemma}

\begin{proof}
Abbreviate $R_i(N)$ to $R_i$ and fix $h\in H$.
Each $R_{i+1}$ grows faster than every fixed polynomial in $R_i$.
By \eqref{NilpotentCoordinateProduct}, left or right multiplication by $h$
therefore changes the $i$-th coordinate by $o(R_i)$, uniformly over
$F_N$: the change is polynomial in the earlier coordinates, and the
first coordinate changes by a constant.
An element can consequently leave the box only through boundary strips
of width $o(R_i)$, which together contain $o(|F_N|)$ elements. Hence
\[
  \frac{|hF_N\mathbin{\triangle}F_N|}{|F_N|}
  =2\frac{|hF_N\setminus F_N|}{|F_N|}
  \longrightarrow0.
\]
The same estimate holds for right translations.
\end{proof}

\begin{lemma}[A F{\o}lner sequence whose squares are F{\o}lner]\label{NilpotentFolnerSquares}
Let $H$ be a torsion-free finitely generated nilpotent group such that
$S:=\operatorname{Sq}(H)$ is a subgroup. Then there is a two-sided F{\o}lner sequence
$(F_N)$ in $H$ such that $(\operatorname{Sq}(F_N))$ is a two-sided F{\o}lner sequence
in $S$.
\end{lemma}

\begin{proof}
Choose coordinates as in \Cref{NilpotentCoordinates}. If $k=0$, take
$F_N=\{1\}$ and we are done. Otherwise, define $B_N,F_N$ as in
\Cref{NilpotentFolnerBoxes}, and put
$E_N=\Phi(2B_N)$ and $Q_N=\operatorname{Sq}(F_N)$.
By \Cref{NilpotentEvenCoordinates}, $\Psi(x)=\Phi(2x)$ is a coordinate map for $S$.
Its multiplication again has the triangular polynomial form of
\Cref{NilpotentCoordinates}. \Cref{NilpotentFolnerBoxes} therefore shows that
$(F_N)$ is two-sided F{\o}lner in $H$ and $(E_N)$ is two-sided F{\o}lner in $S$.

It remains to compare $Q_N$ with $E_N$. By \eqref{NilpotentSquareCoordinates},
the coordinates of $\Phi(x)^2$ are $2x_i+o(R_i(N))$, uniformly for
$x\in B_N$, because the polynomials are polynomial in the earlier
coordinates and $R_i(N)>R_j(N)^N$ for all $j<i$. They are also even by \Cref{NilpotentEvenCoordinates}.
Thus $\Phi(x)^2\in E_N$ except possibly when $x$ lies in boundary strips
containing $o(|B_N|)$ points. Since squaring is injective,
$|Q_N|=|E_N|=|B_N|$, and consequently
\[
  \frac{|Q_N\mathbin{\triangle}E_N|}{|E_N|}
  =2\frac{|Q_N\setminus E_N|}{|B_N|}
  \xrightarrow{N\to\infty}0.
\]
Consequently, as $(E_N)$ is a two-sided F{\o}lner sequence in $S$, so is $(Q_N)$. 
\end{proof}

\begin{proof}[Proof of \Cref{BigAvgsAlongSquaresFGNilp}]
By \Cref{NilpotentSquareSubgroup}, choose a torsion-free finite-index
subgroup $H\leq G$ such that $S:=\operatorname{Sq}(H)$ is a subgroup,
and put $q=[G:H]$. Let $(F_N)$ be the two-sided F{\o}lner sequence in
$H$ provided by \Cref{NilpotentFolnerSquares}, so that
$Q_N:=\operatorname{Sq}(F_N)$ is a two-sided F{\o}lner sequence in $S$.
Since squaring is injective on $H$, \Cref{VonNeumannmuASquaredIntro},
applied to the action restricted to $S$, gives
\begin{align*}
  \UClim_{h\in H}\mu(Y\cap T_h^2Y)
  &=\lim_{N\to\infty}\frac1{|F_N|}
      \sum_{h\in F_N}\mu(Y\cap T_h^2Y)\\
  &=\lim_{N\to\infty}\frac1{|Q_N|}
      \sum_{s\in Q_N}\mu(Y\cap T_sY)
   \geq\mu(Y)^2.
\end{align*}

Now let $(\Omega_N)$ be any left or right F{\o}lner sequence in $G$.
By \Cref{SubgroupsAreNiceWrtFolners}, the sequence $(\Omega_N\cap H)$
is F{\o}lner in $H$ and $|\Omega_N\cap H|/|\Omega_N|\to1/q$.
Since the correlations are nonnegative, it follows that
\begin{align*}
  \UClim_{g\in G}\mu(Y\cap T_g^2Y)
  &=\lim_{N\to\infty}\frac1{|\Omega_N|}
      \sum_{g\in\Omega_N}\mu(Y\cap T_g^2Y)\\
  &\geq\lim_{N\to\infty}
      \frac{|\Omega_N\cap H|}{|\Omega_N|}
      \frac1{|\Omega_N\cap H|}
      \sum_{h\in\Omega_N\cap H}\mu(Y\cap T_h^2Y)\\
  &=\frac1q\UClim_{h\in H}\mu(Y\cap T_h^2Y)
   \geq\frac1q\mu(Y)^2.\qedhere
\end{align*}
\end{proof}

%% file: S5-Solvable.tex
\section{More groups with the SAR property}
\label{SecSemiProds} 
In this section we develop techniques for proving symmetric averaging recurrence for classes of groups beyond the abelian and finitely generated nilpotent cases. 
In \Cref{SecSemiprodsButTrue} we prove that, in a sense, the SAR property is preserved under taking semidirect products. In \Cref{SecMatricesFields}, we prove that \(\textup{GL}_n(Q)\) has the SAR property whenever \(n\geq 2\) and \(Q\) is a tower of finite fields (these are precisely the fields \(Q\) for which \(\textup{GL}_n(Q)\) is amenable).

\subsection{Semidirect products}
\label{SecSemiprodsButTrue}
In this subsection, we show that if a group with the SAR property acts on an abelian group, then the resulting semidirect product also has the SAR property.

Let \(H\) and \(K\) be amenable groups, and let \(\phi\colon K\to\textup{Aut}(H)\) be a homomorphism. Recall that the semidirect product \(G=H\rtimes_\phi K\) is a group with underlying set \(H\times K\) and operation
\begin{equation*}
(h,k)(h',k')=(h\phi_k(h'),kk').
\end{equation*}

We identify $H,K$ with the subgroups $H\times\{1_K\}$ and $\{1_H\}\times K$ in $G$. 

\begin{lemma}
\label{438erwdfiso98iorwe9}
Let $(H_M)_{M\in\mathbb{N}},(K_N)_{N\in\mathbb{N}}$ be left F{\o}lner sequences in $H,K$ respectively. If a sequence $(M(N))_{N\in\mathbb{N}}$ of natural numbers grows fast enough, then the sequence $(F_N)_{N\in\mathbb{N}}$ given below is left F{\o}lner in $G$:
\begin{equation*}
F_N=K_N\cdot H_{M(N)}={\textstyle\bigcup_{k\in K_N}}\phi_k(H_{M(N)})\times\{k\}.
\end{equation*}
\end{lemma}

\begin{proof}
It is enough to check that $\lim_{N\to\infty}\frac{|F_N\Delta kF_N|}{|F_N|}=0$ for all $k\in K$ and $\lim_{N\to\infty}\frac{|F_N\Delta hF_N|}{|F_N|}=0$ for all $h\in H$. For any $k\in K$, letting $F_{N,M}=K_NH_M$, we have
\begin{align*}
\frac{|F_{N,M}\Delta kF_{N,M}|}{|F_{N,M}|}
=\frac{|(K_N\Delta kK_N)H_{M}|}{|K_N|\cdot|H_{M}|}=
\frac{|K_N\Delta kK_N|}{|K_N|}\xrightarrow{N\to\infty}0.
\end{align*}
We now choose the sequence $(M(N))$. By \Cref{FolnerProps}.\ref{AutosPreserveFolner}, $(\phi_k(H_N))_{N\in\mathbb{N}}$ is a left F{\o}lner sequence in $H$ for all $k\in K$. Thus, for any fixed $N$ and $h\in H$,
\begin{align*}
\lim_{M\to\infty}\frac{|F_{N,M}\Delta hF_{N,M}|}{|F_{N,M}|}
&=
\lim_{M\to\infty}
\frac{\left|\bigcup_{k\in K_N}(\phi_k(H_M)\Delta h\phi_k(H_M))\cdot\{k\}\right|}{|K_N|\cdot|H_M|}
\\
&=
\frac{1}{|K_N|}\sum_{k\in K_N}
\lim_M\frac{|\phi_k(H_M)\Delta h\phi_k(H_M)|}{|H_M|}=0.
\end{align*}
Thus, if $(h_n)$ is an enumeration of $H$, it is enough to let $M(N)$ be an integer $M$ such that, for all $n=1,\dots,N$, we have $\frac{|F_{N,M}\Delta h_nF_{N,M}|}{|F_{N,M}|}<\frac{1}{N}$.
\end{proof}

\begin{remark}
The F{\o}lner sequence $F_N$ from \Cref{438erwdfiso98iorwe9} (and thus in \ref{BigSqAvgsInSemiprods}, \ref{TowersSemiprodsPositiveAvgs}) can be taken to be a two-sided F{\o}lner sequence, as long as $(K_N)$ is a two-sided F{\o}lner sequence in $K$ and $(H_M)$ is `central in $G$', in the sense that for all $g\in G$ we have
\begin{equation}
\label{CentralFolnerSeq}
\lim_{M\to\infty}\frac{|gH_M\Delta H_Mg|}{|H_M|}=0.
\end{equation}
Moreover, for any two countable amenable groups $H\trianglelefteq G$ there is a two-sided F{\o}lner sequence in $H$ which is central in $G$. Indeed, letting $\mathcal{G}$ be the group of bijections $f:H\to H$ of the form $f(h)=g_1hg_2$, for $g_1,g_2\in G$ satisfying $g_1Hg_2=H$, then $\mathcal{G}$ is amenable (it can be seen as a subgroup of a homomorphic image of $G\times G^{\textup{op}}$), so there must be a F{\o}lner sequence $(H_M)$ in $H$ that is asymptotically invariant with respect to the action of $\mathcal{G}$ (see \cite[Theorem 11.3.23]{Ba}). We then obtain \Cref{CentralFolnerSeq} by letting $g_1=g=g_2^{-1}$.
\end{remark}

\BigSqAvgsInSemiprods*
\begin{proof}
Let $(X,\mathcal{B},\mu,(T_g)_{g\in G})$ be a m.p.s. and $A\in\mathcal{B}$ such that $\mu(A)\geq\varepsilon$. Choose left F{\o}lner sequences $(H_M)_{M\in\mathbb{N}}$ and $(K_N)_{N\in\mathbb{N}}$ in $H$ and $K$, respectively. Since $K$ has the $(\varepsilon,\delta)$-SAR property,
\begin{equation*}
\lim_{N\to\infty}\frac{1}{|K_N|}\sum_{k\in K_N}\mu(A\cap T_k^2A)
=\UClim_{k\in K}\mu(A\cap T_k^2A)\geq\delta.
\end{equation*}

For each $N,M\in\mathbb{N}$ let $F_{N,M}=K_N\cdot H_{M}=\cup_{k\in K_N}\phi_k(H_M)\times\{k\}$. Any $g\in F_{N,M}$ is of the form $k\cdot h=(\phi_k(h),k)$, where $h\in H_{M}$ and $k\in K_{N}$. Therefore, 
\begin{equation*}
g^2=(\phi_k(h),k)\cdot(\phi_k(h),k)=(\phi_k(h)\phi_k^2(h),k^2).
\end{equation*}
Letting $A_k=A\cap T_k^2A$ for all $k\in K$, we have
\begin{align*}
\frac{1}{|F_{N,M}|}\sum_{g\in F_{N,M}}\mu(A\cap T_g^2A)
&=
\frac{1}{|K_N|}\sum_{k\in K_N}
\frac{1}{|H_M|}\sum_{h\in H_M}
\mu(A\cap T_{\phi_k(h)\phi_k^2(h)}T_k^2A)
\\
&\geq
\frac{1}{|K_N|}\sum_{k\in K_N}
\frac{1}{|H_M|}\sum_{h\in H_M}
\mu(A_k\cap T_{\phi_k(h)\phi_k^2(h)}A_k).
\end{align*}
As $H$ is abelian, for each fixed $k\in K$ the map $h\to\phi_k(h)\phi_{k^2}(h)$ is an endomorphism of $H$. Thus, $\left(T_{\phi_k(h)\phi_k^2(h)}\right)_{h\in H}$ is an action of $H$ on $(X,\mathcal{B},\mu)$ by measure preserving maps. So, for each fixed $k\in K$ we have, by \Cref{LinearAvgsAreBig},
\begin{equation*}
\lim_M
\frac{1}{|H_M|}\sum_{h\in H_M}
\mu(A_k\cap T_{\phi_k(h)\phi_k^2(h)}A_k)\geq\mu(A_k)^2.
\end{equation*}
If $(M(N))$ grows fast enough, then by \Cref{438erwdfiso98iorwe9} the sequence $(F_{N,M(N)})$ is left F{\o}lner and
\begin{align*}
\UClim_{g\in G}\mu(A\cap T_g^2A)
&=\lim_{N\to\infty}\frac{1}{|F_{N,M(N)}|}\sum_{g\in F_{N,M(N)}}\mu(A\cap T_g^2A)\\
&\geq\overlim_{N\to\infty}\frac{1}{|K_N|}\sum_{k\in K_N}\mu(A_k)^2\\
&=\overlim_{N\to\infty}\frac{1}{|K_N|}\sum_{k\in K_N}\mu(A\cap T_k^2A)^2\\
&\geq\lim_{N\to\infty}\left(\frac{1}{|K_N|}\sum_{k\in K_N}\mu(A\cap T_k^2A)\right)^2\geq\delta^2.\qedhere
\end{align*}
\end{proof}

\begin{cor}
\label{TowersSemiprodsPositiveAvgs}
Let $G$ be a countable group with a sequence $G=G_k\supseteq G_{k-1}\supseteq\cdots\supseteq G_1\supseteq G_0=\{e\}$, such that $G_i=H_i\rtimes_{\phi_i}G_{i-1}$, for some abelian group $H_i$ and some homomorphism $\phi_i:G_{i-1}\to\textup{Aut}(H_i)$. Then for all $\varepsilon>0$, $G$ has the $\left(\varepsilon,\varepsilon^{2^k}\right)$-SAR property.
\end{cor}

\begin{proof}
The group $G_1=H_1$ is abelian, so it has the $(\varepsilon,\varepsilon^2)$-SAR property. Apply \Cref{BigSqAvgsInSemiprods} successively to $G_i=H_i\rtimes_{\phi_i}G_{i-1}$ for $i=2,\ldots,k$; each step squares the lower bound.
\end{proof}

Combining \Cref{ProdsInRetsMeasurableIntro} and \Cref{TowersSemiprodsPositiveAvgs}, we obtain:

\begin{cor}
\label{BBAA-1TowersSemiprods}
Let $G$ be a group constructed as in \Cref{TowersSemiprodsPositiveAvgs}. Then for every $A\subseteq G$ and every left or right F{\o}lner sequence $F=(F_N)$ in $G$, there is $B\subseteq G$ with $\overline{d}_F(B)\geq(d^*_l(A))^{2^k}$ and $BB\subseteq AA^{-1}$.
\end{cor}

\begin{remark}
For a fixed $r\in\mathbb{N}$, the same calculation transfers positive lower bounds for averages of $\mu(A\cap T_k^rA)$ along a F{\o}lner sequence in $K$ to positive lower bounds for averages of $\mu(A\cap T_g^rA)$ along a suitable F{\o}lner sequence in $G$. Indeed, for $g\in F_{N,M}$, $g=k\cdot h$, we have
\begin{equation*}
g^r=(\phi_k(h),k)^r=(\phi_k(h)\phi_k^2(h)\cdots\phi_k^r(h),k^r).
\end{equation*}
Iterating this argument from abelian groups and applying \Cref{ProdsInReturnsk=1} gives variants of \Cref{BBAA-1TowersSemiprods} with $B^nB^m\subseteq AA^{-1}$ for any $n,m\in\mathbb{Z}$, with positive upper density along suitable F{\o}lner sequences.
\end{remark}

\HnRHasBiSquareAverages*
\begin{proof}
The case $n=1$ is immediate. Suppose $n\geq2$.
Note that we can express $\mathrm{UT}_n(R)$ as an internal semidirect product
$\mathrm{UT}_n(R)=H_n\rtimes K_n\cong R^{n-1}\rtimes\mathrm{UT}_{n-1}(R)$, where $H_n,K_n\leq\mathrm{UT}_n(R)$ are as follows.
\begin{align*}
K_n&=
\left\{
\begin{pmatrix}
A&0\\0&1
\end{pmatrix};A\in \mathrm{UT}_{n-1}(R)
\right\}\\
H_n&=\left\{
\begin{pmatrix}
\textup{I}&v\\0&1
\end{pmatrix};v\in R^{n-1}
\right\}
\end{align*}
Iterating this decomposition expresses $G$ as a tower of $n-1$ semidirect products with abelian normal factors. The assertion, for every left or right F{\o}lner sequence, follows from \Cref{BBAA-1TowersSemiprods} with $k=n-1$.
\end{proof}

We conclude this subsection with a question whose affirmative resolution would extend the proof of \Cref{BigSqAvgsInSemiprods} from semidirect products \(G=H\rtimes_\phi K\) to arbitrary group extensions, thereby yielding an analogue of \Cref{BigSqAvgsInSemiprods} for all solvable groups.
\begin{question}
Let $H\trianglelefteq G$ be amenable groups, let $K=G/H$ have a F{\o}lner sequence $(F_N)_{N\in\mathbb{N}}$ such that, for all m.p.s. $(X,\mathcal{B},\mu,(T_k)_{k\in K})$ and every $Y\in\mathcal{B}$ with $\mu(Y)>0$, we have 
\begin{equation*}
\overlim_{N\in\mathbb{N}}\frac{1}{|F_N|}\sum_{k\in F_N}\mu(Y\cap T_k^2Y)>0.
\end{equation*}
Is it true that, for all m.p.s. $(X,\mathcal{B},\mu,(T_g)_{g\in G})$ and $Y\in\mathcal{B}$ such that $\mu(Y)>0$, we can choose representatives $(g_k)_{k\in K}$ in $G$ (so $g_kH=k$ for all $k\in K$) such that
\begin{equation*}
\overlim_{N\to\infty}\frac{1}{|F_N|}\sum_{k\in F_N}\mu(Y\cap T_{g_k}^2Y)>0?
\end{equation*}
\end{question}

\subsection{Amenable groups of matrices over fields}
\label{SecMatricesFields}
In this subsection we present a family of non-solvable
amenable groups which have the SAR property; namely, groups of the form $\textup{GL}_n(Q)$, where $Q$ is a countable field which a union of a tower of finite fields and $n\in\mathbb{N}$. 

\begin{lemma}
\label{WhichGLnsAreAmenable}
If $n\geq2$ and $Q$ is a field, the following are equivalent:
\begin{enumerate}[label=(\arabic*)]
\item\label{QUnionFinFields} $Q$ is the increasing union of a sequence of finite fields, $Q=\bigcup_NQ_N$.
\item\label{QAlgOverFp} For some $p\geq2$, $Q$ is isomorphic to an algebraic extension of the finite field $\mathbb{F}_p$ of $p$ elements.
\item\label{GLnQAmenable} The group $G:=\textup{GL}_n(Q)$ is amenable.
\end{enumerate}
\end{lemma}

Examples of amenable groups of the form $\textup{GL}_n(Q)$ were considered in \cite[Page 128]{Dye}.

\begin{proof}
If $\textup{char}(Q)=0$, then the three statements are false, as in that case $G$ contains a copy of $\textup{GL}_2(\mathbb{Z})$, and a classical application of the ping pong lemma implies that the matrices $\begin{pmatrix}
1&2\\0&1
\end{pmatrix},\begin{pmatrix}
1&0\\2&1
\end{pmatrix}$ generate a free subgroup of rank $2$.

Suppose now that $\textup{char}(Q)=:p>0$. Then \ref{QAlgOverFp} implies \ref{QUnionFinFields} because any algebraic extension of $\mathbb{F}_p$ is countable (as there are only countably many polynomials in $\mathbb{F}_p$, each with finitely many roots) and \ref{QUnionFinFields} implies \ref{GLnQAmenable} because if $Q$ is an increasing union of finite fields $\bigcup_NQ_N$, then $G_N=\textup{GL}_2(Q_N)$ is a two-sided F{\o}lner sequence in $G$. 
Finally, let us assume that \ref{QAlgOverFp} is false and prove that $G$ contains a subgroup isomorphic to the free group generated by two elements, so it is not amenable. Let\footnote{The definitions of $A,B,G_A,G_B$ were obtained from a conversation with ChatGPT5.}
\begin{align*}
A&=\{(r_1,r_2)\in Q^2;v(r_1)<v(r_2)\};\quad G_A=\left\{\begin{pmatrix}
1&kt^{-1}\\0&1
\end{pmatrix};k\in\mathbb{F}_p\right\}\\
B&=\{(r_1,r_2)\in Q^2;v(r_1)>v(r_2)\};\quad G_B=\left\{\begin{pmatrix}
1&0\\kt^{-1}&1
\end{pmatrix};k\in\mathbb{F}_p\right\}.
\end{align*}
The subgroups $G_A,G_B$ are isomorphic to $\mathbb{F}_p$, and for any $g_A=\begin{pmatrix}
1&kt^{-1}\\0&1
\end{pmatrix}\in G_A$ with $k\neq0$ and any $\begin{pmatrix}
r_1\\r_2
\end{pmatrix}\in B$, so that $v(r_1)>v(r_2)$, we have $g_A\begin{pmatrix}
r_1\\r_2
\end{pmatrix}=\begin{pmatrix}
r_1+kt^{-1}r_2\\r_2
\end{pmatrix}\in G_A$, as $v(r_1+t^{-1}r_2)=v(r_2)-1<v(r_2)$. Similarly, for any $g_B\in G_B\setminus\{\textup{Id}\}$ and any $v\in A$, we have $g_Bv\in B$. Thus, by the ping-pong lemma (see e.g. \cite[II.24]{Ha}), the subgroups $G_A$ and $G_B$ freely generate a subgroup isomorphic to the free product $G_A*G_B$, which by by \cite[I.1, Proposition 4]{Se} contains a free group of rank $2$ if $p\geq3$. If $p=2$, the same argument works considering instead of $G_A$ the subgroup $G_A'=\left\{\begin{pmatrix}
1&kt^{-1}+lt^{-2}\\0&1
\end{pmatrix};k,l\in\mathbb{F}_p
\right\}$.
\end{proof}

\begin{theorem}
\label{GLnQPosErgAvgsAlongSquares}
If $Q$ can be expressed as an increasing union of finite fields, $Q=\bigcup_{N\in\mathbb{N}}Q_N$, then $\textup{GL}_n(Q)$ has the SAR property.
\end{theorem}

\begin{proof}
Consider the two-sided F{\o}lner sequence $(G_N)=(\textup{GL}_n(Q_N))$. We use the following notation:
\begin{itemize}
    \item Let $\mathfrak{B}_N$ be the family of bases of $(Q_N)^n$ modulo the equivalence relation generated by $(v_1,\dots,v_n)=(w_1,\dots,w_n)$ iff the union of $1$-dimensional subspaces generated by the vectors $v_i$ coincides with the analog for $w_i$. 
    \item For each `basis' $B\in\mathfrak{B}_N$ let $D_N^B\subseteq G$ be the abelian subgroup of matrices which are diagonal in the basis $B$. Thus, $D_N^B\cong (Q_N^\times)^n$ as a group.
    \item For each $B\in\mathfrak{B}_N$ let $E_N^B\subseteq D_N^B$ be the matrices with $n$ distinct eigenvalues; let $E_N=\bigcup_{B\in\mathfrak{B}_N}E_N^B$.
\end{itemize}

For distinct \(B,B'\in\mathfrak{B}_N\), the sets \(E_N^B\) and \(E_N^{B'}\) are disjoint. Hence, $E_N$ is the disjoint union $\bigsqcup_{B\in\mathfrak{B}_N}E_N^B$.
Furthermore,
\[
\lim_{N\to\infty}\inf_{B\in\mathfrak{B}_N}
\frac{|E_N^B|}{|D_N^B|}=1,\textup{ and }
\lim_{N\to\infty}\frac{|E_N|}{|G_N|}
= \frac1{n!}.
\]
To verify the latter equality, first set \(M=|Q_N|\). Then, given \(\varepsilon>0\), for all sufficiently large \(N\), we have
\begin{multline*}
|E_N|=\sum_{B\in\mathfrak{B}_N}\left|E_N^B\right|\geq(1-\varepsilon)\sum_{B\in\mathfrak{B}_N}M^n=
(1-\varepsilon)\frac{\left|\textup{GL}_n(Q_N)\right|}{n!\cdot(M-1)^n}\cdot M^n
\\
=(1-\varepsilon)\frac{|G_N|}{n!(M-1)^n}\cdot M^n\geq(1-\varepsilon)\frac{|G_N|}{n!}.
\end{multline*}
A similar computation yields that, for big enough $N$, $|E_N|\leq(1+\varepsilon)\frac{|G_N|}{n!}$. Thus, by \Cref{ConvSquareAvgsMeasure}, for any m.p.s. $(X,\mathcal{B},\mu,(T_g)_{g\in G})$ and any $A\in\mathcal{B}$,
\begin{align}
\notag\UClim_{g\in G}\mu(A\cap T_g^2A)
&\geq\overlim_{N\to\infty}\frac{1}{n!|E_N|}\sum_{g\in E_N}\mu(A\cap T_g^2A)\\
\notag&\geq\frac{1}{n!}\overlim_{N\to\infty}\inf_{B\in\mathfrak{B}_N}
\frac{1}{|E_N^B|}\sum_{g\in E_N^B}\mu(A\cap T_g^2A)\\
\notag&=\frac{1}{n!}\overlim_{N\to\infty}\inf_{B\in\mathfrak{B}_N}
\frac{1}{|D_N^B|}\sum_{g\in D_N^B}\mu(A\cap T_g^2A)\\
&\geq\frac{\mu(A)^2}{n!}.\label{r9efosdlk}
\end{align}
In the last step we used that for each $B\in\mathfrak{B}_N$, $D_N^B$ is a finite abelian group and $(T_g^2)_{g\in D_N^B}$ is a m.p.s., so $\frac{1}{|D_N^B|}\sum_{g\in D_N^B}\mu(A\cap T_g^2A)\geq\mu(A)^2$ by \Cref{LinearAvgsAreBig}.
\end{proof}

\begin{proof}[Proof of \Cref{BBinAA-1SomeNonSolvableMatrixGroups}]
The amenability of $G=\textup{GL}_n(Q)$ follows from
\Cref{WhichGLnsAreAmenable}, and the bound for the common average is
\eqref{r9efosdlk}. For the combinatorial conclusion, fix $A\subseteq G$
and a left or right F{\o}lner sequence $F=(F_N)$. If $d_l^*(A)=0$, take
$B=\varnothing$. Otherwise, \Cref{ReturnsInDifferences} gives a m.p.s.
and a measurable set $Y$ with $\mu(Y)=d_l^*(A)$ and
$R_\mu^T(Y)\subseteq AA^{-1}$. By \Cref{ConvSquareAvgsMeasure}, the
average along $F$ equals the common average, which is at least
$\mu(Y)^2/n!$ by \eqref{r9efosdlk}. Applying
\Cref{ProdsInRetsMeasurableIntro} to this system and $F$ gives the
required set $B$.
\end{proof}

%% file: S6-Counterexamples.tex
\section{Counterexamples}
\label{SecCounterexamples}
In this section we explain in detail the counterexamples mentioned in 
\Cref{Rmk113,Rmk120}.

We identify the integer Heisenberg group, $\mathrm{UT}_3(\mathbb{Z})$, with $\mathbb{Z}^3$ with the product operation
\begin{equation*}
(x,y,z)\cdot(x',y',z')=(x+x',y+y'+xz',z+z').
\end{equation*}

\begin{theorem}
\label{CounterBBinA-1A}
There exists $A\subseteq \mathrm{UT}_3(\mathbb{Z})$ such that $d_l^*(A)=1$ but, for all $B_1,B_2\subseteq \mathrm{UT}_3(\mathbb{Z})$ such that $d_l^*(B_1),d_l^*(B_2)>0$, $B_1B_2$ is not contained in $A^{-1}A$.
\end{theorem}

\begin{remark}
\label{DensityBFromTheOtherSide}
For the set $A$ from \Cref{CounterBBinA-1A}, we cannot have $B_1B_2\subseteq A^{-1}A$ for sets $B_i$ such that $d_r^*(B_i)>0$, because then we would have $d_l^*(B_i^{-1})>0$ and 
\begin{equation*}
B_2^{-1}B_1^{-1}=(B_1B_2)^{-1}\subseteq (A^{-1}A)^{-1}=A^{-1}A.
\end{equation*}
\end{remark}

Before proving \Cref{CounterBBinA-1A}, we need two lemmas. The first one is nothing else but a variant of the classical Poincaré recurrence lemma.

\begin{lemma}
\label{BigSetsIntersectLeftTranslates}
Let $G$ be a countable amenable group and let $E\subseteq G$ satisfy $d^*_l(E)>0$. Then for any $g\in G$ there exists $k\in\mathbb{N}$ such that $d^*_l(E\cap g^kE)>0$.
\end{lemma}

\begin{proof}
Let $\varepsilon=\frac{1}{2}d^*_l(E)$, $K=\frac{3}{\varepsilon}$, and choose a F{\o}lner sequence $(F_N)$ in $G$ such that $\overline{d}_F(E)\geq\varepsilon$. As $\lim_{N\to\infty}\frac{|F_N\Delta gF_N|}{|F_N|}=0$, for big enough $N$ we must have $|g^kE\cap F_N|=|E\cap g^{-1}F_N|>\frac{\varepsilon}{2}|F_N|$ for all $k=1,\dots,K$. The sets $g^kE\cap F_N$, $k=1,\dots,K$, cannot be pairwise disjoint, in particular there is some constant $\delta=\delta(\varepsilon)>0$ and distinct $1\leq i_N<j_N\leq K$ such that $|g^{i_N}E\cap g^{j_N}E\cap F_N|>\delta|F_N|$. So, if $N$ is big enough, $|g^{i_N-j_N}E\cap E\cap F_N|>\frac{\delta}{2}|F_N|$. Therefore for some value of $k\in\{-K,\cdots,K\}$, we have $|E\cap g^{k}E\cap F_N|>\frac{\delta}{2}|F_N|$ for infinitely many $N$, which implies $d^*(E\cap g^{k}E)>0$.
\end{proof}

\begin{lemma}
\label{MostPointsHaveBigCoords}
For all $N\in\mathbb{N}$, the set $W_N=\{(x,y,z)\in \mathrm{UT}_3(\mathbb{Z});\min(|x|,|y|,|z|)<N\}$ satisfies $d_l^*(W_N)=d_r^*(W_N)=0$.
\end{lemma}

\begin{proof}
If is enough to prove that for all $n\in\mathbb{N}$, the sets $X_n:=\{n\}\times\mathbb{Z}\times\mathbb{Z},Y_n:=\mathbb{Z}\times\{n\}\times\mathbb{Z}$ and $Z_n:=\mathbb{Z}\times\mathbb{Z}\times\{n\}$ have left and right upper Banach density $0$, as $W_N$ is a finite union of such sets. So fix $n\in\mathbb{N}$.

Let $g=(1,0,0)\in \mathrm{UT}_3(\mathbb{Z})$. Then the sets $(g^kX_n)_{k\in\mathbb{N}}$ are pairwise disjoint, and so are the sets $(X_ng^k)_{k\in\mathbb{N}}$, as the points in $g^kX_n$ (or $X_ng^k$) have first coordinate $n+k$. So by \Cref{BigSetsIntersectLeftTranslates} we have $d^*_l(X_n)=d_r^*(X_n)=0$. The proof for $Y_n,Z_n$ is similar.
\end{proof}

Throughout the rest of this section, $[a,b]$ denotes the interval $\{x\in\mathbb{Z};a\leq x\leq b\}$. Moreover, when $x,a\in\mathbb{R}$ and we say `$x$ is close to $a$', we mean that $\frac{|x-a|}{|a|}<10^{-100}$. Similarly, when $\mathbf{x}=(x_1,\dots,x_n),\mathbf{a}=(a_1,\dots,a_n)\in\mathbb{R}^n$ and we say `$\mathbf{x}$ is close to $\mathbf{a}$', we mean that $x_i$ is close to $a_i$ for all $i=1,\dots,n$.

\begin{proof}[Proof of \Cref{CounterBBinA-1A}]
Let $(k_N)_{N\in\mathbb{N}}$ be given by $k_1=1000$ and $k_{N+1}=k_N^{k_N}$, and consider the left F{\o}lner sequence $(A_N)$ given by
\begin{equation*}
A_N=[k_{4N+3},k_{4N+3}+N]\times[k_{4N+2},k_{4N+2}+k_{4N+1}]\times[k_{4N},k_{4N}+N]
\end{equation*}
We define $A=\bigcup_{N\geq10}A_N$, so that $d_l^*(A)=1$. We have 
\begin{equation*}
A^{-1}A=\bigcup_{N,M\geq10}A_N^{-1}A_M.
\end{equation*}
For any $g_N=(x_N,y_N,z_N)\in A_N$ and $g_M=(x_M,y_M,z_M)\in A_M$, we have
\begin{align*}
g_N^{-1}g_M&
=
(-x_N,-y_N+x_Nz_N,-z_N)(x_M,y_M,z_M)\\
&=(x_M-x_N,y_M-y_N-x_N(z_M-z_N),z_M-z_N).
\end{align*}
We have the following three possibilities for the coordinates of $g_N^{-1}g_M$:
\begin{enumerate}[label=(\arabic*)]
    \item\label{M<N} $M<N$. Then $
g_N^{-1}g_M$ is close to $(-k_{4N+3},k_{4N+3}\cdot k_{4N},-k_{4N})$.
\item\label{M>N} $M>N$. Then $
g_N^{-1}g_M$ is close to $(k_{4M+3},k_{4M+2},k_{4M})$.
\item\label{M=N} $M=N$. Let $u=z_M-z_N\in[-N,N]$. Then either $u=0$, so $g_N^{-1}g_M\in\mathbb{Z}\times\mathbb{Z}\times\{0\}$, or $u\neq0$, in which case
\begin{equation*}
g_N^{-1}g_M\in[-N,N]\times[-uk_{4N+3}-2k_{4N+1},-uk_{4N+3}+2k_{4N+1}]\times\{u\}.
\end{equation*}
\end{enumerate}
Given a point $(x,y,z)\in\mathrm{UT}_3(\mathbb{Z})$ such that $z\neq0$, and the knowledge that $(x,y,z)\in A_N^{-1}A_M$ for some $N,M\geq10$, we can determine whether $M<N,M>N$ or $M=N$ just using $|x|$ and $|y|$: if $|x|>|y|$, then $M>N$, if $|x|<|y|<x^2$, then $M<N$, and if $x^2<|y|$, then $M=N$. Similarly, $\max(M,N)$ can be deduced from $|x|,|y|$.

Suppose for the sake of contradiction that there exist sets $B_1,B_2\subseteq \mathrm{UT}_3(\mathbb{Z})$ such that $d_l^*(B_1),d_l^*(B_2)>0$ and $B_1B_2\subseteq A^{-1}A$. By \Cref{BigSetsIntersectLeftTranslates} applied to $B_1\times B_2\subseteq \mathrm{UT}_3(\mathbb{Z})\times \mathrm{UT}_3(\mathbb{Z})$, there is some $k\in\mathbb{N}$ such that $d_l^*((0,0,k)B_1\cap B_1)>0$ and $d_l^*((0,0,k)B_2\cap B_2)>0$. So by \Cref{MostPointsHaveBigCoords}, for $i=2,1$ we may find points $h_i=(a_i,b_i,c_i)\in B_i$ such that $|a_2|,|b_2|,|c_2|>(10k)^{100}$ and $|a_1|,|b_1|,|c_1|>\max(|a_2|,|b_2|,|c_2|)^{100}$, and such that $h_i':=(0,0,k)h_i\in B_i$. We will reach a contradiction using that the following points are in $B_1B_2\subseteq A^{-1}A$:
\begin{align*}
h_1h_2=(a_1,b_1,c_1)(a_2,b_2,c_2)&=(a_1+a_2,b_1+b_2+a_1c_2,c_1+c_2)\\
h_1'h_2=(a_1,b_1,c_1+k)(a_2,b_2,c_2)
&=
(a_1+a_2,b_1+b_2+a_1c_2
,c_1+c_2+k)\\
h_1h_2'=(a_1,b_1,c_1)(a_2,b_2,c_2+k)&=
(a_1+a_2,b_1+b_2+a_1c_2+a_1k,c_1+c_2+k)
\end{align*}
The $z$-coordinates of $h_1h_2,h_1'h_2,h_1h_2'$ are not $0$. Let $N,M\in\mathbb{N}$ be such that, for some $g_N=(x_N,y_N,z_N)\in A_N,g_M=(x_M,y_M,z_M)\in A_M$, we have $h_1h_2=g_N^{-1}g_M\in A_N^{-1}A_M$. That is,
\begin{equation}
\label{39reospd9owerdlks1}
(a_1+a_2,b_1+b_2+a_1c_2,c_1+c_2)=(x_M-x_N,y_M-y_N-x_N(z_M-z_N),z_M-z_N).
\end{equation}
There are three cases:
\begin{enumerate}[label=(\alph*)]
    \item $M<N$. Then, as the first two coordinates of $h_1'h_2$ are equal to those of $h_1h_2$, we must have $h_1'h_2\in A_N^{-1}A_L$ for some $L<N$, or in coordinates, for some $(x_L,y_L,z_L)\in A_L$ and $(x_N',y_N',z_N')\in A_N$ we have
\begin{equation}
\label{39reospd9owerdlks2}
(a_1+a_2,b_1+b_2+a_1c_2,c_1+c_2+k)=(x_L-x_N',y_L-y_N'-x_N'(z_L-z_N'),z_L-z_N').
\end{equation}
So comparing \Cref{39reospd9owerdlks1,39reospd9owerdlks2} we obtain
\begin{align}
x_M-x_N&=x_L-x_N'\label{4rewdsf1}\\
y_M-y_N-x_N(z_M-z_N)&=y_L-y_N'-x_N'(z_L-z_N')\label{4rewdsf2}\\
z_M-z_N+k&=z_L-z_N'\label{4rewdsf3}
\end{align}
Substituting \Cref{4rewdsf3} into \Cref{4rewdsf2}, we obtain 
\begin{equation*}
y_M-y_N-x_N(z_M-z_N)=y_L-y_N'-x_N'(z_M-z_N)-kx_N',
\end{equation*}
and then solving for $kx_N'$ and using \Cref{4rewdsf1}, we obtain
\begin{align}
kx_N'&=y_L-y_N'+(x_N-x_N')(z_M-z_N)-y_M+y_N;\notag\\
kx_N'&=
y_L-y_N'+(x_M-x_L)(z_M-z_N)-y_M+y_N,\label{ContraCase1}
\end{align}
a contradiction because $kx_N'$ is much bigger in absolute value than all the other terms in \Cref{ContraCase1}.

\item $M>N$. Then 
\begin{equation*}
(a_1+a_2,b_1+b_2+a_1c_2,c_1+c_2)\textup{ is close to }(k_{4M+3},k_{4M+2},k_{4M}). 
\end{equation*}
As $|a_1|>\max(|a_2|,|b_2|,|c_2|)^{100}$, we in fact have that $a_1$ is close to $k_{4M+3}$. So, as $b_1+b_2+a_1c_2$ is close to $k_{4M+2}$, we must have that $b_1+b_2+a_1c_2+a_1k$ is close to $k\cdot k_{4M+3}$. That is, the point $(x,y,z):=h_1h_2'\in A_{N'}^{-1}A_M'$ (for some $N',M'$) satisfies that $x$ is close to $k_{4M+3}$ and $y$ is close to $k\cdot k_{4M+3}$. So we have $z\neq0$ and $|x|<|y|<x^2$, which places $h_1h_2'$ in case \ref{M<N} above, contradicting the fact that $x,y$ have the same sign.

\item $M=N$. As $u:=c_1+c_2\neq0$, we must have 
\begin{multline*}
(a_1+a_2,b_1+b_2+a_1c_2,c_1+c_2)\\\in[-N,N]\times[-uk_{4N+3}-2k_{4N+1},-uk_{4N+3}+2k_{4N+1}]\times\{u\},
\end{multline*}
for some $N\geq|u|$. As the first two coordinates of $h_1'h_2$ coincide with those of $h_1h_2$, the element $h_1'h_2$ also is in $A^{-1}_NA_N$, that is,
\begin{multline*}
(a_1+a_2,b_1+b_2+a_1c_2,c_1+c_2+k)\\\in[-N,N]\times[-(u+k)k_{4N+3}-2k_{4N+1},-(u+k)k_{4N+3}+2k_{4N+1}]\times\{u+k\}.
\end{multline*}
This is impossible, because the two intervals centered at $-uk_{4N+3}$ and $-(u+k)k_{4N+3}$ and with radius $2k_{4N+1}$ are disjoint, so $b_1+b_2+a_1c_2$ cannot be in both of them at the same time.\qedhere
\end{enumerate}
\end{proof}

\begin{theorem}
\label{CounterBxBinA-1A}
There is a set $A\subseteq \mathrm{UT}_3(\mathbb{Z})\times \mathrm{UT}_3(\mathbb{Z})$ such that $d^*_l(A)=1$ and we do not have $B_1\times B_2\subseteq A^{-1}A$ for any $B_1,B_2\subseteq \mathrm{UT}_3(\mathbb{Z})$ such that $d^*_l(B_i)>0$.
\end{theorem}

\begin{proof}
Let $k_N,(A_N)$ be as in the proof of \Cref{CounterBBinA-1A} and let $A=\bigcup_{N\geq10}A_{2N}\times A_{2N+1}$. Thus, 
\begin{align*}
A^{-1}A&={\textstyle\bigcup_{N,M\geq10}}(A_{2N}\times A_{2N+1})^{-1}(A_{2M}\times A_{2M+1})\\
&={\textstyle\bigcup_{N,M\geq10}}\left(A_{2N}^{-1}A_{2M}\right)\times\left(A_{2N+1}^{-1}A_{2M+1}\right).
\end{align*}
Let $(g_1,g_2)=((x_1,y_1,z_1),(x_2,y_2,z_2))\in A^{-1}A$. According to the values of $N,M$ such that $(g_1,g_2)\in\left(A_{2N}^{-1}A_{2M}\right)\times\left(A_{2N+1}^{-1}A_{2M+1}\right)$, we have three cases (see cases \ref{M<N},\ref{M>N},\ref{M=N} in \Cref{CounterBBinA-1A}):
\begin{enumerate}[label=(\arabic*)]
    \item $M<N$. Then $(y_1,y_2)$ is close to $(k_{8N+3}\cdot k_{8N},k_{8N+7}\cdot k_{8N+4})$.
    \item $M>N$. Then $(y_1,y_2)$ is close to $(k_{8M+3},k_{8M+7})$.
    \item $M=N$. Then, if $z_1,z_2\neq0$, $(y_1,y_2)$ is close to $(-z_1k_{4N+3},-z_2k_{4N+7})$.
\end{enumerate}
In particular, if $z_1,z_2\neq0$, then for each fixed value of $y_1$, the possible values of $y_2$ are bounded; more concretely, letting $L\in\mathbb{N}$ be such that $k_L\leq |y_1|<k_{L+1}$, we cannot have $|y_2|>k_{L+10}$. But suppose we have $B_1,B_2\subseteq \mathrm{UT}_3(\mathbb{Z})$ such that $d^*_l(B_i)>0$ and $B_1\times B_2\subseteq A^{-1}A$. By \Cref{MostPointsHaveBigCoords} we can find $(x_1,y_1,z_1)\in B_1$ such that $z_1\neq0$, and points $(x_2,y_2,z_2)\in B_2$ such that $z_2\neq0$ and $y_2$ is arbitrarily large, a contradiction.
\end{proof}

Finally, we prove the following result mentioned in \Cref{SecIntro}.

\begin{prop}
\label{UpperAndLowerUBDDiffer}
If a countable amenable group $G$ has an element $a$ with infinite conjugacy class, then there exists a set $A\subseteq G$ such that $d^*_l(A)=1$ but $d^*_r(A)\leq\frac{1}{2}$.
\end{prop}

\begin{proof}
The centralizer of $a$, $C_a:=\{g\in G;a=gag^{-1}\}$, is a subgroup of infinite index. This means that there are infinitely many disjoint left/right cosets of $C_a$, so $d^*_l(C_a)=d^*_r(C_a)=0$ (see the proof of \Cref{BigSetsIntersectLeftTranslates}). So by \Cref{FolnerProps}, for all $g\in G$ we have $d^*_l(C_ag)=d^*_r(C_ag)=0$. 

Now let $(F_N)$ be a left F{\o}lner sequence in $G$. For each $N\in\mathbb{N}$, the set $E_N:=\{g\in G;F_Ng_N\cap F_Ng_Na\neq\varnothing\}$ satisfies $d^*_l(E)=d^*_r(E)=0$. This is because 
\begin{equation*}
E={\textstyle\bigcup_{x,y\in F_N}}\{g\in G;xg=yga\}
={\textstyle\bigcup_{x,y\in F_N}}\{g\in G;y^{-1}x=gag^{-1}\},
\end{equation*}
and each of the sets $\{g\in G;y^{-1}x=gag^{-1}\}$ is either empty or a left coset of $C_a$ with left and right upper Banach density $0$.
This means that we can construct by recursion elements $g_N$, $N\in\mathbb{N}$, such that $F_Ng_N\cap F_Mg_Ma=\varnothing$ for all $M,N\in\mathbb{N}$; it suffices to choose each $g_N$ outside $E_N\cup\bigcup_{n<N}F_ng_nF_N^{-1}$. Then, the set $A=\bigcup_NF_Ng_N$ satisfies $d^*_l(A)=1$, as $d_{(F_Ng_N)}(A)=1$ and $(F_Ng_N)$ is a left F{\o}lner sequence, but $d^*_r(A)\leq\frac{1}{2}$ because $A\cap Aa=\varnothing$, so $A$ cannot have density $>\frac{1}{2}$ in any right F{\o}lner sequence.
\end{proof}

%% file: S7-Questions.tex
\section{Questions/future directions}
\label{SecQuests}

In this section we collect and discuss some questions which are motivated by the results of this paper. We first recall our main question:
\BigQuest*
\begin{remark}
Functions of the form $g\mapsto\mu(Y\cap T_g^2Y)$ need not be positive definite\footnote{Recall that if $G$ is a countable group, then a function $f:G\to\mathbb{C}$ is  positive definite if and only if it can be expressed as $f(g)=\langle\xi,U_g\xi\rangle$, where $(U_g)_{g\in G}$ is a unitary representation of $G$ on a Hilbert space $\mathcal{H}$ and $\xi\in\mathcal{H}$ (see for example \cite[Section 3.3]{Fol}). In particular, for any m.p.s. $(X,\mathcal{B},\mu,(T_g)_{g\in G})$ and $Y\in\mathcal{B}$, the function $g\mapsto\mu(Y\cap T_gY)$ is positive definite.} To see why, first note that if $f:G\to\mathbb{C}$ is a positive definite function and some element $h$ satisfies $f(s)=f(1_G)$, then we must have $f(gs)=f(g)$ for all $g\in G$.\footnote{If $f$ is expressed as $g\mapsto\langle\xi,U_g\xi\rangle$, then $f(s)=f(1_G)$ implies $\langle \xi,U_s\xi\rangle=\|\xi\|^2$, so $U_s\xi=\xi$, so $U_{gs}\xi=U_g\xi$ for all $g\in G$, so $f(g)=f(gs)$ for all $g\in G$.}
Let $\mathbf{D}=\{z\in\mathbb{C};|z|\leq1\}$ be the unit disk equipped with normalized Lebesgue measure $\mu$. Consider the action of the dihedral group $D_\infty=\langle r,s|s^2=1,rs=sr^{-1}\rangle$ on $\mathbf{D}$ given by $rz=iz,sz=\overline{z}$. Then for $A=\{z\in\mathbf{D};\textup{Re}(z)\geq0\}$, the function $f:D_\infty\to\mathbb{C}$; $f(g)=\mu(A\cap T_g^2A)$ is not positive definite. Indeed, the element $s$ satisfies $f(s)=f(1_G)=\frac{1}{2}$, however $f(r)=0\neq\frac{1}{2}=f(rs)$.

\end{remark}

\begin{definition}
For each $a\in[0,1]$, 
let $f_{\overline{d}}(a)$ be the supremum of the set of values $b\in[0,1]$ such that, for all $A\subseteq\mathbb{Z}$ such that $d^*(A)\geq a$, there is $B\subseteq\mathbb{Z}$ such that $\overline{d}(B)\geq b$ and $B+B\subseteq A-A$.

If $\mu$ is the usual Lebesgue measure on the torus $\mathbb{T}=\mathbb{R}/\mathbb{Z}$, then let $f_{\mu}(a)$ be the supremum of the set of values $b\in[0,1]$ such that, for all measurable $A\subseteq\mathbb{T}$ such that $\mu(A)\geq a$, there is some measurable $B\subseteq\mathbb{T}$ such that $\mu(B)\geq b$ and $B+B\subseteq A-A$.
\end{definition}

\begin{question}
What is the asymptotic growth rate of $f_{\overline{d}}(a)$ when $a\to0$? What about $f_{\mu}(a)$?
\end{question}

\begin{remark}
\label{OptimalBForB+BinA-A?}
One can show that $0\leq a<\frac{1}{2}$ we have $a^2\leq f_{\overline{d}}(a)\leq a$. The inequality $a^2\leq f_{\overline{d}}(a)$ follows from \Cref{BxBinAA^-1}. One can check that $f_{\overline{d}}(a)\leq a$ by letting $A=\{n\in\mathbb{Z};n\pi\in(0,a)\textup{ mod }1\}$ and using the theorem of Kneser stating that, for each open subset $X,Y$ of the torus $\mathbb{T}=\mathbb{R}/\mathbb{Z}$, we have $\mu(X+Y)\geq\min(1,\mu(X)+\mu(Y))$ (see \cite[Theorem 1]{Kn}). So, if $B\subseteq\mathbb{Z}$ satisfies $B+B\subseteq A-A$, then for all $\varepsilon>0$ the $\varepsilon$-neighborhood of $\{\pi b\textup{ mod }1;b\in B\}\subseteq\mathbb{T}$ must have measure $\leq a+2\varepsilon$, which implies that $\overline{d}(B)\leq a$.
\end{remark}

We conclude this section (and the paper) with a question which deals with a rather natural analogue of \Cref{B+BinA-AinZ} in $\mathbb{R}$.
\begin{question}
\label{QuestionUncountableB+BinA-A}
Call a set $A\subseteq\mathbb{R}$ large if $\nu(A)>0$ for some translation invariant mean $\nu$ in $(\mathbb{R},\mathcal{P}(\mathbb{R}))$. Is it true that if $A\subseteq\mathbb{R}$ is large, then for some large set $B\subseteq\mathbb{R}$ one has $B+B\subseteq A-A$?
\end{question}

%% file: S8-Appendix.tex
\appendix
\section{Convergence of Cesàro averages of squares of unitary operators}
\label{SecSquareAveragesAppendix}
This appendix was written with the help of AI; see the AI statement at the end of the paper.

Recall that the von Neumann ergodic theorem for amenable groups states
that, for any unitary representation $(U_g)_{g\in G}$ of a countable
amenable group $G$ on a Hilbert space $\mathcal{H}$ and any right
F{\o}lner sequence $(F_N)$ in $G$,
\begin{equation}
\label{Eq:VNET}
    \frac{1}{|F_N|}\sum_{g\in F_N}U_gv\longrightarrow Pv
    \qquad\text{in norm for every }v\in\mathcal{H},
\end{equation}
where $P$ is the orthogonal projection onto the subspace of
$G$-invariant vectors.

The main purpose of this appendix is to establish an analogous result for averages along the squares $(U_g^2)_{g\in G}$.
\begin{theorem}
\label{NormConvergenceSquareAverages}
Let $G$ be a countable amenable group and let $(U_g)_{g\in G}$ be a
unitary representation of $G$ on a Hilbert space $\mathcal{H}$. For every
$v\in\mathcal{H}$ and every right F{\o}lner sequence $(F_N)$ in $G$, the
averages
\begin{equation}
\label{Eq:NormConvergenceSquareAverages}
    \frac{1}{|F_N|}\sum_{g\in F_N}U_g^2v
\end{equation}
converge in norm, and the limit is the same for all right
F{\o}lner sequences.
\end{theorem}

For left F{\o}lner sequences, norm convergence can fail even for the
averages of $U_gv$. Indeed, let $G$ be the group of finitely supported
permutations of $\mathbb{N}$, acting on $\ell^2(\mathbb{N})$ by
$U_g\delta_j=\delta_{g(j)}$, and let $K_N$ be the subgroup of
permutations supported on $\{1,\ldots,N\}$. With the usual notation $(a_1,\dots,a_n)$ for cycles, the sets
$F_N=K_N\cdot(1,N+1)$ form a left F{\o}lner sequence, and
\begin{equation}
\label{erw9iofsdkl}
    \frac{1}{|F_N|}\sum_{g\in F_N}U_g\delta_1=\delta_{N+1}.
\end{equation}
Similarly, the F{\o}lner sequence
$F_N'=K_N\cdot(1,N+1,N+2)$ satisfies
\begin{equation}
\label{erw9iofsdkl2}
    \frac{1}{|F_N'|}\sum_{g\in F_N'}U_g^2\delta_1=\delta_{N+2}.
\end{equation}
Clearly, neither of the expressions in \ref{erw9iofsdkl} or \ref{erw9iofsdkl2} converge in norm.
However, the weak convergence of the expression in \ref{Eq:NormConvergenceSquareAverages} (as well as in \ref{Eq:VNET}) takes place for both left and right F{\o}lner sequences:

\begin{cor}
\label{ScalarConvergenceSquareAverages}
Let $G$ be a countable amenable group and let $(U_g)_{g\in G}$ be a
unitary representation of $G$ on a Hilbert space $\mathcal{H}$. For every
$v,w\in\mathcal{H}$, the limit
\[
    \lim_{N\to\infty}\frac{1}{|F_N|}\sum_{g\in F_N}
    \langle w,U_g^2v\rangle
\]
exists and is the same for all left and right F{\o}lner sequences $(F_N)$ in $G$.
\end{cor}

\begin{proof}
For right F{\o}lner sequences, this follows from
\Cref{NormConvergenceSquareAverages} and continuity of the inner product.
If $(F_N)$ is left F{\o}lner, then $(F_N^{-1})$ is right F{\o}lner, and
\[
    \frac{1}{|F_N|}\sum_{g\in F_N}\langle w,U_g^2v\rangle
    =\frac{1}{|F_N^{-1}|}\sum_{h\in F_N^{-1}}
    \langle U_h^2w,v \rangle,
\]
so by \Cref{NormConvergenceSquareAverages} applied to the right F{\o}lner sequence $\left(F_N^{-1}\right)$, the limit exists and is the same in all left-F{\o}lner sequences. Since $G$ admits
a two-sided F{\o}lner sequence, the left and right limits agree.
\end{proof}

We reduce \Cref{NormConvergenceSquareAverages} to the case of irreducible unitary representations using the following
lemma.

\begin{lemma}[{\cite[Example $e$ in page 33 and Theorem 4.1]{Cho56}}]
\label{GeneralHerglotzTheorem}
Let $G$ be a countable group, and let $\varphi:G\to\mathbb{C}$ be a
positive definite function with $\varphi(1_G)=1$. Then we can express
$\varphi$ as
\begin{equation}
\label{rewf09dscol}
\varphi(g)=\int_X\langle\xi_x,U_{x,g}\xi_x\rangle\,d\mu(x),
\end{equation}
where
\begin{enumerate}
    \item $(X,\mathcal{B},\mu)$ is a Lebesgue probability space.
    \item For each $x\in X$, $(U_{x,g})_{g\in G}$ is an irreducible
    unitary action of $G$ on a Hilbert space $\mathcal{H}_x$, and
    $\xi_x\in\mathcal{H}_x$ with $\|\xi_x\|=1$.
    \item For each $g\in G$, the function
    $x\mapsto\langle\xi_x,U_{x,g}\xi_x\rangle$ is measurable.
\end{enumerate}
\end{lemma}

\begin{proof}
Let $\mathcal{P}_1$ be the family of positive definite functions
$f:G\to\mathbb{C}$ with $f(1_G)=1$, considered as a subset of
$\mathbb{C}^G$ with the product topology. Then $\mathcal{P}_1$ is convex,
metrizable and compact, since it is a closed subset of the countable
product $\mathbb{D}^G$ of the closed unit disk.

Let $X\subseteq\mathcal{P}_1$ be the set of extreme points of
$\mathcal{P}_1$, and let $d$ be a metric generating the topology of
$\mathcal{P}_1$. We have
\[
    X=\mathcal{P}_1\setminus\bigcup_{n\geq1}
    \left\{\frac{f+h}{2}:f,h\in\mathcal{P}_1,
    \ d(f,h)\geq\frac1n\right\}.
\]
Each set in the union is compact, being the continuous image under the
midpoint map of a closed subset of $\mathcal{P}_1^2$. Hence $X$ is a
$G_\delta$ subset of the compact metrizable space $\mathcal{P}_1$.
In particular, $X$ with its Borel $\sigma$-algebra $\mathcal{B}$ is a
standard Borel space.

By a result of Segal \cite{Seg} (see \cite[Theorem 3.25]{Fol}),
each extreme point $x\in X$ can be expressed as
$x(g)=\langle\xi_x,U_{x,g}\xi_x\rangle$, where $(U_{x,g})_{g\in G}$
is an irreducible unitary representation on a Hilbert space
$\mathcal{H}_x$ and $\xi_x\in\mathcal{H}_x$ is a unit vector.

Applying Choquet's theorem \cite[p.~14]{Phelps}\footnote{In the metrizable case, Choquet's
theorem states that if $K$ is a compact metrizable convex subset of a
Hausdorff locally convex real vector space $V$, then every $z\in K$ admits
a Borel probability measure $\nu$ on the set of extreme points of $K$
such that $F(z)=\int F(x)\,d\nu(x)$ for every continuous linear functional on $V$.}
to the compact convex set $\mathcal{P}_1$,
we obtain a Borel probability measure $\mu$ on $X$ such that
$\varphi=\int_X x\,d\mu(x)$. That is, for every continuous linear
functional $F:\mathbb{C}^G\to\mathbb{C}$,
\[
    F(\varphi)=\int_X F(x)\,d\mu(x).
\]
For each $g\in G$, the coordinate functional $x\mapsto x(g)$ is
continuous. Its restriction to $X$ is therefore measurable, and
\[
    \varphi(g)=\int_X x(g)\,d\mu(x)
    =\int_X\langle\xi_x,U_{x,g}\xi_x\rangle\,d\mu(x).\qedhere
\]
\end{proof}

We will deduce \Cref{NormConvergenceSquareAverages} from the following
special case.

\begin{theorem}
\label{IrreducibleNormConvergenceSquareAverages}
Let $G$ be a countable amenable group, let $(U_g)_{g\in G}$ be an
irreducible unitary representation of $G$ on a Hilbert space
$\mathcal{H}$, and let $(F_N)$ be a right F{\o}lner sequence in $G$.
\begin{enumerate}[label=\textup{(\roman*)}]
    \item\label{IrreducibleSquareAveragesInfiniteDimensional}
    If $\mathcal{H}$ is infinite-dimensional, then for every $v\in\mathcal{H}$,
    \[
        \lim_{N\to\infty}\left\|\frac{1}{|F_N|}\sum_{g\in F_N}U_g^2v\right\|=0.
    \]
    \item\label{IrreducibleSquareAveragesFiniteDimensional}
    If $\mathcal{H}$ is finite-dimensional, give the unitary group $\mathrm{U}(\mathcal{H})$ the operator norm topology. Then we have
    \[
        \lim_{N\to\infty}\frac{1}{|F_N|}\sum_{g\in F_N}U_g^2
        =\int_K V^2\,dm_K(V),
    \]
    where $K=\overline{\{U_g;g\in G\}}\subseteq\mathrm{U}(\mathcal{H})$, and $m_K$ is normalized
    Haar measure on $K$.
\end{enumerate}
\end{theorem}

\begin{remark}
\label{Rmk:riewfdosckl}
As we will see in \Cref{SubsecHigherPowerAveragesCounterexample}, the convergence in \Cref{IrreducibleNormConvergenceSquareAverages}.\ref{IrreducibleSquareAveragesInfiniteDimensional} fails for averages of $U_g^3$. On the other hand, in the finite dimensional case, the proof we provide for \Cref{IrreducibleNormConvergenceSquareAverages}.\ref{IrreducibleSquareAveragesFiniteDimensional} does work for $U_g^3$. More generally, the same proof shows that  for any Banach space $B$ and any vector-valued continuous function $f:\mathrm{U}(\mathcal{H})\to B$, we have 
\[
        \lim_{N\to\infty}\frac{1}{|F_N|}\sum_{g\in F_N}f(U_g)
        =\int_Kf(V)\,dm_K(V).
    \]
\end{remark}
\begin{proof}[Proof of \Cref{NormConvergenceSquareAverages} from \Cref{IrreducibleNormConvergenceSquareAverages}]
We may assume
$\|v\|=1$. Let $\mu,U_{x,g},\xi_x$ be as in
\Cref{GeneralHerglotzTheorem}, applied to the positive definite function
$\varphi(g)=\langle v,U_gv\rangle$. Fix a right F{\o}lner sequence
$(F_N)$, and put
\[
    v_N=\frac1{|F_N|}\sum_{g\in F_N}U_g^2v,
    \qquad
    \xi_{x,N}=\frac1{|F_N|}\sum_{g\in F_N}U_{x,g}^2\xi_x.
\]
For every $g,h\in G$, \Cref{rewf09dscol} gives
\begin{equation}
\label{4rwedsc9oilk}
    \langle U_g^2v,U_h^2v\rangle
    =\varphi(g^{-2}h^2)
    =\int_X\langle U_{x,g}^2\xi_x,U_{x,h}^2\xi_x\rangle\,d\mu(x).
\end{equation}
\Cref{4rwedsc9oilk} then yields, for all $N,M\in\mathbb{N}$,
\begin{align}
\label{SquareAveragesCauchyIntegral}
    \|v_N-v_M\|^2
    &=
    \langle v_N-v_M,v_N-v_M\rangle\\
    \notag&=\int_X\langle
    \xi_{x,N}-\xi_{x,M},\xi_{x,N}-\xi_{x,M}\rangle\,d\mu(x)\\
    \notag&=
    \int_X\|\xi_{x,N}-\xi_{x,M}\|^2\,d\mu(x).
\end{align}

By \Cref{IrreducibleNormConvergenceSquareAverages}, $(\xi_{x,N})_N$ converges in
$\mathcal{H}_x$ for every $x\in X$. Thus
$\|\xi_{x,N}-\xi_{x,M}\|^2\to0$ as $N,M\to\infty$.
Moreover, $\|\xi_{x,N}\|\leq1$, so dominated convergence applied to \Cref{SquareAveragesCauchyIntegral} gives
\[
    \|v_N-v_M\|^2\longrightarrow0
    \qquad\text{as }N,M\to\infty.
\]
Hence $(v_N)$ is Cauchy and therefore converges in $\mathcal{H}$.

We have thus proved that the averages from \Cref{Eq:NormConvergenceSquareAverages} converge for all right F{\o}lner sequences. If the limit was different for two sequences $(F_N)$ and $(G_N)$, then the limit along the sequence $F_1,G_1,F_2,G_2,\dots$ would not exist, a contradiction.
\end{proof}

It remains to prove \Cref{IrreducibleNormConvergenceSquareAverages}. We start with the finite-dimensional case.

\begin{proof}[Proof of \Cref{IrreducibleNormConvergenceSquareAverages}\ref{IrreducibleSquareAveragesFiniteDimensional}]
Identify $\mathcal{H}$ with $\mathbb{C}^n$, so that $\mathrm{U}(\mathcal{H})$ has the topology inherited from the matrix space $\mathbb{C}^{n^2}$.
For each $N$, consider the probability measure
\[
    \nu_N=\frac1{|F_N|}\sum_{g\in F_N}\delta_{U_g}
\]
on $K$. We claim that $\nu_N\xrightarrow{N\to\infty}m_K$ weakly. Indeed, for
every $f\in C(K)$ and $h\in G$, right F{\o}lner invariance gives
\[
    \left|\int_K f(VU_h)\,d\nu_N(V)-\int_K f(V)\,d\nu_N(V)\right|
    \leq\|f\|_\infty\frac{|F_Nh\mathbin\triangle F_N|}{|F_N|}
    \xrightarrow{N\to\infty}0.
\]
Thus, for every weak limit point $\nu$ of $(\nu_N)$ and every $U\in\{U_g;g\in G\}$,
\begin{equation}
\label{Eq:4rweodslk}
\int_K f(VU_g)\,d\nu(V)=\int_K f(V)d\nu(V).
\end{equation}
As $\{U_g;g\in G\}$ is dense in $K$, \Cref{Eq:4rweodslk} holds for all $U\in K$, so 
$\nu=m_K$. Since the space of probability measures on the compact metric
space $K$ is compact for weak convergence, it follows that
$\nu_N\to m_K$.

The matrix-valued map $V\mapsto V^2$ is continuous on $K$, so the weak convergence $\nu_N\to m_K$ implies that
\begin{equation*}
\frac1{|F_N|}\sum_{g\in F_N}U_g^2
    =\int_K V^2\,d\nu_N(V)
    \longrightarrow\int_K V^2\,dm_K(V).\qedhere
\end{equation*}
\end{proof}

Before proving the infinite-dimensional case of \Cref{IrreducibleNormConvergenceSquareAverages}, we briefly review the compact-weakly mixing decomposition, which in its modern form goes back to Godement \cite{God}. Let $U=(U_g)_{g\in G}$ be a unitary action  of a countable amenable group $G$ on a Hilbert space $\mathcal{H}$.  Then, $\mathcal{H}$ decomposes as an orthogonal sum of two $U$-invariant subspaces, $\mathcal{H}=\mathcal{H}_{\textup{c}}\oplus \mathcal{H}_{\textup{w.m.}}$, where:
\begin{itemize}
    \item $\mathcal{H}_{\textup{c}}$ decomposes into a direct sum of finite-dimensional invariant subspaces.
    \item For any $v\in \mathcal{H}_{\textup{w.m.}}$, $w\in \mathcal{H}$ and $\delta>0$, the set 
    \begin{equation*}
    W_{v,w}^\delta:=\{g\in G;\left|\langle w,U_gv\rangle\right|>\delta\}
    \end{equation*}
     satisfies $d_l^*(W_{v,w}^\delta)=d_r^*(W_{v,w}^\delta)=0$.
\end{itemize}
For more details see \cite[Theorems 2.24,2.20,D.17]{KL16}.

\begin{proof}[Proof of \Cref{IrreducibleNormConvergenceSquareAverages}\ref{IrreducibleSquareAveragesInfiniteDimensional}]
We will need two lemmas, which hold for any countable amenable group $G$ and any
irreducible unitary representation $(U_g)_{g\in G}$ of $G$ on an infinite-dimensional Hilbert space $\mathcal{H}$.
\begin{lemma}
\label{AlmostOrthogonalOrbitVectors}
For every $v\in\mathcal{H}$, 
$r\in\mathbb{N}$ and $\varepsilon>0$, there exist $h_1,\ldots,h_r\in G$
such that
\begin{equation}
\label{Eq:AlmostOrthogonalOrbitVectors}
    |\langle U_{h_i}v,U_{h_j}v\rangle|<\varepsilon\textup{ for all }i\neq j.
\end{equation}
\end{lemma}

\begin{proof} 
The unitary representation $(U_g)_{g\in G}$ is weakly mixing (i.e. $\mathcal{H}=\mathcal{H}_{\textup{w.m.}}$), as it has no finite-dimensional invariant subspaces.
We can therefore obtain the elements $h_1,\ldots,h_r$ by recursion, as given $l\in\mathbb{N}$ and any elements $h_1,\dots,h_l\in G$, 
the set 
\begin{equation*}
\{g\in G;\left|\langle U_{h_i}v,U_gv\rangle\right|>\varepsilon\textup{ for all }i=1,\dots,l\}
\end{equation*}
has upper Banach density $0$, so we can choose $h_{l+1}$ outside of it.
\end{proof}

\begin{lemma}[Conjugation averages]
\label{ScalarConjugationLimitPoints}
Let $(F_N)$ be a right F{\o}lner sequence in $G$. For every bounded
operator $B$ on $\mathcal{H}$, put
\[
    C_N(B)=\frac1{|F_N|}\sum_{g\in F_N}U_g^*BU_g.
\]
Every subsequence of $(C_N(B))_N$ has a further subsequence converging
in the weak operator topology to a scalar operator $cI$, where
$|c|\leq\|B\|$.
\end{lemma}

\begin{proof}
Right F{\o}lner invariance gives
\begin{equation}
\label{Eq:ConjugationAveragesAlmostInvariant}
    \|U_h^*C_N(B)U_h-C_N(B)\|
    \leq\|B\|\frac{|F_Nh\mathbin\triangle F_N|}{|F_N|}
    \xrightarrow{N\to\infty}0
    \qquad(h\in G).
\end{equation}
Thus every weak operator limit point of $(C_N(B))_N$ commutes with all
$U_h$. By Schur's lemma \cite[Theorem~3.5(a)]{Fol}, every such limit
point is a scalar $cI$, with $|c|\leq\|B\|$.
These limit points exist because $\|C_N(B)\|\leq\|B\|$ and bounded
sets of operators have compact closure in the weak operator topology.
We may use subsequences because, as $\mathcal{H}$ is
separable, the weak operator topology is metrizable on bounded sets.
\end{proof}

We now conclude the proof of \Cref{IrreducibleNormConvergenceSquareAverages}\ref{IrreducibleSquareAveragesInfiniteDimensional}.
Suppose $\|v\|=1$. Fix an integer $r\geq1$ and $\varepsilon>0$, and
choose $h_1,\ldots,h_r$ as in \Cref{AlmostOrthogonalOrbitVectors}.
With $C_N$ as in \Cref{ScalarConjugationLimitPoints}, we have 
{\allowdisplaybreaks[3]
\begin{align}
\label{re9fsdoilkefds1}
\overlim_{N\to\infty}
        \left\|\frac1{|F_N|}\sum_{g\in F_N}U_g^2v\right\|^2
    &=\overlim_{N\to\infty}
        \left\|\frac1r\sum_{i=1}^r 
        \frac1{|F_N|}\sum_{g\in F_N}
        U_{gh_i}^2v\right\|^2\\
\label{re9fsdoilkefds2}&\leq\overlim_{N\to\infty}\frac1{|F_N|}\sum_{g\in F_N}
        \left\|\frac1r\sum_{i=1}^r U_{gh_i}^2v\right\|^2\\
     \notag &\leq\frac1{r^2}\sum_{i,j=1}^r\overlim_{N\to\infty}
        \left|\frac1{|F_N|}\sum_{g\in F_N}
        \langle U_{gh_i}^2v,U_{gh_j}^2v\rangle\right|\\
     \notag &=\frac1{r^2}\sum_{i,j=1}^r\overlim_{N\to\infty}
        \left|\frac1{|F_N|}\sum_{g\in F_N}
        \langle U_{h_i}v,U_g^*U_{h_i}^*U_{h_j}U_gU_{h_j}v\rangle\right|\\
     \notag &=\frac1{r^2}\sum_{i,j=1}^r\overlim_{N\to\infty}
        \left|\langle U_{h_i}v,C_N(U_{h_i}^*U_{h_j})U_{h_j}v\rangle\right|\\
\label{re9fsdoilkefds3}&\leq\frac{1}{r^2}\left(r+r(r-1)\varepsilon\right)
    =\frac1r+\frac{r-1}{r}\varepsilon.
\end{align}}
\Cref{re9fsdoilkefds1} uses right invariance of $(F_N)$, and \Cref{re9fsdoilkefds2} uses convexity of the squared norm.
We now explain the inequality in \ref{re9fsdoilkefds3}. Let $i,j\in\{1,\dots,r\}$, $i\neq j$. By
\Cref{ScalarConjugationLimitPoints}, every weak operator limit point
of $C_N(U_{h_i}^*U_{h_j})$ is $cI$ with $|c|\leq1$, so
\[
\overlim_{N\to\infty}\left|\langle U_{h_i}v,C_N(U_{h_i}^*U_{h_j})U_{h_j}v\rangle\right|
\leq\left|\langle U_{h_i}v,U_{h_j}v\rangle\right|\leq\varepsilon
\]
Letting $\varepsilon\to0$ and $r\to\infty$ finishes the proof.
\end{proof}

\subsection{Failure of convergence for higher powers}
\label{SubsecHigherPowerAveragesCounterexample}

We now show that \Cref{ScalarConvergenceSquareAverages}, and therefore, \Cref{NormConvergenceSquareAverages}, do not hold
with $n$-th powers in place of squares, for every integer $n\geq3$. That is, averages of the form $\frac{1}{|F_N|}\sum_{g\in F_N}\langle\xi,U_{g^n}\xi\rangle$, where $(F_N)$ is a double-sided F{\o}lner sequence, need not 
converge. The unitary representation $(U_g)_{g\in G}$ we construct below is infinite-dimensional 
and irreducible; a finite-dimensional example cannot exist because the proof of \Cref{IrreducibleNormConvergenceSquareAverages}\ref{IrreducibleSquareAveragesFiniteDimensional} still works for exponents $n\geq3$, see \Cref{Rmk:riewfdosckl}.
\begin{theorem}
\label{HigherPowerAveragesCounterexample}
For every integer $n\geq3$, there exist a countable amenable group $G$, a unitary representation
$(U_g)_{g\in G}$ on a Hilbert space $\mathcal{H}$, and a unit vector
$\xi\in\mathcal{H}$ such that, if
\[
    A_0=\{g\in G:U_{g^n}\xi=\xi\},
    \qquad
    A_1=\{g\in G:\langle\xi,U_{g^n}\xi\rangle=0\},
\]
then $A_0$ and $A_1$ form a partition of $G$ and there are two-sided F{\o}lner sequences
$F^0=(F_N^0)_{N\in\mathbb{N}}$ and $F^1=(F_N^1)_{N\in\mathbb{N}}$ in $G$
satisfying
\[
    d_{F^0}(A_0)=d_{F^1}(A_1)=1.
\]
\end{theorem}
Note that interleaving the F{\o}lner sequences $F^0$ and $F^1$ from \Cref{HigherPowerAveragesCounterexample} gives
a two-sided F{\o}lner sequence along which the averages of
$\langle\xi,U_{g^n}\xi\rangle$ do not converge.

We will need the following lemma. For any set $X$, let $\operatorname{Sym}(X)$ denote the group of permutations of $X$.

\begin{lemma}
\label{FinitePermutationPowerExtension}
For every integer $n\geq3$ and every finite set $Y$, there exist a finite set $\widetilde Y$
containing $Y$, an embedding
$\operatorname{Sym}(Y)\hookrightarrow\operatorname{Sym}(\widetilde Y)$
whose restriction to $Y$ is the original action, and an element
$t\in\operatorname{Sym}(\widetilde Y)$ such that, identifying each
$h\in\operatorname{Sym}(Y)$ with its image,
\[
    (ht)^ny=y
    \textup{ for all }h\in\operatorname{Sym}(Y),\ y\in Y.
\]
\end{lemma}

\begin{proof}
We first treat the case $n=3$. Put $K=\operatorname{Sym}(Y)$. Introduce distinct new points
$a_{y,k}$ and $b_{y,k}$ for $y\in Y$ and $k\in K$, and let
\[
    \widetilde Y
    =Y\sqcup\{a_{y,k}:y\in Y,\ k\in K\}
       \sqcup\{b_{y,k}:y\in Y,\ k\in K\}.
\]
Extend the action of $K$ by
\[
    h a_{y,k}=a_{y,hk},
    \qquad h b_{y,k}=b_{y,hk}.
\]
Writing $e$ for the identity of $K$, let $t\in\operatorname{Sym}(\widetilde Y)$ satisfy for all $y\in Y$ and $k\in K$,
\begin{equation}
\label{Eq:ActionOft}
    t(y)=a_{y,e},\qquad
    t(a_{y,k})=b_{k^{-1}y,k^{-1}},\qquad
    t(b_{y,e})=y.
\end{equation}
We then have for all $y\in Y$ and $h\in K$,
\[
    y\xrightarrow{ht}a_{y,h}
     \xrightarrow{ht}b_{h^{-1}y,e}
     \xrightarrow{ht}y,
\]
which proves the claim for $n=3$.

For $n>3$, add distinct new points $c_{y,1},\ldots,c_{y,n-3}$ for
each $y\in Y$ to $\widetilde Y$, and extend the action of $K$ by fixing
all these points. In \Cref{Eq:ActionOft}, we replace $t(y)=a_{y,e}$ by the conditions
\[
    t(y)=c_{y,1},\qquad
    t(c_{y,j})=c_{y,j+1}\quad(1\leq j<n-3),\qquad
    t(c_{y,n-3})=a_{y,e}.
\]
Then we are done, as 
\begin{equation*}
y\xrightarrow{ht}c_{y,1}\xrightarrow{ht}\cdots\xrightarrow{ht}c_{y,n-3}
    \xrightarrow{ht}a_{y,h}\xrightarrow{ht} b_{h^{-1}y,e}\xrightarrow{ht} y.\qedhere
\end{equation*}    
\end{proof}

\begin{proof}[Proof of \Cref{HigherPowerAveragesCounterexample}]
Fix $n\geq3$. Starting with an $n$-point set $Y_1$, iterate the construction in
\Cref{FinitePermutationPowerExtension} for this fixed $n$. We obtain increasing finite
sets $Y_N$, groups $G_N=\operatorname{Sym}(Y_N)$ with compatible
embeddings $G_N\hookrightarrow G_{N+1}$, and elements $t_N\in G_{N+1}$
such that
\[
    (ht_N)^n\big|_{Y_N}=\operatorname{id}_{Y_N}
    \quad\textup{ for all }h\in G_N.
\]
Identify each $G_N$ with its image in $G_{N+1}$, and put
\[
    G=\bigcup_{N\geq1}G_N,
    \qquad \Omega=\bigcup_{N\geq1}Y_N.
\]
Then $G$ is a union of finite subgroups, hence it is amenable, and the
compatible actions give an action of $G$ on $\Omega$. Consider the
permutation representation on $\mathcal{H}=\ell^2(\Omega)$ given by
$U_g\delta_y=\delta_{gy}$. Fix $o\in Y_1$ and take $\xi=\delta_o$.
Thus $G=A_0\sqcup A_1$, where
\[
    A_0=\{g\in G:g^no=o\},
    \qquad A_1=\{g\in G:g^no\neq o\}.
\]
Set
\[
    F_N^0=G_Nt_NG_N,
    \qquad F_N^1=G_N.
\]
Both sequences are two-sided F{\o}lner: for any fixed $a\in G$,
we have $a\in G_N$ for all sufficiently large $N$, and then
$aF_N^i=F_N^i=F_N^ia$ for $i=0,1$.
If $g=ht_Nk\in F_N^0$, with $h,k\in G_N$, then
\[
    g^n=k^{-1}\big((kh)t_N\big)^nk
\]
fixes $Y_N$ pointwise. Hence $F_N^0\subseteq A_0$ and $d_{F^0}(A_0)=1$.

Finally, we check $d_{F^1}(A_1)=1$. Let $d_N=|Y_N|$, which tends to infinity by construction.
For a uniformly chosen permutation $\sigma$ in $G_N=\operatorname{Sym}(Y_N)$,
the length of the cycle containing $o$ is uniform on
$\{1,\ldots,d_N\}$: for each length $j$, the number of such permutations is
\[
    \binom{d_N-1}{j-1}(j-1)!(d_N-j)!=(d_N-1)!.
\]
But $\sigma^n$ fixes $o$ exactly when this length divides $n$. Since $d_N\geq n$, all positive divisors of $n$ are possible cycle lengths. Therefore,
\begin{equation*}
    \frac{|A_1\cap F_N^1|}{|F_N^1|}
    =1-\frac{|\{j\in\mathbb{N}:j\mid n\}|}{d_N}
    \xrightarrow{N\to\infty}1.\qedhere
\end{equation*}
\end{proof}

\begin{remark}
\label{HigherPowerMeasureCorrelationsCounterexample}
The construction above also shows that \Cref{ConvSquareAvgsMeasure}
fails if, for some $n\geq3$, we replace $T_g^2$ by $T_g^n$. For the action $G\curvearrowright\Omega$ constructed in the proof of \Cref{HigherPowerAveragesCounterexample},
consider $X=\{0,1\}^{\Omega}$ with the product probability measure $\mu$
giving each coordinate the values $0$ and $1$ with equal probability.
The Bernoulli action $(T_gx)(y)=x(g^{-1}y)$ preserves $\mu$.
For $Y=\{x\in X:x(o)=1\}$, independence of distinct coordinates gives
\[
    \mu(Y\cap T_g^nY)=
    \begin{cases}
        1/2,& g^no=o,\\
        1/4,& g^no\neq o.
    \end{cases}
\]
Thus, we have
\begin{equation*}
\lim_N\frac{1}{|F^0_N|}\sum_{g\in F^0_N}\mu(Y\cap T_g^nY)=\frac{1}{2}\neq\frac{1}{4}=\lim_N\frac{1}{|F^1_N|}\sum_{g\in F^1_N}\mu(Y\cap T_g^nY)
\end{equation*}
Interleaving $F^0$ and $F^1$ gives a two-sided
F{\o}lner sequence along which these averages do not converge. 
\end{remark}